\documentclass[12pt]{article}
\usepackage{amssymb,amsmath,amsthm,epsfig,a4,verbatim,pstricks,url}
\usepackage[ruled]{algorithm2e}
\usepackage{ifpdf}  
\newtheorem{thm}{\sc Theorem.}[section]

\newtheorem{rem}{\sc Remark.}[section]

\newenvironment{AMS}%
{{\upshape\bfseries AMS subject classifications. }\ignorespaces}{}
\newenvironment{keywords}{{\upshape\bfseries Key words. }\ignorespaces}{}

\newcommand{\R}{{\mathbb R}}

\newcommand{\spa}{\operatorname{span}}

\newcommand{\diag}{\operatorname{diag}}
\newcommand{\diam}{\operatorname{diam}}
\newcommand{\vol}{\mathcal{L}^d} 
\newcommand{\conv}{\operatorname{conv}}

\newcommand{\Mloss}{\mathcal{L}_{\rm loss}}

\newcommand{\dH}[1]{\;{\rm d}{\cal H}^{#1}} 
\newcommand{\dL}[1]{\;{\rm d}{\cal L}^{#1}} 

\newcommand{\bigchi}{\ensuremath{\mathrm{\mathcal{X}}}}
\newcommand{\charfcn}[1]{\bigchi_{#1}} 
\newcommand{\Domain}{\Omega}

\newcommand{\Vh}{\underline{V}(\Gamma^m)}

\newcommand{\Wh}{W(\Gamma^m)}
\newcommand{\Vht}{\underline{V}(\Gamma^h(t))}
\newcommand{\Wht}{W(\Gamma^h(t))}
\newcommand{\uspace}{\mathbb{U}}
\newcommand{\pspace}{\mathbb{P}}
\newcommand{\kspace}{\mathbb{K}}
\renewcommand{\xspace}{\mathbb{X}}
\newcommand{\sigmaO}{o}
\newcommand{\nabs}{\nabla_{\!s}}

\newcommand{\Id}{I\!d}
\newcommand{\id}{\rm id}
\newcommand{\ddt}{\frac{\rm d}{{\rm d}t}}

\newcommand{\NbulkT}{\vec{N}_{\Gamma,\Omega}^T}
\newcommand{\Nbulk}{\vec{N}_{\Gamma,\Omega}}

\newcommand{\unitn}{\vec{\rm n}}
\newcommand{\unitt}{\vec{\rm t}}
\newcommand{\ek}{e}
\newcommand{\strikec}{\mbox{$c\!\!\!\!\:/$}}
\newcommand{\strikes}{\mbox{$s\!\!\!\!\:/$}}
\newcommand{\XFEMGAMMA}{XFEM$_\Gamma$}
\def\epsilon{\varepsilon} 
\newcommand{\mat}[1]{\underline{\underline{#1}}\rule{0pt}{0pt}}

\def\vL{L\kern-0.08cm\char39}
\begin{document}
\title{
A Stable Parametric Finite Element Discretization \\ of Two-Phase 
Navier--Stokes Flow}
\author{John W. Barrett\footnotemark[2] \and 
        Harald Garcke\footnotemark[3]\ \and 
        Robert N\"urnberg\footnotemark[2]}

\renewcommand{\thefootnote}{\fnsymbol{footnote}}
\footnotetext[2]{Department of Mathematics, 
Imperial College London, London, SW7 2AZ, UK}
\footnotetext[3]{Fakult{\"a}t f{\"u}r Mathematik, Universit{\"a}t Regensburg, 
93040 Regensburg, Germany}

\date{}

\maketitle

\begin{abstract}
We present a parametric finite element approximation of two-phase flow. This
free boundary problem is given by the Navier--Stokes equations in the two
phases, which are coupled via jump conditions across the interface. Using
a novel variational formulation for the interface evolution gives rise to a
natural discretization of the mean curvature of the interface. 
The parametric finite element approximation of the evolving interface is then
coupled to a standard finite element approximation of the two-phase
Navier--Stokes equations in the bulk. Here enriching the pressure approximation 
space with the help of an XFEM function ensures good volume conservation
properties for the two phase regions. In addition, 
the mesh quality of the parametric approximation of the interface in general
does not deteriorate over time, and an equidistribution property can be shown
for a semidiscrete continuous-in-time variant of our scheme in two space
dimensions.
Moreover, our
finite element approximation can be shown to be unconditionally stable. 
We demonstrate the applicability of our method with some numerical
results in two and three space dimensions.
\end{abstract} 


\renewcommand{\thefootnote}{\arabic{footnote}}

\section{Introduction}

Numerical methods for two-phase incompressible flows have many
important applications, which range from bubble column reactors to
ink-jet printing to fuel injection in engines and to biomedical
engineering. 
In contrast to one-phase flows, several new
aspects arise in the numerical treatment of two-phase flows. {\it First} of
all a computational technique for the {\it numerical treatment of the
  unknown interface} has to be developed. One class of approaches is
based on interface capturing methods using an indicator function to
describe the interface. The volume of fluid (VOF) method and the level set method fall into this category.
In the former, the characteristic function of one of the phases is
approximated numerically, see e.g.\ \cite{HirtN81,RenardyR02,Popinet09};
whereas in the latter, the interface is given as the level set of a function, which has to be
determined, see e.g.\ \cite{SussmanSO94,Sethian,OsherF03,GrossR07}. 
In phase field methods the interface is assumed
to have a small, but positive, thickness and an additional parabolic
equation, defined in the whole domain, has to be solved in these
so-called diffuse interface models. We refer to
\cite{HohenbergH77,AndersonMW98,LowengrubT98,Feng06,KaySW08,AbelsGG12,%
GrunK12preprint} for details.
In this paper we use a direct description of the interface using a
parameterization of the unknown surface. 
In such an approach the interface is approximated by
a polyhedral surface, see \cite{DeckelnickDE05}, and equations on the
surface mesh have to be coupled to quantities defined on the bulk
mesh. We refer e.g.\ to \cite{UnverdiT92,Bansch01,Tryggvason_etal01,GanesanMT07} for
further details, and to \cite{LevequeL97,Peskin02} for the
related immersed boundary method. 

{\it Secondly} it is important to {\it numerically approximate
 capillarity effects} in an accurate and stable way. In two-phase
flows, or in free surface flows, capillarity effects, which are given by
quantities involving the curvature of the interface, often determine
the flow behaviour to a large extent. 
Typically an explicit treatment of surface tension forces (also called 
capillary forces) leads 
to severe restrictions on the time step, see e.g.\
\cite{Bansch01,Bansch01habil}, 
and so more advanced approaches use an implicit treatment.
This approach is discussed e.g.\ in
\cite{GrossR07} for the level set approach, and in 
\cite{Bansch01} for the parametric approach. 
We note that an inadequate approximation
of capillarity effects can trigger oscillations of
the velocity at the interface, which can 
lead to so-called spurious currents,
see e.g.\ \cite{HysingTKPBGT09,GrossR11,spurious}.
In this paper we propose an implicit treatment of the surface tension forces
that leads to an unconditionally stable approximation of two-phase
Navier--Stokes flow.

In each of the approaches mentioned above (parametric approach, level
set method, volume of fluid method, phase field method) surface
tension forces, and hence curvature quantities, 
have to be computed. A particular successful method, in
the context of the parametric approach and the level set method, is
to compute the mean curvature of the approximated interface with the
help of a discretization of the identity 
\begin{equation}\label{eq:LBop}
\Delta_s\,\vec x=\varkappa\,\vec\nu\,.
\end{equation}
Here $\vec x$ is a parameterization of the interface, $\Delta_s$ is
the Laplace--Beltrami operator, $\varkappa$ is the 
sum of the principal curvatures (often simply called the mean curvature) 
and $\vec\nu$ is a unit normal to the interface. 
This identity was used for the  
numerical approximation of curvature driven interface evolution
for the first time by Dziuk, \cite{Dziuk91}. 
Later this idea was used in e.g.\ \cite{Bansch01,GanesanMT07,GrossR11}, among
others, in the context of capillarity driven free surface and two-phase flows.
The approximation of curvature in the present paper
also relies on the identity (\ref{eq:LBop}).

A {\it third} important issue relevant for the simulation of two-phase
flows is to ensure {\it a good approximation of the
interface} and in particular {\it a good mesh quality} during the
evolution. In phase field methods refinement of the mesh
close to the interface and choosing the interface width sufficiently
small ensures good approximation properties of the interface. However, this
leads to high computational costs. In volume of fluid methods the
interface has to be reconstructed after an advection step of the
characteristic function. Although, second order reconstruction methods
exist, see e.g.\ \cite{ScardovelliZ03,PilliodP04}, it still
remains challenging to approximate geometric quantities, such as the 
mean curvature and the normal of the interface, 
accurately. 

In level set methods the level set function is advected with the fluid
velocity. This typically leads to distortions of the level set
function, which in turn leads to a poor approximation of the interface. 
Hence so-called re-initialization steps have to be performed
frequently after some time steps, see e.g.\ 
\cite{Sethian,OsherF03,GrossR11} for details. 

In the parametric approach the interface parameterization is
transported with the help of the fluid velocity, see
e.g.\ \cite{UnverdiT92,Bansch01,GanesanMT07}. Typically this
leads to degeneracies in the mesh, e.g.\ coalescence of mesh
points and very small angles in the polyhedral interface mesh. 
Often severe reparameterization steps have to be
employed, or the computation even has to be stopped. 
In our approach the interface is advected in the normal direction with the
normal part of the fluid velocity, but {\it the tangential degrees
of freedom are implicitly used to ensure a good mesh quality}. 
This treatment of the
interface is based on earlier work of the authors on the numerical
approximation of geometric evolution equations and on free boundary
problems related to crystal growth, see e.g.\ 
\cite{triplej,gflows3d,dendritic,crystal}. 

A {\it fourth} issue is the {\it approximation of the pressure}, which is discontinuous
across the interface due to capillarity effects. There are three approaches to handle this
in the parametric or level set approach for the interface, combined with a finite element approximation
of the fluid quantities.
One is to use a fitted bulk mesh that is adapted to the interface, see e.g.\ 
\cite{GanesanMT07}.
In the case of an unfitted bulk mesh, where the interface and bulk meshes are totally independent,
one can augment the pressure finite element space with additional degrees
of freedom in elements of the bulk mesh, which cut the interface.
This is an example of the extended finite element method (XFEM),
see e.g.\ \cite{GrossR07,AusasBI12}.
A simpler approach is just to adapt the bulk mesh in the vicinity of the interface,
which we adopt in this paper.
However, the XFEM or the fitted approaches could be used in 
our approximation, see Subsection \ref{sec:33} for a discussion of the latter. 
 
Finally, a {\it fifth} issue is the {\it volume (area in 2d) conservation of the two phases}.
We achieve this by a very simple XFEM approach, where the pressure space is enriched
by one extra degree of freedom. This leads to exact volume conservation for a semidiscrete
continuous-in-time version of our scheme. Moreover, the fully discrete scheme 
shows good volume conservation properties in practice.

To summarize, in this paper we 
extend our parametric approximation of two-phase Stokes flow in \cite{spurious} to
two-phase Navier--Stokes flow with different densities. 
We present a {\it linear scheme},
i.e.\ a linear system of equations has to be solved at each time level, 
for 
this problem, which leads to {\it an unconditional stability bound}. 
Although, there already exists such a
stability bound for a nonlinear scheme for free capillary flows, see
\cite{Bansch01}; to
our knowledge the stability proof in this paper is the first one for a
linear scheme, and the first one for two-phase flows.

Finally, let us mention that developing and analyzing numerical
methods for two-phase flows is a very active research field, and we
refer to \cite{SussmanO09,AlandV12,ChengF12,ChoCCY12,ZahediKK12,%
JemsionLSSMMW13,LiYLSJK13} for some examples of recent contributions.

The outline of the paper is as follows. In Section~\ref{sec:Two} we give
a precise mathematical formulation of the two-phase Navier--Stokes
problem. In Section~\ref{sec:2} we introduce a weak formulation for
the resulting problem that will form the basis of our novel finite element
approximation, which we consider in Section~\ref{sec:3}. There we show that our
approximation is unconditionally stable, and introduce our XFEM approach for
volume conservation.
In Section~\ref{sec:4} we discuss 
how the arising discrete system of linear equations at each time level can 
be solved in practice. Finally,
in Section~\ref{sec:5} we discuss our mesh adaptation algorithms, and
then present several numerical experiments in Section~\ref{sec:6}.

\section{The two-phase Navier--Stokes problem}\label{sec:Two}

In this paper we consider two-phase flows in a given domain
$\Omega\subset\mathbb{R}^d$, where $d=2$ or $d=3$. The domain $\Omega$
contains two different immiscible incompressible phases (liquid-liquid
or liquid-gas) which for all $t\in[0,T]$ occupy time dependent regions
$\Omega_+(t)$ and $\Omega_-(t):=\Omega\setminus\overline{\Omega}_+(t)$
and which are separated by an interface
$(\Gamma(t))_{t\in[0,\overline{T}]}$, $\Gamma(t)\subset\Omega$. 
See Figure~\ref{fig:sketch} for an illustration.
\begin{figure}
\begin{center}
\unitlength15mm
\psset{unit=\unitlength,linewidth=1pt}
\begin{picture}(4,3.5)(0,0)
\psline[linestyle=solid]{-}(0,0)(4,0)(4,3.5)(0,3.5)(0,0)
\psccurve[showpoints=false,linestyle=solid] 
 (1,1.3)(1.5,1.6)(2.7,1.0)(2.5,2.6)(1.6,2.5)(1.1,2.7)
\psline[linestyle=solid]{->}(2.92,1.6)(3.4,1.6)
\put(3.25,1.7){{\black $\vec\nu$}}
\put(2.9,1.0){{$\Gamma(t)$}}
\put(2,2.1){{$\Omega_-(t)$}}
\put(0.5,0.5){{$\Omega_+(t)$}}
\end{picture}
\end{center}
\caption{The domain $\Omega$ in the case $d=2$.}
\label{fig:sketch}
\end{figure}
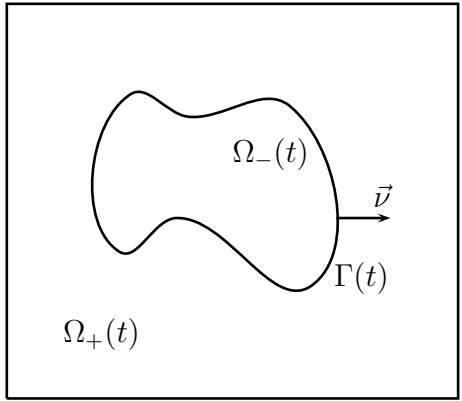%
For later use, we assume that $(\Gamma(t))_{t\in [0,T]}$ 
is a sufficiently smooth evolving hypersurface without boundary that is
parameterized by $\vec{x}(\cdot,t):\Upsilon\to\R^d$,
where $\Upsilon\subset \R^d$ is a given reference manifold, i.e.\
$\Gamma(t) = \vec{x}(\Upsilon,t)$. Then
$\mathcal{V} := \vec{x}_t \,.\,\vec{\nu}$ is
the normal velocity of the evolving hypersurface $\Gamma$,
where $\vec\nu$ is the unit normal on $\Gamma(t)$ pointing into $\Omega_+(t)$.

Let $\rho(t) = \rho_+\,\charfcn{\Omega_+(t)} + \rho_-\,\charfcn{\Omega_-(t)}$,
with $\rho_\pm \in \R_{\geq0}$, denote the fluid densities, where here and
throughout $\charfcn{\mathcal{A}}$ defines the characteristic function for a
set $\mathcal{A}$.
Denoting by $\vec u : \Omega \times [0, T] \to \R^d$ the fluid velocity,
by $\mat\sigma : \Omega \times [0,T] \to \R^{d \times d}$ the stress tensor,
and by $\vec f : \Omega \times [0, T] \to \R^d$ a possible forcing,
the incompressible Navier--Stokes equations in the two phases are given by
\begin{subequations}
\begin{alignat}{2}
\rho\,(\vec u_t + (\vec u \,.\,\nabla)\,\vec u)
- \nabla\,.\,\mat\sigma & = \vec f := \rho\,\vec f_1 + \vec f_2
\qquad &&\mbox{in } 
\Omega_\pm(t)\,, \label{eq:NSa} \\
\nabla\,.\,\vec u & = 0 \qquad &&\mbox{in } \Omega_\pm(t)\,, \label{eq:NSb} \\
\vec u & = \vec 0 \qquad &&\mbox{on } \partial_1\Omega\,, \label{eq:NSc} \\
\vec u \,.\,\unitn = 0\,,\quad
[\mat\sigma\,\unitn + \beta\,\vec u]\,.\,\unitt& = 0 
\quad\forall\ \unitt \in \{\unitn\}^\perp 
\qquad &&\mbox{on } \partial_2\Omega\,, 
\label{eq:NSd} 
\end{alignat}
\end{subequations}
where $\partial\Domain =\partial_1\Omega \cup
\partial_2\Omega$, with $\partial_1\Omega \cap \partial_2\Omega
=\emptyset$, denotes the boundary of $\Domain$ with outer unit normal $\unitn$
and $\{\unitn\}^\perp := \{ \unitt \in \R^d : \unitt \,.\,\unitn = 0\}$.
Hence (\ref{eq:NSc}) prescribes a no-slip condition on 
$\partial_1\Omega$, while (\ref{eq:NSd}) prescribes a general slip condition on 
$\partial_2\Omega$. Here we assume that $\beta \geq 0$ and note that 
$\beta = 0$ corresponds to the so-called free-slip conditions.
Moreover, the stress tensor in (\ref{eq:NSa}) is defined by
\begin{equation} \label{eq:sigma}
\mat\sigma = \mu \,(\nabla\,\vec u + (\nabla\,\vec u)^T) - p\,\mat\Id\,,
\end{equation}
where $\mat\Id \in \R^{d \times d}$ denotes the identity matrix,
$p : \Omega \times [0, T] \to \R$ is the pressure and
$\mu(t) = \mu_+\,\charfcn{\Omega_+(t)} + \mu_-\,\charfcn{\Omega_-(t)}$,
with $\mu_\pm \in \R_{>0}$, denotes the dynamic viscosities in the two 
phases.
On the free surface $\Gamma(t)$, the following conditions need to hold:
\begin{subequations}
\begin{alignat}{2}
[\vec u]_-^+ & = \vec 0 \qquad &&\mbox{on } \Gamma(t)\,, \label{eq:1a} \\ 
[\mat\sigma\,\vec \nu]_-^+ & = -\gamma\,\varkappa\,\vec\nu 
\qquad &&\mbox{on } \Gamma(t)\,, \label{eq:1b} \\ 
\mathcal{V} &= \vec u\,.\,\vec \nu \qquad &&\mbox{on } \Gamma(t)\,, 
\label{eq:1c} 
\end{alignat}
where $\gamma>0$ is the surface tension coefficient
and $\varkappa$ denotes the mean curvature of $\Gamma(t)$, i.e.\ the sum of
the principal curvatures of $\Gamma(t)$, where we have adopted the sign
convention that $\varkappa$ is negative where $\Omega_-(t)$ is locally convex.
Moreover, as usual, $[\vec u]_-^+ := \vec u_+ - \vec u_-$ and
$[\mat\sigma\,\vec\nu]_-^+ := \mat\sigma_+\,\vec\nu - \mat\sigma_-\,\vec\nu$
denote the jumps in velocity and normal stress across the interface
$\Gamma(t)$. Here and throughout, we employ the shorthand notation
$\vec g_\pm := \vec g\!\mid_{\Omega_\pm(t)}$ for a function 
$\vec g : \Omega \times [0,T] \to \R^d$; and similarly for scalar and
matrix-valued functions.
The system (\ref{eq:NSa}--d), (\ref{eq:sigma}), (\ref{eq:1a}--c) 
is closed with the initial conditions
\begin{equation} \label{eq:1d}
\Gamma(0) = \Gamma_0 \,, \qquad 
\rho(\cdot,0)\,\vec u(\cdot,0) = \rho(\cdot,0)\,\vec u_0 
\qquad \mbox{in } \Omega\,,
\end{equation}
\end{subequations}
where $\Gamma_0 \subset \Omega$ and
$\vec u_0 : \Omega \to \R^d$ are given initial data. 

\setcounter{equation}{0}
\section{Weak formulation} \label{sec:2}

In preparation for the introduction of the weak formulation considered in this
paper, we note that the system (\ref{eq:NSa}--d), (\ref{eq:sigma}), 
(\ref{eq:1a}--d) can equivalently be formulated as follows. Find
time and space dependent functions
$\vec u$, $p$ and the interface $(\Gamma(t))_{t\in[0,T]}$ such that
\begin{subequations}
\begin{alignat}{2}
\rho\,(\vec u_t + (\vec u \,.\,\nabla)\,\vec u)
- \mu\,\nabla\,.\,(\nabla\,\vec u + (\nabla\,\vec u)^T)
+ \nabla\,p & = \vec f \qquad &&\mbox{in } 
\Omega_\pm(t)\,, \label{eq:2a} \\
\nabla\,.\,\vec u & = 0 \qquad &&\mbox{in } \Omega_\pm(t)\,, \label{eq:2b} \\
\vec u & = \vec 0 \qquad &&\mbox{on } \partial_1\Omega\,, \label{eq:2c} \\
\vec u \,.\,\unitn = 0\,,\quad
[\mu\,(\nabla\,\vec u + (\nabla\,\vec u)^T)\,\unitn + \beta\,\vec u]\,.\,\unitt& = 0 
\quad\forall\ \unitt \in \{ \unitn \}^\perp
\qquad &&\mbox{on } \partial_2\Omega\,, \label{eq:2cc} \\
[\vec u]_-^+ & = \vec 0 \qquad &&\mbox{on } \Gamma(t)\,, \label{eq:2d} \\ 
[\mu \,(\nabla\,\vec u + (\nabla\,\vec u)^T)\,\vec\nu - p\,\vec \nu]_-^+ 
& = -\gamma\,\varkappa\,\vec\nu \qquad &&\mbox{on } \Gamma(t)\,, \label{eq:2e} \\ 
\mathcal{V} &= \vec u\,.\,\vec \nu \qquad &&\mbox{on } \Gamma(t)\,, 
\label{eq:2f}  \\
\Gamma(0) = \Gamma_0 \,, \quad 
\rho(\cdot,0)\,\vec u(\cdot,0) & = \rho(\cdot,0)\,\vec u_0 
\qquad && \mbox{in } \Omega\,.
\label{eq:2g} 
\end{alignat}
\end{subequations}
Moreover, we observe that for arbitrary functions $\vec v$, 
$\vec w$, $\vec \xi \in H^1(\Omega,\R^d)$ it holds that
\begin{align} \label{eq:tripleterm}
[(\vec v\,.\,\nabla)\,\vec w]\,.\,\vec \xi
&= (\vec v\,.\,\nabla)\,(\vec w\,.\,\vec \xi) - 
[(\vec v\,.\,\nabla)\,\vec \xi]\,.\,\vec w \nonumber \\ &
= \tfrac12 \left[\,
[(\vec v\,.\,\nabla)\,\vec w]\,.\,\vec \xi - [(\vec v\,.\,\nabla)\,\vec \xi]
\,.\,\vec w\,\right] + \tfrac12\,(\vec v\,.\,\nabla)\,(\vec w\,.\,\vec \xi)\,.
\end{align}
For what follows we introduce the function spaces
\begin{align}
\uspace & := \{ \vec\phi \in H^1(\Omega, \R^d) : \vec\phi = \vec0 \ \mbox{ on } 
\partial_1\Omega\,,\ \vec\phi \,.\,\unitn = 0 \ \mbox{a.e. on } \partial_2\Omega
 \} \,,\
\pspace := L^2(\Omega) \nonumber \\
\quad\mbox{and}\quad
\widehat\pspace & := \{\eta \in \pspace : \int_\Omega\eta \dL{d}=0 \}\,. 
\label{eq:UP}
\end{align}
Here and throughout $\mathcal{L}^d$ denotes the Lebesgue measure in $\R^d$,
while $\mathcal{H}^{d-1}$ denotes the $(d-1)$-dimensional Hausdorff measure.
Then we have for any $\vec v \in \uspace$ 
that 
\begin{align} \label{eq:ibp0}
( \rho,(\vec v \,.\,\nabla)\,\phi) & = (\rho, \nabla\,.\,(\phi\,\vec v))
- (\rho\,(\nabla\,.\,\vec v), \phi) \nonumber \\ & =
- \left\langle [\rho]_-^+\,\vec v\,.\,\vec \nu, 
   \phi \right\rangle_{\Gamma(t)} 
- (\rho\,(\nabla\,.\,\vec v), \phi)
\end{align}
which holds for all $\phi\in W^{1,\frac{3}{2}}(\Omega)$, and it is
precisely this regularity that will be needed in order to derive
(\ref{eq:advect}) below. 
Here and throughout $(\cdot,\cdot)$ denotes the $L^2$--inner product on
$\Omega$, while $\langle \cdot,\cdot\rangle_{\Gamma(t)}$ is the 
$L^2$--inner product on $\Gamma(t)$.
Hence, it follows from (\ref{eq:2b},c), (\ref{eq:tripleterm}) and (\ref{eq:ibp0})
that
\begin{align}
( \rho\,(\vec u \,.\,\nabla)\,\vec u, \vec \xi)
&= \tfrac12 \left[ (\rho\,(\vec u\,.\,\nabla)\,\vec u, \vec \xi) -
(\rho\,(\vec u\,.\,\nabla)\,\vec \xi,\vec u)
-\langle [\rho]_-^+\,\vec u\,.\,\vec \nu, 
  \vec u\,.\,\vec \xi \rangle_{\Gamma(t)} \right]
\nonumber \\ 
& \hspace{3.2in} \qquad \forall\ \vec \xi \in H^1(\Omega,{\mathbb R}^d)\,.
\label{eq:advect}
\end{align}
Next, on noting (\ref{eq:2f}), we have that
\begin{align}
\ddt(\rho \,\vec u, \vec \xi) & = 
\ddt\left(\rho_+\,\int_{\Omega_+(t)} \vec u \,.\,\vec \xi \dL{d}
+ \rho_-\,\int_{\Omega_-(t)} \vec u\,.\,\vec \xi \dL{d}  \right) \nonumber \\ & 
=  (\rho\,\vec u_t, \vec \xi)
-\left\langle [\rho]_-^+\,\mathcal{V}, \vec u \,.\,\vec \xi 
\right\rangle_{\Gamma(t)}
\nonumber \\
&= (\rho\,\vec u_t, \vec \xi)- \left\langle [\rho]_-^+\,\vec u\,.\,\vec \nu, 
\vec u \,.\,\vec \xi \right\rangle_{\Gamma(t)}
\qquad \forall \ 
\vec \xi \in H^1(\Omega, {\mathbb R}^d)
\,.
\label{eq:rhot1}
\end{align}
Therefore, it follows from (\ref{eq:rhot1}) that
\begin{equation*} 
(\rho\,\vec u_t, \vec \xi) = 
\tfrac{1}{2} \left[
\ddt (\rho\,\vec u,\vec \xi) + (\rho\,\vec u_t, \vec \xi)
+ \left\langle [\rho]_-^+\,\vec u\,.\,\vec \nu, 
\vec u \,.\,\vec \xi \right\rangle_{\Gamma(t)} 
\right]
\qquad \forall\ \vec \xi \in H^1(\Omega, {\mathbb R}^d)\,, 
\end{equation*}
which on combining with (\ref{eq:advect}) yields that
\begin{equation} \label{eq:rhot3}
(\rho\,[\vec u_t + (\vec u \,.\,\nabla)\,\vec u], \vec \xi)
= \tfrac12\left[ \ddt (\rho\,\vec u, \vec \xi) + (\rho\,\vec u_t, \vec \xi)
+ (\rho, [(\vec u\,.\,\nabla)\,\vec u]\,.\,\vec \xi
- [(\vec u\,.\,\nabla)\,\vec \xi]\,.\,\vec u) \right].
\end{equation}
Finally, it holds on noting (\ref{eq:2cc},f) 
that for all $\vec \xi \in \uspace$ 
\begin{align}
& \int_{\Omega_+(t)\cup\Omega_-(t)} (\nabla\,.\,\mat\sigma)\,.\, \vec \xi \dL{d} 
 = -(\mat\sigma, \nabla\,\vec \xi)
- \left\langle [\mat\sigma\,\vec\nu]_-^+, \vec \xi \right\rangle_{\Gamma(t)} 
+ \int_{\partial\Omega} (\mat\sigma\,\unitn)\,.\,\vec \xi \dH{d-1}
\nonumber \\ & \hspace{3cm}
= - 2\,(\mu\,\mat D(\vec u), \mat D(\vec \xi)) + (p, \nabla\,.\,\vec \xi)
+ \gamma\left\langle \varkappa\,\vec \nu, \vec \xi \right\rangle_{\Gamma(t)}
- \beta\left\langle \vec u, \vec \xi \right\rangle_{\partial_2\Omega, \unitt}\,,
\label{eq:sigmaibp}
\end{align}
where $\mat D(\vec u) := \tfrac12\,( \nabla\,\vec u + (\nabla\,\vec u)^T)$,
and where 
\begin{equation*} 
\left\langle \vec u, \vec \xi \right\rangle_{\partial_2\Omega, \unitt} :=
\sum_{i=1}^{d-1} \int_{\partial_2 \Omega} (\vec u \,.\, \unitt_i)
\,(\vec \xi \,.\,\unitt_i) \dH{d-1}\,,
\end{equation*}
with
$\{ \unitt_i \}_{i=1}^{d-1}$ denoting an orthonormal basis of $\{ \unitn
\}^\perp$.

In addition, we define
\begin{equation*}
\xspace := H^1(\Upsilon,\R^d) \quad\mbox{and}\quad
\kspace := L^2(\Upsilon,\R)\,, 
\end{equation*}
where we recall that $\Upsilon$ is a given reference manifold.
On combining (\ref{eq:rhot3}) and (\ref{eq:sigmaibp}),
a possible weak formulation of (\ref{eq:2a}--h),
which utilizes the novel weak representation of $\varkappa\,\vec\nu$
introduced in \cite{triplej} for $d=2$ and in \cite{gflows3d} for $d=3$, 
is then given as follows.
Find time dependent functions $\vec u$, $p$, $\vec{x}$ and $\varkappa$ such
that $\vec u(\cdot,t)\in \uspace$, $p(\cdot,t)\in \widehat\pspace$,
$\vec{x}(\cdot,t) \in \xspace$, $\varkappa(\cdot,t)\in \kspace$ and
\begin{subequations}
\begin{align}
\tfrac12\left[ \ddt (\rho\,\vec u, \vec \xi) + (\rho\,\vec u_t, \vec \xi)
+ (\rho, [(\vec u\,.\,\nabla)\,\vec u]\,.\,\vec \xi
- [(\vec u\,.\,\nabla)\,\vec \xi]\,.\,\vec u) \right] &
+ 2\,(\mu\,\mat D(\vec u), \mat D(\vec \xi)) 
\nonumber \\ 
- (p, \nabla\,.\,\vec \xi)
+ \beta\left\langle \vec u, \vec \xi \right\rangle_{\partial_2\Omega, \unitt}
- \gamma\left\langle \varkappa\,\vec \nu, \vec \xi \right\rangle_{\Gamma(t)}
& = (\vec f, \vec \xi) \quad \forall\ \vec \xi \in \uspace\,,
\label{eq:weaka} \\
 (\nabla\,.\,\vec u, \varphi) & = 0  \qquad \forall\ \varphi \in 
\widehat\pspace\,,
\label{eq:weakb} \\
  \left\langle \vec x_t - \vec u, \chi\,\vec\nu \right\rangle_{\Gamma(t)} & = 0
 \qquad\forall\ \chi \in \kspace\,,
\label{eq:weakc} \\
 \left\langle \varkappa\,\vec\nu, \vec\eta \right\rangle_{\Gamma(t)}
+ \left\langle \nabs\,\vec x, \nabs\,\vec \eta \right\rangle_{\Gamma(t)}
& = 0  \qquad\forall\ \vec\eta \in \xspace
\label{eq:weakd}
\end{align}
\end{subequations}
hold for almost all times $t\in (0,T]$, as well as the initial conditions
(\ref{eq:2g}). 
Here we have observed that if $p \in \pspace$ is part of a
solution to (\ref{eq:2a}--g), then so is $p + c$ for an arbitrary $c \in \R$.
We remark that (\ref{eq:weaka}--d) in the case of Stokes flow, i.e.\ 
$\rho_+ = \rho_- = 0$, 
with $\partial_1 \Omega = \partial\Omega$
collapses to the weak formulation introduced by the present 
authors in \cite{spurious}. Similarly to \cite{spurious},
we have adopted a slight abuse of notation in (\ref{eq:weaka},c,d), in the
sense that here, and throughout this paper, we identify functions defined
on the reference manifold $\Upsilon$ with functions defined on $\Gamma(t)$.
In addition, we observe that in the special case 
$\rho_+ = \rho_- > 0$, $\mu_+ = \mu_- > 0$ and $\gamma = 0$, 
{\rm (\ref{eq:weaka},b)} reduces to a weak formulation of the 
incompressible Navier--Stokes equations in $\Omega$.

We can establish the following formal a priori bound, where we first note that
on allowing time-dependent test functions $\vec \xi$ in (\ref{eq:rhot1}),
which yields the extra term $( \rho\,\vec u, \vec \xi_t)$ on the right hand 
side of (\ref{eq:rhot1}), 
we obtain the following amended version of (\ref{eq:weaka}) 
for time-dependent test functions $\vec\xi(\cdot,t) \in \uspace$:
\begin{alignat}{2}
\tfrac12\left[ \ddt (\rho\,\vec u, \vec \xi) + (\rho\,\vec u_t, \vec \xi)
- (\rho\,\vec u, \vec \xi_t)
+ (\rho, [(\vec u\,.\,\nabla)\,\vec u]\,.\,\vec \xi
- [(\vec u\,.\,\nabla)\,\vec \xi]\,.\,\vec u) \right]
&\nonumber \\ 
+ 2\,(\mu\,\mat D(\vec u), \mat D(\vec \xi)) - (p, \nabla\,.\,\vec \xi)
+ \beta\left\langle \vec u, \vec \xi \right\rangle_{\partial_2\Omega, \unitt}
- \gamma\left\langle \varkappa\,\vec \nu, \vec \xi \right\rangle_{\Gamma(t)}
& = (\vec f, \vec \xi)\,.
\label{eq:weakat} 
\end{alignat}
Now choosing $\vec\xi=\vec u$ in (\ref{eq:weakat}), 
$\varphi = p$ in (\ref{eq:weakb}),
$\chi=\gamma\,\varkappa$ in (\ref{eq:weakc}) and 
$\vec\eta=\gamma\,\vec{x}_t$ in (\ref{eq:weakd}) we obtain, 
on using the identity
\begin{equation*} 
\ddt\,\mathcal{H}^{d-1}(\Gamma(t)) 
= \left\langle \nabs\,\vec{x},\nabs\,\vec{x}_t 
\right\rangle_{\Gamma(t)} \,,
\end{equation*}
that
\begin{equation} \label{eq:testD}
 \ddt\left(\tfrac12\,\|\rho^\frac12\,\vec u\|^2_0 + 
\gamma\,\mathcal{H}^{d-1}(\Gamma(t)) \right) 
+ 2\,\|\mu^\frac12\,\mat D(\vec u)\|^2_0
+ \beta\left\langle \vec u, \vec u \right\rangle_{\partial_2\Omega, \unitt}
= (\vec f, \vec u) \,,
\end{equation}
where $\|\cdot\|_0 := [(\cdot,\cdot)]^\frac12$ 
denotes the $L^2$--norm on $\Omega$. Moreover, the volume
of $\Omega_-(t)$ is preserved in time. To see this, choose $\chi = 1$ in
(\ref{eq:weakc}) and 
$\varphi = (\charfcn{\Omega_-(t)} -
\frac{\mathcal{L}^d(\Omega_-(t))}{\mathcal{L}^d(\Omega)})$
in (\ref{eq:weakb}) to obtain
\begin{equation}
\frac{\rm d}{{\rm d}t} \vol(\Omega_-(t)) = 
\left\langle \vec{x}_t , \vec\nu \right\rangle_{\Gamma(t)}
= \left\langle \vec u, \vec\nu \right\rangle_{\Gamma(t)}
= \int_{\Omega_-(t)} \nabla\,.\,\vec u \dL{d} 
=0\,. \label{eq:conserved}
\end{equation}

\setcounter{equation}{0}
\section{Discretization} \label{sec:3}

Let $0= t_0 < t_1 < \ldots < t_{M-1} < t_M = T$ be a
partitioning of $[0,T]$ into possibly variable time steps 
$\tau_m := t_{m+1} -
t_{m}$, $m=0,\ldots, M-1$. We set $\tau := \max_{m=0,\ldots, M-1}\tau_m$.
First we introduce standard finite element spaces of piecewise polynomial
functions on $\Omega$.

Let $\Omega$ be a polyhedral domain. For $m\geq0$, let ${\cal T}^m$ 
be a regular partitioning of $\Omega$ into disjoint open simplices, so that 
$\overline{\Omega}=\cup_{\sigmaO^m\in{\cal T}^m}\overline{\sigmaO^m}$. 
Let $J_\Omega^m$ be the number of elements in $\mathcal{T}^m$, so that
$\mathcal{T}^m=\{ \sigmaO^m_j : j = 1 ,\ldots, J^m_\Omega\}$, and set
$h^m := \max_{j=1,\ldots, J^m_\Omega} \diam(\sigmaO^m_j)$.
Associated with ${\cal T}^m$ are the finite element spaces
\begin{equation*} 
 S^m_k := \{\chi \in C(\overline{\Omega}) : \chi\!\mid_{\sigmaO^m} \in
\mathcal{P}_k(\sigmaO^m) \quad \forall\ \sigmaO^m \in {\cal T}^m\} 
\subset H^1(\Omega)\,, \quad k \in \mathbb{N}\,,
\end{equation*}
where $\mathcal{P}_k(\sigmaO^m)$ denotes the space of polynomials of degree $k$
on $\sigmaO^m$. We also introduce the space of discontinuous, 
piecewise constant functions
\begin{equation*} 
 S^m_0 := \{\chi \in L^1(\Omega) : \chi\!\mid_{\sigmaO^m} \in
\mathcal{P}_0(\sigmaO^m) \quad \forall\ \sigmaO^m \in {\cal T}^m\} \,.
\end{equation*}
Let $\{\phi_{k,j}^m\}_{j=1}^{K_k^m}$ be the standard basis functions 
for $S^m_k$, $k\geq 0$.
We introduce $I^m_k:C(\overline{\Omega})\to S^m_k$, $k\geq 1$, 
the standard interpolation
operators, such that $(I^m_k \eta)(\vec{p}_{k,j}^m)= \eta(\vec{p}_{k,j}^m)$ 
for $j=1,\ldots, K_k^m$; where
$\{\vec p_{k,j}^m\}_{j=1}^{K_k^m}$ denote the coordinates of the degrees of
freedom of $S^m_k$, $k\geq 1$. In addition we define the standard projection
operator $I^m_0:L^1(\Omega)\to S^m_0$, such that
\begin{equation*}
(I^m_0 \eta)\!\mid_{o^m} = \frac1{\mathcal{L}^d(o^m)}\,\int_{o^m}
\eta \dL{d} \qquad \forall\ o^m \in \mathcal{T}^m\,.
\end{equation*}
Let $(\uspace^m,\pspace^m)$, with $\uspace^m\subset\uspace$,
be a pair of finite element spaces on
$\mathcal{T}^m$ that satisfy the LBB inf-sup condition. I.e.\ there exists a
constant $C_0 \in \R_{>0}$ independent of $h^m$ such that
\begin{equation} \label{eq:LBB}
\inf_{\varphi \in \widehat\pspace^m} \sup_{\vec \xi \in \uspace^m}
\frac{( \varphi, \nabla \,.\,\vec \xi)}
{\|\varphi\|_0\,\|\vec \xi\|_1} \geq C_0 > 0\,,
\end{equation}
where $\|\cdot\|_1 := \|\cdot\|_0 + \|\nabla\,\cdot\|_0$ denotes the 
$H^1$--norm on $\Omega$, and
where $\widehat\pspace^m:= \pspace^m \cap \widehat\pspace$, 
recall (\ref{eq:UP}); 
see e.g.\ \cite[p.~114]{GiraultR86}. 
For example, we may choose
\begin{subequations}
\begin{equation} \label{eq:P2P1}
(\uspace^m,\pspace^m) = ([S^m_2]^d\cap\uspace, S^m_1)
\end{equation}
for the lowest order Taylor--Hood element, also called the P2-P1 element, or
\begin{equation} \label{eq:P2P0}
(\uspace^m,\pspace^m) = ([S^m_2]^d\cap\uspace, S^m_0)
\end{equation}
for the P2-P0 element, or
\begin{equation} \label{eq:P2P10}
(\uspace^m,\pspace^m) = ([S^m_2]^d\cap\uspace, 
S^m_1 + S^m_0)
\end{equation}
\end{subequations}
for the P2-(P1+P0) element.
For the LBB stability of (\ref{eq:P2P1}) in the case
$\partial_1\Omega = \partial\Omega$ we refer to \cite[p.\ 252]{BrezziF91} for
$d=2$ and to \cite{Boffi97} for $d=3$, 
while the stability of (\ref{eq:P2P0}) is shown in 
\cite[p.\ 221]{BrezziF91}. The LBB stability of (\ref{eq:P2P10}) is shown in
\cite{BoffiCGG12} for $d=2$ and $d=3$.
Here the results for (\ref{eq:P2P1},c) need the weak
constraint that all the elements $o^m \in \mathcal{T}^m$ have a vertex in
$\Omega$. The LBB stability of (\ref{eq:P2P1}--c) for the
general case $\partial_2\Omega \not= \emptyset$ then follows trivially from
(\ref{eq:LBB}), since the space $\uspace$ is now less constrained.
Let $\{\{\phi_i^{\uspace^m}\,\vec \ek_j\}_{j=1}^d \}_{i=1}^{K_{\uspace}^m}$ and
$\{\phi_i^{\pspace^m}\}_{i=1}^{K_{\pspace}^m}$ denote the standard basis
functions of $\uspace^m$ and $\pspace^m$, respectively,
where $\{\vec \ek_j\}_{j=1}^d$ denotes the standard basis in $\R^d$.

The parametric finite element spaces in order to approximate $\vec{x}$ and
$\varkappa$ in (\ref{eq:weaka}--d), are defined as follows.
Similarly to \cite{gflows3d}, 
we introduce the following discrete spaces, based on the seminal
paper \cite{Dziuk91}.
Let $\Gamma^{m}\subset\R^d$ be a $(d-1)$-dimensional {\em polyhedral surface},
i.e.\ a union of nondegenerate $(d-1)$-simplices with no hanging vertices
(see \cite[p.~164]{DeckelnickDE05} for $d=3$), approximating the
closed surface $\Gamma(t_m)$, $m=0 ,\ldots, M$.
In particular, let $\Gamma^m=\bigcup_{j=1}^{J^m_\Gamma} 
\overline{\sigma^m_j}$,
where $\{\sigma^m_j\}_{j=1}^{J^m_\Gamma}$ is a family of mutually disjoint open 
$(d-1)$-simplices 
with vertices $\{\vec{q}^m_k\}_{k=1}^{K^m_\Gamma}$. 
Then for $m =0 ,\ldots, M-1$, let
\begin{equation*}
\Vh := \{\vec\chi \in C(\Gamma^m,\R^d):\vec\chi\!\mid_{\sigma^m_j}
\mbox{ is linear}\ \forall\ j=1,\ldots, J^m_\Gamma\} 
=: [\Wh]^d \subset H^1(\Gamma^m,\R^d)\,,
\end{equation*}
where $\Wh \subset H^1(\Gamma^m,\R)$ is the space of scalar continuous
piecewise linear functions on $\Gamma^m$, with 
$\{\chi^m_k\}_{k=1}^{K^m_\Gamma}$ denoting the standard basis of $\Wh$.
For later purposes, we also introduce 
$\pi^m: C(\Gamma^m,\R)\to \Wh$, the standard interpolation operator
at the nodes $\{\vec{q}_k^m\}_{k=1}^{K^m_\Gamma}$,
and similarly $\vec\pi^m: C(\Gamma^m,\R^d)\to \Vh$.
Throughout this paper, we will parameterize the new closed surface 
$\Gamma^{m+1}$ over $\Gamma^m$, with the help of a parameterization
$\vec{X}^{m+1} \in \Vh$, i.e.\ $\Gamma^{m+1} = \vec{X}^{m+1}(\Gamma^m)$.
Moreover, for $m \geq 0$, we will often identify $\vec{X}^m$ with 
$\vec{\rm id} \in \Vh$, the identity function on $\Gamma^m$. 

For scalar, vector and matrix valued functions 
we introduce the $L^2$--inner product 
$\langle\cdot,\cdot\rangle_{\Gamma^m}$ over
the current polyhedral surface $\Gamma^m$ 
as follows
\begin{equation*} 
\langle v, w\rangle_{\Gamma^m}  := \int_{\Gamma^m} v\,.\,w \; \dH{d-1}\,.
\end{equation*}
If $v,w$ are piecewise continuous, with possible jumps
across the edges of $\{\sigma_j^m\}_{j=1}^{J^m_\Gamma}$,
we introduce the mass lumped inner product
$\langle\cdot,\cdot\rangle_{\Gamma^m}^h$ as
\begin{equation*}
\langle v, w \rangle^h_{\Gamma^m}  :=
\tfrac1d \sum_{j=1}^{J^m_\Gamma} \mathcal{H}^{d-1}(\sigma^m_j)\,\sum_{k=1}^{d} 
(v\,.\,w)((\vec{q}^m_{j_k})^-),
\end{equation*}
where $\{\vec{q}^m_{j_k}\}_{k=1}^{d}$ are the vertices of $\sigma^m_j$,
and where
we define $v((\vec{q}^m_{j_k})^-):=
\underset{\sigma^m_j\ni \vec{p}\to \vec{q}^m_{j_k}}{\lim}\, v(\vec{p})$.

Given $\Gamma^m$, we 
let $\Omega^m_+$ denote the exterior of $\Gamma^m$ and let
$\Omega^m_-$ denote the interior of $\Gamma^m$, so that
$\Gamma^m = \partial \Omega^m_- = \overline\Omega^m_- \cap 
\overline\Omega^m_+$. 
We then partition the elements of the bulk mesh 
$\mathcal{T}^m$ into interior, exterior and interfacial elements as follows.
Let
\begin{align}
\mathcal{T}^m_- & := \{ o^m \in \mathcal{T}^m : o^m \subset
\Omega^m_- \} \,, \nonumber \\
\mathcal{T}^m_+ & := \{ o^m \in \mathcal{T}^m : o^m \subset
\Omega^m_+ \} \,, \nonumber \\
\mathcal{T}^m_{\Gamma^m} & := \{ o^m \in \mathcal{T}^m : o^m \cap
\Gamma^m \not = \emptyset \} \,. \label{eq:partT}
\end{align}
Clearly $\mathcal{T}^m = \mathcal{T}^m_- \cup \mathcal{T}^m_+ \cup
\mathcal{T}^m_{\Gamma^m}$ is a disjoint partition, which in practice
can easily be found e.g.\ with the Algorithm~4.1 in \cite{crystal}. Here we
assume that $\Gamma^m$ has no self intersections, and for the numerical
experiments in this paper this was always the case.
In addition, we define the piecewise constant unit normal 
$\vec{\nu}^m$ to $\Gamma^m$ by
\begin{equation*}
\vec{\nu}^m_j := \vec{\nu}^m\!\mid_{\sigma^m_j} :=
\frac{(\vec{q}^m_{j_2}-\vec{q}^m_{j_1}) \land \cdots \land
(\vec{q}^m_{j_d}-\vec{q}^m_{j_1})}{|(\vec{q}^m_{j_2}-\vec{q}^m_{j_1}) \land
\cdots \land (\vec{q}^m_{j_d}-\vec{q}^m_{j_1})|}\,,
\end{equation*}
where $\land$ is the standard wedge product on $\R^d$, and where
we have assumed that the vertices $\{\vec{q}^m_{j_k}\}_{k=1}^d$
of $\sigma_j^m$ are ordered such that $\vec\nu^m:\Gamma^m\to\R^d$ 
induces an orientation on $\Gamma^m$, and such that $\vec\nu^m$ points into
$\Omega^m_+$.

Before we can introduce our approximation to (\ref{eq:weaka}--d), we have to
introduce the notion of a vertex normal on $\Gamma^m$. We will combine this
definition with a natural assumption that is needed in order to show existence
and uniqueness, where applicable, for the introduced finite element 
approximation.
\begin{itemize}
\item[$({\cal A})$]
We assume for $m=0,\ldots, M-1$ that $\mathcal{H}^{d-1}(\sigma^m_j) > 0$ 
for all $j=1,\ldots, J^m_\Gamma$,
and that $\Gamma^m \subset \overline\Omega$.
For $k= 1 ,\ldots, K^m_\Gamma$, let
$\Xi_k^m:= \{\sigma^m_j : \vec{q}^m_k \in \overline{\sigma^m_j}\}$
and set
\begin{equation*}
\Lambda_k^m := \cup_{\sigma^m_j \in \Xi_k^m} \overline{\sigma^m_j}
 \qquad \mbox{and} \qquad
\vec\omega^m_k := \frac{1}{\mathcal{H}^{d-1}(\Lambda^m_k)}
\sum_{\sigma^m_j\in \Xi_k^m} \mathcal{H}^{d-1}(\sigma^m_j)
\;\vec{\nu}^m_j\,. 
\end{equation*}
Then we further assume that 
$\dim \spa\{\vec{\omega}^m_k\}_{k=1}^{K^m_\Gamma} = d$,
$m=0,\ldots, M-1$.
\end{itemize}

Given the above definitions, we also introduce the piecewise linear 
vertex normal function
\begin{equation*} 
\vec\omega^m := \sum_{k=1}^{K^m_\Gamma} \chi^m_k\,\vec\omega^m_k \in \Vh \,,
\end{equation*}
and note that 
\begin{equation} 
\langle \vec{v}, w\,\vec\nu^m\rangle_{\Gamma^m}^h =
\left\langle \vec{v}, w\,\vec\omega^m\right\rangle_{\Gamma^m}^h 
\qquad \forall\ \vec{v} \in \Vh,\ w \in \Wh \,.
\label{eq:NI}
\end{equation}

Following a similar approach used by the authors in \cite{crystal}
in the context of the parametric approximation of dendritic crystal growth,
we consider an unfitted finite element approximation of (\ref{eq:weaka}--d).
On recalling (\ref{eq:partT}), we introduce the discrete density 
$\rho^m \in S^m_0$ and the discrete viscosity $\mu^m \in S^m_0$,
for $m\geq 0$, as either
\begin{subequations}
\begin{equation} \label{eq:rhoma}
\rho^m\!\mid_{o^m} = \begin{cases}
\rho_- & o^m \in \mathcal{T}^m_-\,, \\
\rho_+ & o^m \in \mathcal{T}^m_+\,, \\
\tfrac12\,(\rho_- + \rho_+) & o^m \in \mathcal{T}^m_{\Gamma^m}\,,
\end{cases}
\quad\text{and}\quad
\mu^m\!\mid_{o^m} = \begin{cases}
\mu_- & o^m \in \mathcal{T}^m_-\,, \\
\mu_+ & o^m \in \mathcal{T}^m_+\,, \\
\tfrac12\,(\mu_- + \mu_+) & o^m \in \mathcal{T}^m_{\Gamma^m}\,;
\end{cases}
\end{equation}
or
\begin{align} \label{eq:rhomc}
\rho^m\!\mid_{o^m} & =
v_-(o^m)\, \rho_- + (1 - v_-(o^m))\,\rho_+
\quad\text{and}\quad
\mu^m\!\mid_{o^m} =
v_-(o^m)\, \mu_- + (1 - v_-(o^m))\,\mu_+\,, \nonumber \\
\quad& \text{where}\quad
v_-(o^m) = \frac{\mathcal{L}^d (o^m \cap \Omega^m_-)}{\mathcal{L}^d (o^m)}\,,
\qquad \forall\ o^m \in \mathcal{T}^m\,.
\end{align}
\end{subequations}
We also set $\rho^{-1} := \rho^0$. Clearly (\ref{eq:rhoma},b) only differ in
the definitions of $\rho^m$ and $\mu^m$ on the elements in 
$\mathcal{T}^m_{\Gamma^m}$.

Our finite element approximation for two-phase Navier--Stokes flow
is then given as follows.
Let $\Gamma^0$, an approximation to $\Gamma(0)$, 
and $\vec U^0\in \uspace^0$ be given.
For $m=0,\ldots, M-1$, find $\vec U^{m+1} \in \uspace^m$, 
$P^{m+1} \in \widehat\pspace^m$, $\vec{X}^{m+1}\in\Vh$ and 
$\kappa^{m+1} \in \Wh$ such that
\begin{subequations}
\begin{align}
&
\tfrac12 \left( \frac{\rho^m\,\vec U^{m+1} - (I^m_0\,\rho^{m-1})
\,I^m_2\,\vec U^m}{\tau_m}
+(I^m_0\,\rho^{m-1}) \,\frac{\vec U^{m+1}- I^m_2\,\vec{U}^m}{\tau_m}, \vec \xi 
\right)
 \nonumber \\ & \qquad
+ 2\left(\mu^m\,\mat D(\vec U^{m+1}), \mat D(\vec \xi) \right)
+ \tfrac12\left(\rho^m, 
 [(I^m_2\,\vec U^m\,.\,\nabla)\,\vec U^{m+1}]\,.\,\vec \xi
- [(I^m_2\,\vec U^m\,.\,\nabla)\,\vec \xi]\,.\,\vec U^{m+1} \right)
\nonumber \\ & \qquad
- \left(P^{m+1}, \nabla\,.\,\vec \xi\right)
+ \beta\left\langle \vec U^{m+1}, \vec \xi 
\right\rangle_{\partial_2\Omega, \unitt}
 - \gamma\,\left\langle \kappa^{m+1}\,\vec\nu^m,
   \vec\xi\right\rangle_{\Gamma^m}
\nonumber \\ & \qquad\qquad\qquad\qquad
= \left(\rho^m\,\vec f^{m+1}_1 + \vec f^{m+1}_2, \vec \xi\right)
\qquad \forall\ \vec\xi \in \uspace^m \,, \label{eq:HGa}\\
& \left(\nabla\,.\,\vec U^{m+1}, \varphi\right)  = 0 
\qquad \forall\ \varphi \in \widehat\pspace^m\,,
\label{eq:HGb} \\
&  \left\langle \frac{\vec X^{m+1} - \vec X^m}{\tau_m} ,
\chi\,\vec\nu^m \right\rangle_{\Gamma^m}^h
- \left\langle \vec U^{m+1}, 
\chi\,\vec\nu^m \right\rangle_{\Gamma^m}  = 0
 \qquad\forall\ \chi \in \Wh\,,
\label{eq:HGc} \\
& \left\langle \kappa^{m+1}\,\vec\nu^m, \vec\eta \right\rangle_{\Gamma^m}^h
+ \left\langle \nabs\,\vec X^{m+1}, \nabs\,\vec \eta \right\rangle_{\Gamma^m}
 = 0  \qquad\forall\ \vec\eta \in \Vh
\label{eq:HGd}
\end{align}
\end{subequations}
and set $\Gamma^{m+1} = \vec{X}^{m+1}(\Gamma^m)$.
Here we have defined $\vec f^{m+1}_i(\cdot) := I^m_2\,\vec f_i(\cdot,t_{m+1})$,
$i=1,2$.
We remark that (\ref{eq:HGa}--d), in the case that $\rho_+ = \rho_- = 0$
and $\partial_1\Omega = \partial\Omega$, 
collapses to the finite element approximation for two-phase Stokes flow
introduced in \cite{spurious}. Moreover,
on setting $\rho_+ = \rho_- > 0$, $\mu_+ = \mu_- > 0$ and $\gamma = 0$,
the scheme {\rm (\ref{eq:HGa},b)} 
reduces to a standard finite element approximation of the 
incompressible Navier--Stokes equations in $\Omega$; 
see e.g.\ \cite{Temam01}.

Let 
\begin{equation*} 
\mathcal{E}(\xi,\vec V,\mathcal{M}) := 
\tfrac12\,(\xi\,\vec V, \vec V) + \gamma\, \mathcal{H}^{d-1}(\mathcal{M})\,,
\end{equation*}
where $\xi \in L^1(\Omega)$, $\vec V \in \uspace$ and $\mathcal{M} \subset
\R^d$ is a $(d-1)$-dimensional manifold.

\begin{thm} \label{thm:stab}
Let the assumption ($\mathcal{A}$) hold. Then for $m=0 ,\ldots, M-1$
there exists a unique solution
$(\vec U^{m+1}, P^{m+1}, \vec{X}^{m+1}, \kappa^{m+1}) 
\in \uspace^m\times \widehat\pspace^m \times \Vh \times \Wh$ to 
{\rm (\ref{eq:HGa}--d)}. Moreover, the solution
satisfies
\begin{align}
& \mathcal{E}(\rho^m,\vec U^{m+1}, \Gamma^{m+1})
+ \tfrac12\left((I^m_0\rho^{m-1})\,(\vec U^{m+1} - I^m_2\,\vec U^m), 
\vec U^{m+1} - I^m_2\,\vec U^m \right) \nonumber \\ & \hspace{5cm}
+ 2\,\tau_m\left(\mu^m\,\mat D(U^{m+1}), \mat D(U^{m+1}) \right)
+ \beta\left\langle \vec U^{m+1}, \vec U^{m+1} 
\right\rangle_{\partial_2\Omega, \unitt}
\nonumber \\ & \hspace{3cm}
\leq \mathcal{E}(I^m_0\,\rho^{m-1},I^m_2\,\vec U^m,\Gamma^m) 
+ \tau_m\left( \rho^m\,\vec f^{m+1}_1 + \vec f^{m+1}_2, 
\vec U^{m+1} \right)\,.
\label{eq:stab}
\end{align}
\end{thm}
\begin{proof}
As the system (\ref{eq:HGa}--d) is linear, existence follows from uniqueness.
In order to establish the latter, we consider the system:
Find $(\vec U, P, \vec{X}, \kappa) \in \uspace^m\times\widehat\pspace^m 
\times \Vh \times \Wh$ such that
\begin{subequations}
\begin{align}
&
\tfrac1{2\,\tau_m}\,\left( (\rho^m+I^m_0\,\rho^{m-1})\,\vec U, \vec \xi \right)
+ 2\left(\mu^m\,\mat D(\vec U), \mat D(\vec \xi) \right)
- \left(P, \nabla\,.\,\vec \xi\right)
\nonumber \\ & \qquad
+ \tfrac12\left(\rho^m, [(I^m_2\,\vec U^m\,.\,\nabla)\,\vec U]\,.\,\vec \xi
- [(I^m_2\,\vec U^m\,.\,\nabla)\,\vec \xi]\,.\,\vec U \right)
 \nonumber \\ & \qquad
+ \beta\left\langle \vec U, \vec \xi \right\rangle_{\partial_2\Omega, \unitt}
 - \gamma\,\left\langle \kappa\,\vec\nu^m, \vec\xi\right\rangle_{\Gamma^m}
= 0 
 \qquad \forall\ \vec\xi \in \uspace^m \,, \label{eq:proofa}\\
& \left(\nabla\,.\,\vec U, \varphi\right)  = 0 
\qquad \forall\ \varphi \in \widehat\pspace^m\,,
\label{eq:proofb} \\
&  \left\langle \frac{\vec X}{\tau_m} ,
\chi\,\vec\nu^m \right\rangle_{\Gamma^m}^h
- \left\langle \vec U, 
\chi\,\vec\nu^m \right\rangle_{\Gamma^m} = 0
 \qquad\forall\ \chi \in \Wh\,,
\label{eq:proofc} \\
& \left\langle \kappa\,\vec\nu^m, \vec\eta \right\rangle_{\Gamma^m}^h
+ \left\langle \nabs\,\vec X, \nabs\,\vec \eta \right\rangle_{\Gamma^m}
 = 0  \qquad\forall\ \vec\eta \in \Vh\,.
\label{eq:proofd}
\end{align}
\end{subequations}
Choosing $\vec\xi=\vec U$ in (\ref{eq:proofa}), 
$\varphi =  P$ in (\ref{eq:proofb}), 
$\chi = \gamma\,\kappa$ in (\ref{eq:proofc}) 
and $\vec\eta=\gamma\,\vec{X}$ in (\ref{eq:proofd})
yields 
that
\begin{align}
& \tfrac12\left((\rho^m + I^m_0\,\rho^{m-1})\,\vec U, \vec U \right) + 
2\,\tau_m\left(\mu^m\,\mat D(\vec U), \mat D(\vec U) \right)
+ \beta\left\langle \vec U, \vec U \right\rangle_{\partial_2\Omega, \unitt}
\nonumber \\ & \qquad\qquad
+ \gamma\,\left\langle \nabs\,\vec{X}, \nabs\,\vec{X} \right\rangle_{\Gamma^m} 
=0\,. \label{eq:proof2}
\end{align}
It immediately follows from (\ref{eq:proof2}) and a Korn's inequality
that $\vec U = \vec 0 \in \uspace^m$.
In addition, it holds that $\vec{X}\equiv \vec{X}_c \in \R^d$. 
Together with (\ref{eq:proofc}) for $\vec U=\vec 0$, 
(\ref{eq:NI}) and the assumption $(\mathcal{A})$ this
immediately yields that $\vec{X} \equiv \vec0$, while
(\ref{eq:proofd}) with $\vec\eta=\vec\pi^m[\kappa\,\vec\omega^m]$, 
recall (\ref{eq:NI}), implies that $\kappa \equiv 0$.
Finally, it now follows from (\ref{eq:proofa}) with $\vec U = \vec 0$ and
$\kappa = 0$, on recalling (\ref{eq:LBB}), that $P = 0 \in \widehat\pspace^m$.
Hence there exists a unique solution
$(\vec U^{m+1}, P^{m+1}, \vec{X}^{m+1}, \kappa^{m+1}) \in \uspace^m\times
\widehat\pspace^m \times \Vh \times \Wh$ to (\ref{eq:HGa}--d).

It remains to establish the bound (\ref{eq:stab}). 
Choosing $\vec\xi = \vec U^{m+1}$ in (\ref{eq:HGa}), 
$\varphi = P^{m+1}$ in (\ref{eq:HGb}), 
$\chi = \gamma\,\kappa^{m+1}$ in (\ref{eq:HGc}) and
$\vec\eta=\gamma\,({\vec{X}^{m+1}-\vec{X}^m})$ in (\ref{eq:HGd}) yields that
\begin{align*}
& \tfrac12\left(\rho^m\,\vec U^{m+1}, \vec U^{m+1}\right)
+ \tfrac12\left((I^m_0\,\rho^{m-1})\,(\vec U^{m+1} - I^m_2\,\vec U^m), 
\vec U^{m+1} - I^m_2\,\vec U^m \right) \nonumber \\ & \hspace{2cm}
+ 2\,\tau_m\left(\mu^m\,\mat D(U^{m+1}), \mat D(U^{m+1}) \right)
+ \beta\left\langle \vec U^{m+1}, \vec U^{m+1} 
\right\rangle_{\partial_2\Omega, \unitt} \nonumber \\ & \hspace{2cm} 
+ \gamma\,\left\langle \nabs\,\vec{X}^{m+1}, \nabs\,(\vec{X}^{m+1} - \vec{X}^m) 
\right\rangle_{\Gamma^m} \nonumber \\ & \hspace{1cm}
= \tfrac12\left((I^m_0\,\rho^{m-1})\,I^m_2\,\vec U^{m}, I^m_2\,\vec U^{m}\right)
+ \tau_m\left( \rho^m\,\vec f^{m+1}_1 + \vec f^{m+1}_2,
  \vec U^{m+1} \right)\,.
\end{align*}
and hence (\ref{eq:stab}) follows immediately, where
we have used the result that
\begin{equation*}
\left\langle \nabs\,\vec{X}^{m+1}, \nabs\,(\vec{X}^{m+1} - \vec{X}^m) 
\right\rangle_{\Gamma^m}
\geq \mathcal{H}^{d-1}(\Gamma^{m+1}) - \mathcal{H}^{d-1}(\Gamma^{m})
\end{equation*}
see e.g.\ \cite{triplej} and \cite{gflows3d} 
for the proofs for $d=2$ and $d=3$, respectively.
\end{proof}

The above theorem allows us to prove unconditional stability, in terms of the
chosen time step sizes, for our scheme under certain conditions.
\begin{thm} \label{thm:stabstab}
Let the assumption ($\mathcal{A}$) hold and let $\{t_i\}_{i=0}^M$ be an
arbitrary partitioning of $[0,T]$. In addition, assume that 
\begin{equation} \label{eq:thm42ass}
((I^m_0\rho^{m-1})\,I^m_2\,\vec U^{m}, I^m_2\,\vec U^{m}) \leq
(\rho^{m-1}\,\vec U^{m}, \vec U^{m})
\quad \text{for $m=1,\ldots, M-1$.}
\end{equation}
Then it holds that
\begin{align}
&\mathcal{E}(\rho^m,\vec U^{m+1}, \Gamma^{m+1}) 
+ \tfrac12\,\sum_{k=0}^m \left[\left(\rho^{k-1}\,
(\vec U^{k+1} - I^k_2\,\vec U^k), 
\vec U^{k+1} - I^k_2\,\vec U^k \right) \right. \nonumber \\ & \hspace{5cm}
\left.
+ 4\,\tau_k\left(\mu^k\,\mat D(\vec U^{k+1}), \mat D(\vec U^{k+1})
\right)
+ 2\,\beta\,\tau_k\left\langle \vec U^{k+1}, \vec U^{k+1} 
\right\rangle_{\partial_2\Omega, \unitt}
\right]
\nonumber \\ & \hspace{2cm}
\leq \mathcal{E}(\rho^{-1},\vec U^0, \Gamma^0)
+ \sum_{k=0}^m \tau_k\left(\rho^k\,\vec f^{k+1}_1 + \vec f^{k+1}_2, 
 \vec U^{k+1} \right)
\label{eq:stabstab}
\end{align}
for $m=0,\ldots, M-1$.
\end{thm}
\begin{proof}
The result immediately follows from (\ref{eq:stab}) on noting that 
our assumptions yield that
$\mathcal{E}(I^m_0\,\rho^{m-1}, I^m_2\,\vec U^{m}, \vec{X}^{m}) \leq
\mathcal{E}(\rho^{m-1}, \vec U^{m}, \vec{X}^{m})$
for $m=1,\ldots, M-1$.
\end{proof}

The assumption (\ref{eq:thm42ass}) for {\rm Theorem~\ref{thm:stabstab}} is 
trivially satisfied in the case $\rho_+=\rho_-=0$,
see also \cite{spurious}. Otherwise it is
for instance satisfied when either (i) $\mathcal{T}^m = \mathcal{T}^0$ for 
$m=1,\ldots, M-1$; i.e.\ when no mesh adaptation is employed, or when (ii) 
$\uspace^{m-1} \subset \uspace^m$ for $m=1,\ldots, M-1$;
e.g.\ when mesh refinement routines without coarsening are employed.
In principle, one can completely avoid the assumption (\ref{eq:thm42ass}) by
considering a variant of (\ref{eq:HGa}) in
which $I^m_0\,\rho^{m-1}$ is replaced by $\rho^{m-1}$ and $I^m_2\,\vec U^{m}$ 
is replaced by $\vec U^{m}$, i.e.\ no interpolation to the current finite
element spaces is used for the solutions from the previous time step.
For this approach Theorem~\ref{thm:stabstab} holds without 
assumption (\ref{eq:thm42ass}). However, as this strategy requires the
computation of integrals involving finite element functions from two different
spatial meshes, its implementation is far more involved than the 
implementation of \mbox{(\ref{eq:HGa}--d)}.
For all the computations presented in this paper we will always use the more
practical variant (\ref{eq:HGa}--d).

The stability bounds (\ref{eq:stab}) and (\ref{eq:stabstab}) control
the total surface area (length in two dimensions)
$\mathcal{H}^{d-1}(\Gamma^{m+1})$ and correspond to the continuous
energy bound (\ref{eq:testD}). For a larger surface energy density
$\gamma$ this control is stronger and fluid drops are less
unstable. However, if a stable numerical scheme does not conserve the total
volume (area in two dimensions) of the fluid phases, a large value of
$\gamma$ can lead to situations where drops decrease their size during
the evolution in order to reduce their surface area. Of course this is
an artefact and has no analogue in the continuous problem. 
Conversely, an unstable numerical scheme that does conserve the total
volume of the fluid phases may for large values of $\gamma$ suffer from
oscillations in the discrete representation of the interface. Hence for
numerical approximations of two-phase Navier--Stokes flow it is important to
have both: stability and conservation of the phase volumes.

The stability bounds
(\ref{eq:stab}) and (\ref{eq:stabstab}) are the main result of this paper. In
practice we observe that for large values of $\gamma$, the numerical solution
is {\it better} behaved than for small $\gamma$, in analogue to the continuous
situation. We note that this is in contrast to existing methods for two-phase
Navier--Stokes flow, for which no stability results can be shown; see e.g.\
\cite[p.\ 280]{GrossR11}.

\begin{rem} \label{rem:nonlinear}
We stress that our approximation {\rm (\ref{eq:HGa}--d)} 
results in a linear system of equations. This
is due to the lagging in the approximation $\rho^m$ of the densities.
Alternatively, one could choose to not lag $\rho^m$ and then obtain 
the following nonlinear approximation. 
For $m=0,\ldots, M-1$, find $\vec U^{m+1} \in \uspace^m$, 
$P^{m+1} \in \widehat\pspace^m$, $\vec{X}^{m+1}\in\Vh$ and 
$\kappa^{m+1} \in \Wh$ such that
\begin{align}
& \tfrac12 \left( \frac{\rho^{m+1}\,\vec U^{m+1} - (I^m_0\,\rho^{m})\,
I^m_2\,\vec U^m}{\tau_m}
+ (I^m_0\,\rho^m)\,\frac{\vec U^{m+1} - I^m_2\,\vec U^m}{\tau_m},\vec \xi
  \right)
+ 2\left(\mu^m\,\mat D(\vec U^{m+1}), \mat D(\vec \xi) \right)
\nonumber \\ & \qquad
- \left(P^{m+1}, \nabla\,.\,\vec \xi\right)
+ \tfrac12\left(\rho^{m+1}, 
[(I^m_2\,\vec U^m\,.\,\nabla)\,\vec U^{m+1}]\,.\,\vec \xi
- [(I^m_2\,\vec U^m\,.\,\nabla)\,\vec \xi]\,.\,\vec U^{m+1} \right)
 \nonumber \\ & \qquad
+ \beta\left\langle \vec U^{m+1}, \vec \xi 
\right\rangle_{\partial_2\Omega, \unitt}
 - \gamma\,\left\langle \kappa^{m+1}\,\vec\nu^m,
   \vec\xi\right\rangle_{\Gamma^m}
= \left(\rho^{m+1}\,\vec f^{m+1}_1 + \vec f^{m+1}_2, \vec \xi\right)
\qquad \forall\ \vec\xi \in \uspace^m  \label{eq:HGaa}
\end{align}
and {\rm (\ref{eq:HGb}--d)} hold. Now, as $\rho^{m+1}$, via the analogues of
{\rm (\ref{eq:rhoma},b)},
depends on $\Gamma^{m+1} = \vec X^{m+1}(\Gamma^m)$, the system
{\rm (\ref{eq:HGaa})}, {\rm (\ref{eq:HGb}--d)} is highly nonlinear. 
Assuming existence of a solution, 
it is then straightforward to establish the corresponding stability results,
i.e.\ {\rm (\ref{eq:stab})} and {\rm (\ref{eq:stabstab})} with $\rho^{\ell-1}$
replaced by $\rho^{\ell}$, for $\ell \geq 0$.
\end{rem}

\begin{rem} \label{rem:cons}
It is worthwhile to consider a continuous-in-time semidiscrete version of our 
scheme {\rm (\ref{eq:HGa}--d)}. Let $\mathcal{T}^h$ be an arbitrarily fixed
regular partitioning of $\Omega$ into disjoint open simplices
and define the finite element spaces $S^h_k$, $\uspace^h$ and
$\pspace^h$ similarly to $S^m_k$, $\uspace^m$ and $\pspace^m$, with the
corresponding interpolation operators $I^h_k$ and 
discrete approximations $\rho^h(t) \in S^h_0$ and $\mu^h(t) \in S^h_0$, which
will depend on $\Gamma^h(t)$ via the analogues of {\rm (\ref{eq:rhoma},b)}.
Then, given
$\Gamma^h(0)$ and $\vec U^h(0) \in \uspace^h$, for $t\in (0,T]$ find
$\vec U^h(t) \in \uspace^h$, $P^h(t) \in \widehat\pspace^h$,
$\vec{X}^h(t)\in \Vht$ and $\kappa^h(t) \in \Wht$ such that
\begin{subequations}
\begin{align}
&
\tfrac12 \left[ \ddt \left( \rho^h\,\vec U^h , \vec \xi \right)
+ \left( \rho^h\,\vec U^h_t , \vec \xi \right)
- (\rho^h\,\vec U^h, \vec \xi_t) \right]
+ 2\left(\mu^h\,\mat D(\vec U^h), \mat D(\vec \xi) \right)
\nonumber \\ & \qquad
- \left(P^h, \nabla\,.\,\vec \xi\right)
+ \tfrac12\left(\rho^h, [(\vec U^h\,.\,\nabla)\,\vec U^h]\,.\,\vec \xi
- [(\vec U^h\,.\,\nabla)\,\vec \xi]\,.\,\vec U^h \right)
 \nonumber \\ & \qquad
+ \beta\left\langle \vec U^h, \vec \xi \right\rangle_{\partial_2\Omega, \unitt}
 - \gamma\,\left\langle \kappa^h\,\vec\nu^h,
   \vec\xi\right\rangle_{\Gamma^h(t)}
= \left(\rho^h\,\vec f^h_1 + \vec f^h_2, \vec \xi\right)
\qquad \forall\ \vec\xi \in \uspace^h \,, \label{eq:sda}\\
& \left(\nabla\,.\,\vec U^h, \varphi\right)  = 0 
\qquad \forall\ \varphi \in \widehat\pspace^h\,,
\label{eq:sdb} \\
&  \left\langle \vec X^h_t ,
\chi\,\vec\nu^h \right\rangle_{\Gamma^h(t)}^h
- \left\langle \vec U^h, 
\chi\,\vec\nu^h \right\rangle_{\Gamma^h(t)} = 0
 \qquad\forall\ \chi \in \Wht\,,
\label{eq:sdc} \\
& \left\langle \kappa^h\,\vec\nu^h, \vec\eta \right\rangle_{\Gamma^h(t)}^h
+ \left\langle \nabs\,\vec X^h, \nabs\,\vec \eta \right\rangle_{\Gamma^h(t)}
 = 0  \qquad\forall\ \vec\eta \in \Vht\,,
\label{eq:sdd}
\end{align}
\end{subequations}
where $\vec f^h_i := I^h_2\,\vec f_i(t)$, $i=1,2$.
In {\rm (\ref{eq:sda}--d)} 
we always integrate over the current surface $\Gamma^h(t)$, with normal
$\vec\nu^h(t)$, described by the identity function 
$\vec{X}^h(t) \in \Vht$.
Moreover, $\langle \cdot,\cdot\rangle_{\Gamma^h(t)}^h$
is the same as $\langle \cdot,\cdot \rangle_{\Gamma^m}^h$ with 
$\Gamma^m$ and $\vec{X}^m$ replaced by $\Gamma^h(t)$ and $\vec{X}^h(t)$, 
respectively;
and similarly for $\langle \cdot,\cdot\rangle_{\Gamma^h(t)}$.

Using the results from \cite{gflows3d} it is straightforward to show that
\begin{equation*}
\ddt \mathcal{H}^{d-1}(\Gamma^h(t)) =  
\left\langle \nabs\,\vec{X}^h, \nabs\,\vec{X}^h_t \right\rangle_{\Gamma^h(t)}
\,.
\end{equation*}
It is then not difficult to derive the following
stability bound for the solution $(\vec U^h, P^h, \vec{X}^h, \kappa^h)$ of the
semidiscrete scheme {\rm (\ref{eq:sda}--d)}{\rm :}
\begin{align}
& \ddt\left(\tfrac12\,\|[\rho^h]^\frac12\,\vec U^h\|^2_{0} + 
\gamma\,\mathcal{H}^{d-1}(\Gamma^h(t)) \right) 
+ 2\,\|[\mu^h]^\frac12\,\mat D(\vec U^h)\|^2_{0}
+ \beta\left\langle \vec U^h, \vec U^h \right\rangle_{\partial_2\Omega, \unitt}
\nonumber \\ & \qquad \qquad
= \left(\rho^h\,\vec f^h_1 + \vec f^h_2, \vec U^h\right) \,.
\label{eq:stabsd}
\end{align}
Clearly, {\rm (\ref{eq:stabsd})} is the natural discrete analogue of
{\rm (\ref{eq:testD})}.
In addition, it is possible to prove that the vertices of $\Gamma^h(t)$ are
well distributed. As this follows already from the equations 
{\rm (\ref{eq:sdd})}, we
refer to our earlier work in \cite{triplej,gflows3d} for further details. In
particular, we observe that in the case $d=2$, i.e.\ for the planar two-phase
problem, an equidistribution property for the vertices of $\Gamma^h(t)$ can be
shown.
\end{rem}

\subsection{\XFEMGAMMA\ for conservation of the phase volumes} \label{sec:31}
In general, the fully discrete approximation (\ref{eq:HGa}--d) will not
conserve the volume of the two phase regions, i.e.\
the volume $\mathcal{L}^d(\Omega^m_-)$, 
enclosed by $\Gamma^m$ will in general not be preserved.
Clearly, given that volume conservation holds on the continuous level, recall
(\ref{eq:conserved}), it would be desirable to have an analogue property also
on the discrete level.

For the semidiscrete approximation {\rm (\ref{eq:sda}--d)} 
from Remark~\ref{rem:cons} we can show the following 
true volume conservation property in the case that 
\begin{equation} \label{eq:XFEM1}
\charfcn{\Omega_-^h(t)} \in \pspace^h\,,
\end{equation}
where here we need to allow the pressure space to be time-dependent.
Choosing $\chi=1$ in {\rm (\ref{eq:sdc})} and
$\varphi=(\charfcn{\Omega_-^h(t)} -
\frac{\mathcal{L}^d(\Omega_-^h(t))}{\mathcal{L}^d(\Omega)})
\in \widehat\pspace^h(t)$ in {\rm (\ref{eq:sdb})}, we
then obtain that
\begin{equation}
\frac{\rm d}{{\rm d}t} \vol(\Omega_-^h(t)) = 
\left\langle \vec{X}^h_t , \vec\nu^h \right\rangle_{\Gamma^h(t)}
= \left\langle \vec{X}^h_t , \vec\nu^h \right\rangle^h_{\Gamma^h(t)}
= \left\langle \vec U^h, \vec\nu^h \right\rangle_{\Gamma^h(t)}
= \int_{\Omega_-^h(t)} \nabla\,.\,\vec U^h \dL{d} 
=0\,; \label{eq:cons}
\end{equation}
which is the discrete analogue of {\rm (\ref{eq:conserved})}. 
Clearly, for the discrete pressure spaces $\pspace^h$ induced by 
{\rm (\ref{eq:P2P1}--c)} the condition {\rm (\ref{eq:XFEM1})} 
will in general not hold.
However, the assumption {\rm (\ref{eq:XFEM1})} 
can be satisfied with the help of the extended finite element method (XFEM), 
see e.g.\ \cite[\S7.9.2]{GrossR11}. 
Here the pressure spaces $\pspace^m$ need to
be suitably extended, so that they satisfy the time-discrete analogue of
(\ref{eq:XFEM1}), i.e.\ 
$\charfcn{\Omega_-^m} \in \pspace^m$, 
which means that then (\ref{eq:HGb}) implies
\begin{equation} \label{eq:HGbXFEM1}
\left\langle \vec U^{m+1}, \vec\nu^m \right\rangle_{\Gamma^m} = 0\,,
\end{equation}
which together with (\ref{eq:HGc})
then yields that
\begin{equation} \label{eq:consm}
\left\langle \vec X^{m+1} - \vec X^m, \vec \nu^m \right\rangle^h_{\Gamma^m} =
0\,.
\end{equation}
We recall that for area/volume preserving geometric evolution equations the
authors, in previous works, 
observed excellent area/volume preservation for fully discrete finite
element approximations satisfying (\ref{eq:consm}), see 
\cite{triplej,gflows3d}. 

Hence the obvious strategy to guarantee (\ref{eq:consm}) is to
add only a single new basis function to the basis of $\pspace^m$, 
namely $\charfcn{\Omega_-^m}$. Then the new contributions to (\ref{eq:HGa},b)
can be written in terms of integrals over $\Gamma^m$, since
\begin{equation*} 
\left( \nabla\,.\,\vec \xi, \charfcn{\Omega_-^m} \right) =
\int_{\Omega_-^m} \nabla\,.\,\vec \xi \dL{d} = 
\left\langle \vec \xi, \vec\nu^m \right\rangle_{\Gamma^m}
\qquad \forall\ \vec \xi \in \uspace^m\,.
\end{equation*}

In the remainder of this subsection we are going to consider this approach. Let
$(\uspace^m, \widehat\pspace^m)$ satisfy (\ref{eq:LBB}). For example,  
$\pspace^m$ may be given by one of (\ref{eq:P2P1}--c). Then we let 
$\phi^{\pspace^m}_{{K_\pspace^m}+1} := \charfcn{\Omega_-^m}$ and define
$\pspace^m_{\rm XFEM} := \spa \{ \phi^{\pspace^m}_i \}_{i=1}^{K_\pspace^m+1}$,
with $\widehat \pspace^m_{\rm XFEM} := \pspace^m_{\rm XFEM} \cap
\widehat \pspace$; recall (\ref{eq:UP}).

We are unable to prove that the element $(\uspace^m, 
\widehat\pspace^m_{\rm XFEM})$ satisfies an LBB condition, i.e.\
that
\begin{equation} \label{eq:LBBXFEM1}
\inf_{\varphi \in \widehat\pspace^m_{\rm XFEM}} \sup_{\vec \xi \in \uspace^m}
\frac{( \varphi, \nabla \,.\,\vec \xi)}
{\|\varphi\|_0\,\|\vec \xi\|_1} \geq C_1 > 0\,,
\end{equation}
where the constant $C_1 \in \R_{>0}$ is independent of $h^m$. The lack of a
proof for {\rm (\ref{eq:LBBXFEM1})} means that we cannot prove existence and 
uniqueness of the discrete pressure $P^{m+1} \in \widehat\pspace^m_{\rm XFEM}$
for the system {\rm (\ref{eq:HGa}--d)}, {\rm (\ref{eq:HGbXFEM1})}. It is for
this reason that we consider the following reduced system for our existence 
result for the extended pressure space $\pspace^m_{\rm XFEM}$ instead.

Let $\uspace^m_0 := 
\{ \vec U \in \uspace^m : (\nabla\,.\,\vec U, \varphi) = 0 \ \
\forall\ \varphi \in \pspace^m_{\rm XFEM} \}$. Then
any solution \linebreak $(\vec U^{m+1}, P^{m+1}, \vec{X}^{m+1}, \kappa^{m+1}) 
\in \uspace^m\times \widehat\pspace^m_{\rm XFEM} \times \Vh \times \Wh$
to {\rm (\ref{eq:HGa}--d)}, {\rm (\ref{eq:HGbXFEM1})}
is such that $(\vec U^{m+1}, \vec{X}^{m+1}, \kappa^{m+1})
\in \uspace^m_0\times \Vh \times \Wh$ satisfies
\begin{subequations}
\begin{align}
&
\tfrac12 \left( \frac{\rho^m\,\vec U^{m+1} - (I^m_0\,\rho^{m-1})
\,I^m_2\,\vec U^m}{\tau_m}
+(I^m_0\,\rho^{m-1}) \,\frac{\vec U^{m+1}- I^m_2\,\vec{U}^m}{\tau_m}, \vec \xi 
\right)
 \nonumber \\ & \quad
+ 2\left(\mu^m\,\mat D(\vec U^{m+1}), \mat D(\vec \xi) \right)
+ \tfrac12\left(\rho^m, 
 [(I^m_2\,\vec U^m\,.\,\nabla)\,\vec U^{m+1}]\,.\,\vec \xi
- [(I^m_2\,\vec U^m\,.\,\nabla)\,\vec \xi]\,.\,\vec U^{m+1} \right)
\nonumber \\ & \quad
+ \beta\left\langle \vec U^{m+1}, \vec \xi 
\right\rangle_{\partial_2\Omega, \unitt}
 - \gamma\,\left\langle \kappa^{m+1}\,\vec\nu^m,
   \vec\xi\right\rangle_{\Gamma^m}
= \left(\rho^m\,\vec f^{m+1}_1 + \vec f^{m+1}_2, \vec \xi\right)
\quad \forall\ \vec\xi \in \uspace^m_0 \,, \label{eq:reda}\\
&  \left\langle \frac{\vec X^{m+1} - \vec X^m}{\tau_m} ,
\chi\,\vec\nu^m \right\rangle_{\Gamma^m}^h
- \left\langle \vec U^{m+1}, 
\chi\,\vec\nu^m \right\rangle_{\Gamma^m}  = 0
 \qquad\forall\ \chi \in \Wh\,,
\label{eq:redb} \\
& \left\langle \kappa^{m+1}\,\vec\nu^m, \vec\eta \right\rangle_{\Gamma^m}^h
+ \left\langle \nabs\,\vec X^{m+1}, \nabs\,\vec \eta \right\rangle_{\Gamma^m}
 = 0  \qquad\forall\ \vec\eta \in \Vh \,.
\label{eq:redc}
\end{align}
\end{subequations}

\begin{thm} \label{thm:stabXFEM1}
Let the assumption 
($\mathcal{A}$) hold. Then there exists a unique solution \linebreak
$(\vec U^{m+1}, \vec{X}^{m+1}, \kappa^{m+1}) 
\in \uspace^m_0\times \Vh \times \Wh$ to {\rm (\ref{eq:reda}--c)}. 
Moreover, the solution satisfies the stability bound {\rm (\ref{eq:stab})}.
\end{thm}
\begin{proof}
As $\uspace^m_0$ is a subspace of $\uspace^m$, existence to the linear system
(\ref{eq:reda}--c) follows from uniqueness, which is easy to show.
In fact, similarly to the proof of
Theorem~\ref{thm:stab} we obtain (\ref{eq:proof2}) and hence the desired
uniqueness result. The stability result follows analogously.
\end{proof}

\subsection{Alternative curvature treatment} \label{sec:36}
There is an alternative way to 
approximate the curvature vector $\varkappa\,\vec\nu$ 
in (\ref{eq:LBop}). In contrast to
the strategy employed in (\ref{eq:HGa}--d), where $\varkappa$ and $\vec\nu$ are
discretized separately, it is also possible to discretize 
$\vec\varkappa:= \varkappa\,\vec\nu$ directly, as proposed in the seminal paper
\cite{Dziuk91}. We then obtain the following finite element approximation.
For $m=0\to M-1$, find $\vec U^{m+1} \in \uspace^m$, 
$P^{m+1} \in \widehat\pspace^m$, $\vec{X}^{m+1}\in\Vh$ and 
$\vec\kappa^{m+1} \in \Vh$ such that
\begin{subequations}
\begin{align}
&
\tfrac12 \left( \frac{\rho^m\,\vec U^{m+1} - (I^m_0\,\rho^{m-1})
\,I^m_2\,\vec U^m}{\tau_m}
+(I^m_0\,\rho^{m-1}) \,\frac{\vec U^{m+1}- I^m_2\,\vec{U}^m}{\tau_m}, \vec \xi 
\right)
 \nonumber \\ & \qquad
+ 2\left(\mu^m\,\mat D(\vec U^{m+1}), \mat D(\vec \xi) \right)
+ \tfrac12\left(\rho^m, 
 [(I^m_2\,\vec U^m\,.\,\nabla)\,\vec U^{m+1}]\,.\,\vec \xi
- [(I^m_2\,\vec U^m\,.\,\nabla)\,\vec \xi]\,.\,\vec U^{m+1} \right)
\nonumber \\ & \qquad
- \left(P^{m+1}, \nabla\,.\,\vec \xi\right)
+ \beta\left\langle \vec U^{m+1}, \vec \xi 
\right\rangle_{\partial_2\Omega, \unitt}
 - \gamma\,\left\langle \vec\kappa^{m+1},
   \vec\xi\right\rangle_{\Gamma^m}
\nonumber \\ & \qquad\qquad\qquad\qquad
= \left(\rho^m\,\vec f^{m+1}_1 + \vec f^{m+1}_2, \vec \xi\right)
\qquad \forall\ \vec\xi \in \uspace^m \,, \label{eq:GDa}\\
& \left(\nabla\,.\,\vec U^{m+1}, \varphi\right) = 0 
\qquad \forall\ \varphi \in \widehat\pspace^m\,,
\label{eq:GDb} \\
& \left\langle \frac{\vec X^{m+1} - \vec X^m}{\tau_m} ,
\vec\chi \right\rangle_{\Gamma^m}^h
- \left\langle \vec U^{m+1}, \vec\chi \right\rangle_{\Gamma^m}  = 0
 \qquad\forall\ \vec\chi \in \Vh\,,
\label{eq:GDc} \\
& \left\langle \vec\kappa^{m+1} , \vec\eta \right\rangle_{\Gamma^m}^h
+ \left\langle \nabs\,\vec X^{m+1}, \nabs\,\vec \eta \right\rangle_{\Gamma^m}
 = 0  \qquad\forall\ \vec\eta \in \Vh
\label{eq:GDd}
\end{align}
\end{subequations}
and set $\Gamma^{m+1} = \vec{X}^{m+1}(\Gamma^m)$. A 
highly nonlinear discretization based on (\ref{eq:GDa}--d)
has first been proposed by B{\"a}nsch in \cite{Bansch01} 
for one-phase flow with a free capillary surface in the very
special situation that
\begin{equation*} 
\vec\xi\!\mid_{\Gamma^m} \in \Vh \quad \forall\ \vec\xi\in \uspace^m\,,
\end{equation*}
which in general cannot be satisfied for the unfitted approach.
It is not difficult to extend the results from Theorem~\ref{thm:stab}
to the linear scheme (\ref{eq:GDa}--d). 
However, the crucial difference between (\ref{eq:GDa}--d) and 
(\ref{eq:HGa}--d) is that in (\ref{eq:GDc})
the tangential velocity of the discrete interface is fixed by $\vec U^{m+1}$, 
and this has two consequences. Firstly, there is no
guarantee that the mesh quality of $\Gamma^m$ will be preserved. In fact, as
mentioned in the Introduction, typically the mesh will deteriorate over time.
And secondly, even for the case that $\charfcn{\Omega_-^m} \in \pspace^m$, 
it is not possible to prove (\ref{eq:consm}) for \mbox{(\ref{eq:GDa}--d)}, 
as $\vec\chi = \vec\nu^m$ is not a valid test function in (\ref{eq:GDc}),
and so true 
volume conservation in the semidiscrete setting, recall (\ref{eq:cons}), 
cannot be shown.
It is for these reasons that we prefer to use (\ref{eq:HGa}--d). 

\subsection{The fitted mesh approach} \label{sec:33}

Although in deriving the finite element approximation (\ref{eq:HGa}--d) we have
assumed an {\em unfitted} bulk mesh triangulation $\mathcal{T}^m$ that is
independent of $\Gamma^m$, the approximation (\ref{eq:HGa}--d) can also be
employed for a {\em fitted} bulk mesh.
In particular, it is possible to use (\ref{eq:HGa}--d) for a moving
fitted mesh approach, where for $m = 0 ,\ldots, M-1$ it holds that
\begin{equation} \label{eq:fit}
\Gamma^m \subset \bigcup_{o^m \in \mathcal{T}^m} \partial o^m\,.
\end{equation}
Here the solution $\vec X^{m+1}$ to (\ref{eq:HGa}--d) defines the 
position of $\Gamma^{m+1}$, as usual, but now a new bulk mesh 
$\mathcal{T}^{m+1}$ needs to be obtained by fitting it to $\Gamma^{m+1}$.

The main advantages of the fitted mesh approach (\ref{eq:fit}) over the 
unfitted approach are that 
(a) with the elements (\ref{eq:P2P0},c) 
the discontinuity in the pressure at $\Gamma^m$ can be
resolved and that (b) the inner products 
$\langle \cdot, \cdot \rangle_{\Gamma^m}$ in (\ref{eq:HGa},c) 
now only involve integration over edges/faces of bulk
elements $o^m \in \mathcal{T}^m$, which is standard.
A further consequence of (a) is that for the elements (\ref{eq:P2P0},c) 
it automatically holds that
$\charfcn{\Omega^m_-}$ is an admissible test function in (\ref{eq:HGb}), which
yields (\ref{eq:consm}), and so approximate volume conservation on the fully
discrete level. 
An additional advantages of the fitted mesh approach is that, since
$\mathcal{T}^m_{\Gamma^m} = \emptyset$, it holds that (\ref{eq:rhoma},b)
reduce to 
$\rho^m = \rho_-\,\charfcn{\Omega^m_-} + \rho_+\,(1-\charfcn{\Omega^m_-})$
and
$\mu^m = \mu_-\,\charfcn{\Omega^m_-} + \mu_+\,(1-\charfcn{\Omega^m_-})$.

We stress that the fitted mesh approach (\ref{eq:fit}) for 
(\ref{eq:HGa}--d) would also satisfy the stability result (\ref{eq:stab}).
However, due to the nature of the moving bulk mesh, the assumptions of
Theorem~\ref{thm:stabstab}, for $\rho_\pm \not=0$, in general do not hold, and
so the stability result (\ref{eq:stabstab}) need not hold over
several time steps. Another disadvantage of the fitted mesh approach is that at
every time step, due to the fact that the underlying bulk mesh changes
globally, the obtained velocity solution $\vec U^{m+1} \in \uspace^m$ needs to 
be appropriately interpolated on the new mesh $\mathcal{T}^{m+1}$. 
These interpolation errors may significantly
impact on the accuracy of the approximation.

A variant of the method described above, which avoids the repeated
interpolation onto a new finite element bulk mesh, is the so-called 
Arbitrary Lagrangian Eulerian (ALE) approach, see \cite{HughesLZ81}. 
Here a prescribed flow drives the
movement of the bulk mesh vertices, and this prescribed flow 
needs to be accounted for in the approximation of the momentum equation.
This means that at present it does not
appear possible to prove a stability result similar to
Theorem~\ref{thm:stabstab} for the ALE approach.
On the other hand, the fact that the movement of the bulk mesh is incorporated
in the finite element approximation means that
an interpolation of finite element
data after every time step is not needed. The prescribed bulk mesh flow
is usually chosen in a way to obtain a good quality bulk mesh.
However, in practice the ALE moving mesh approach may fail if
the approximated phases change their shape dramatically, as was reported in
e.g.\ \cite{HysingTKPBGT09} for the two-dimensional test case 2 there. In
higher space dimensions such pathological mesh defects are more frequent, which
poses a significant computational challenge.
For further details on the ALE approach for two-phase Navier--Stokes flow
we refer to \cite{Ganesan06,GerbeauLL06}.

In this paper we will focus on the general unfitted mesh approach, similarly to
our previous work for Stefan problems
in \cite{dendritic,crystal}. An investigation of the possible
benefits of fitted mesh approaches for (\ref{eq:HGa}--d), and in particular
of ALE methods, is left for future research.

\setcounter{equation}{0}
\section{Solution of the discrete system}  \label{sec:4}
As is standard practice for the solution of linear systems arising from
discretizations of Stokes and Navier--Stokes equations, we avoid the
complications of the constrained pressure space $\widehat\pspace^m$ in practice
by considering an overdetermined linear system with $\pspace^m$ instead. 
Introducing the obvious abuse of notation, the linear system (\ref{eq:HGa}--d),
with $\widehat\pspace^m$ replaced by $\pspace^m$,
can be formulated as: Find $(\vec U^{m+1},P^{m+1},
\kappa^{m+1},\delta\vec{X}^{m+1})$, where $\vec X^{m+1} = \vec X^m + 
\delta\vec X^{m+1}$, such that
\begin{subequations}
\begin{equation}
\begin{pmatrix}
 \vec B_\Omega & \vec C_\Omega & -\gamma\,\Nbulk & 0 \\
 \vec C^T_\Omega & 0 & 0 & 0 \\
 \NbulkT & 0 & 0 & -\frac1{\tau_m}\,\vec{N}_\Gamma^T \\
0 & 0 & \vec{N}_\Gamma & \vec{A}_\Gamma 
\end{pmatrix} 
\begin{pmatrix} \vec U^{m+1} \\ P^{m+1} \\ \kappa^{m+1} \\ 
\delta\vec{X}^{m+1} \end{pmatrix}
=
\begin{pmatrix} \vec g \\ 0 \\
0 \\ -\vec{A}_\Gamma\,\vec{X}^{m} \end{pmatrix} \,,
\label{eq:lin}
\end{equation}
where $(\vec U^{m+1},P^{m+1},\kappa^{m+1},\delta\vec{X}^{m+1})\in
(\R^d)^{K^m_\uspace}\times \R^{K^m_\pspace} \times
\R^{K^m_\Gamma}\,\times (\R^d)^{K^m_\Gamma}$
here denote the coefficients of these finite element functions with respect to
the standard bases of $\uspace^m$, $\pspace^m$, $\Wh$ and $\Vh$, respectively.
The definitions of the matrices and vectors in (\ref{eq:lin}) directly follow 
from (\ref{eq:HGa}--d),
but we state them here for completeness. 
Let $i,\,j = 1 ,\ldots, K_\uspace^m$,
$n,q = 1 ,\ldots, K_\pspace^m$ and $k,l = 1 ,\ldots, K^m_\Gamma$. Then
\begin{align}
& [\vec B_\Omega]_{ij} := 
\left( \tfrac{\rho^m+I^m_0\,\rho^{m-1}}{2\,\tau_m}\, 
\phi_j^{\uspace^m} , \phi_i^{\uspace^m} 
\right)\,\mat\Id
+ 2\left(\left(\mu^m\,\mat D(\phi_j^{\uspace^m} \vec \ek_r), 
\mat D( \phi_i^{\uspace^m}\,\vec \ek_s) \right) \right)_{r,s=1}^d 
\nonumber \\ & \qquad\qquad
+ \tfrac12\left(\rho^m, [(I^m_2\,\vec U^m\,.\,\nabla)\,\phi_j^{\uspace^m}
]\,\phi_i^{\uspace^m} - 
[(I^m_2\,\vec U^m\,.\,\nabla)\,\phi_i^{\uspace^m}]
\,\phi_j^{\uspace^m} \right)\,\mat\Id \,,
\nonumber \\
& [\vec C_\Omega]_{iq} := - \left( 
\left(\nabla\,.\,(\phi_i^{\uspace^m}\,\vec \ek_r),
\phi_q^{\pspace^m} \right) \right)_{r=1}^d,\qquad
[\Nbulk]_{il} :=  
\left\langle \phi_i^{\uspace^m}, \chi^m_l \,\vec\nu^m \right\rangle_{\Gamma^m}
\,, \nonumber \\
&  [\vec{N}_\Gamma]_{kl} := \left\langle 
\chi^m_l, \chi^m_k\,\vec\nu^m \right\rangle_{\Gamma^m}^h \,, \qquad
[\vec{A}_\Gamma]_{kl} := 
\left\langle \nabs\,\chi^m_l, \nabs\,\chi^m_k \right\rangle_{\Gamma^m}
\,\mat\Id \,, \nonumber \\
& \vec g_i = \left( \tfrac{I^m_0\,\rho^{m-1}}{\tau_m}\,I^m_2\,\vec U^m + 
\rho^m\,\vec f^{m+1}_1 + \vec f^{m+1}_2,
\phi_i^{\uspace^m}\right)\,;
\label{eq:mats}
\end{align}
\end{subequations}
where we recall that 
$\{\vec \ek_r\}_{r=1}^d$ denotes the standard basis in $\R^d$
and where we have used the convention that the subscripts in the matrix
notations refer to the test and trial domains, respectively. 
A single subscript is used where the two domains are the same.

For the the solution of (\ref{eq:lin})
we employ a Schur complement approach that eliminates
$(\kappa^{m+1}, \delta \vec X^{m+1})$ from (\ref{eq:lin}), and
then use an iterative solver for the remaining system in $(\vec U^{m+1},
P^{m+1})$. This approach has the advantage that for the reduced system
well-known solution methods for finite element discretizations for the 
standard Navier--Stokes equations may be employed. E.g.\ the authors in 
\cite{ElmanSW05}, for the reduced system matrix
{\scriptsize$\begin{pmatrix}
\vec B_\Omega & \vec C_\Omega \\
\vec C^T_\Omega & 0 \end{pmatrix}$}
recommend a GMRES iterative solver with the preconditioner
\begin{equation} \label{eq:ESW}
\mathcal{P} = \begin{pmatrix}
\vec{\mathcal{P}}_{\vec B} & \vec C_\Omega \\
0 & -\mathcal{P}_S
\end{pmatrix}
\end{equation}
where
$\vec{\mathcal{P}}_{\vec B}$ is some preconditioner for the matrix 
$\vec B_\Omega$, and $\mathcal{P}_S$ acts as a preconditioner for the Schur
complement operator $S=\vec C^T_\Omega \,\vec B_\Omega^{-1}\,\vec C_\Omega$.
An application of the preconditioner (\ref{eq:ESW}) amounts to
solving the equations
\begin{equation*} 
\begin{pmatrix}
\vec{\mathcal{P}}_{\vec B} & \vec C_\Omega \\
0 & -\mathcal{P}_S
\end{pmatrix}
\begin{pmatrix} \vec U \\ P \end{pmatrix}
= \begin{pmatrix} \vec v \\ q \end{pmatrix}
\iff
\mathcal{P}_S\,P = -q\,,\quad \vec{\mathcal{P}}_{\vec B}\,\vec U = 
\vec v - \vec C_\Omega\,P\,.
\end{equation*}
As the initial guess for the iterative solver the authors in \cite{ElmanSW05} 
recommend $(\vec B_\Omega^{-1}\,(\vec g - \vec C_\Omega\,P^{(0)}), P^{(0)})$,
or an approximation thereof, where $P^{(0)}$ is the initial guess for the
solution of the pressure approximation. 
It remains to discuss the choices of $\vec{\mathcal{P}}_{\vec B}$ and 
$\mathcal{P}_S$. The optimal choice for $\vec{\mathcal{P}}_{\vec B}$ is
$\vec{\mathcal{P}}_{\vec B} = \vec B_\Omega$. 
When that is not practical this can be
replaced with a suitable multigrid or Krylov solver approximation.
In the case 
$\rho_+ = \rho_- = 0$, a good choice for $\mathcal{P}_S$ 
is $\mathcal{P}_S = M_\mu$, where
\begin{equation*} 
[M_\mu]_{nq} = 
\left((\mu^m)^{-1}\,\phi_q^{\pspace^m}, \phi_n^{\pspace^m}\right)
\qquad n,q = 1 ,\ldots, K_\pspace^m\,;
\end{equation*}
see e.g.\ \cite{OlshanskiiR06}.
For the general Navier--Stokes equation, the following $BFBt$ preconditioner
works better. Here
\begin{subequations}
\begin{equation} \label{eq:BFBt}
\mathcal{P}_S^{-1} = (\vec C_\Omega^T\,\vec M_{u,1}^{-1}\,\vec C_\Omega)^{-1}
\vec C_\Omega^T\,\vec M_{u,1}^{-1}\,\vec B_\Omega\,\vec M_{u,1}^{-1}\,
\vec C_\Omega\,(\vec C_\Omega^T\,\vec M_{u,1}^{-1}\,\vec C_\Omega)^{-1}\,,
\end{equation}
where $\vec M_{u,1} = \diag(\vec M_{u})$ is the diagonal part of the mass 
matrix for the velocity space $\uspace^m$, and, for later purposes similarly 
$M_{p,1} = \diag(M_{p})$, i.e.\
\begin{equation*} 
[\vec M_u]_{ij} = 
\left(\phi_j^{\uspace^m}, \phi_i^{\uspace^m}\right) \mat\Id
\quad i,j = 1 ,\ldots, K_\uspace^m\,,\qquad
[M_p]_{nq} = 
\left(\phi_q^{\pspace^m}, \phi_n^{\pspace^m}\right)
\quad n,q = 1 ,\ldots, K_\pspace^m\,.
\end{equation*}
See e.g.\ \cite{ElmanSW05,GrossR11} for more details. 
Here we note that the rank deficiency of $\vec C_\Omega$ means that
e.g.\ $\vec C_\Omega^T\,\vec M_{u,1}^{-1}\,\vec C_\Omega$ is singular, with
kernel $\{ \underline{1} \}$, $\underline{1} = (1,\ldots,1)^T \in
\R^{K^m_\pspace}$. But restricted to the subspace $\{ \underline{1} \}^\perp$,
$\vec C_\Omega^T\,\vec M_{u,1}^{-1}\,\vec C_\Omega$ is a positive definite
matrix with range $\{ \underline{1} \}^\perp$. Computing
$(\vec C_\Omega^T\,\vec M_{u,1}^{-1}\,\vec C_\Omega)^{-1}$ with a
preconditioned conjugate gradient (pCG) solver can then be rigorously 
justified, e.g.\ for a right-oriented
preconditioning of GMRES with (\ref{eq:BFBt}). We refer to
\cite[\S 8.3.4]{ElmanSW05} for more details. However, in practice it turns out
that the inner pCG iteration for the computations of 
$(\vec C_\Omega^T\,\vec M_{u,1}^{-1}\,\vec C_\Omega)^{-1}$ in (\ref{eq:BFBt}) 
is much more stable if the projections $P_1 := \Id -
\frac{\underline{1}\,\underline{1}^T}{\underline{1}^T\,\underline{1}}$ are
included explicitly. In particular, in the case that $\vec C_\Omega$ does not
have full rank, we implement (\ref{eq:BFBt}) as
\begin{equation} \label{eq:BFBt_proj1}
\mathcal{P}_S^{-1} = 
P_1\,(P_1\,\vec C_\Omega^T\,\vec M_{u,1}^{-1}\,\vec C_\Omega\,P_1)^{-1}\,P_1\,
\vec C_\Omega^T\,\vec M_{u,1}^{-1}\,\vec B_\Omega\,\vec M_{u,1}^{-1}\,
\vec C_\Omega\,
P_1\,(P_1\,\vec C_\Omega^T\,\vec M_{u,1}^{-1}\,\vec C_\Omega\,P_1)^{-1}\,P_1\,.
\end{equation}
\end{subequations}

The desired Schur complement approach for eliminating 
$(\kappa^{m+1},\delta \vec X^{m+1})$ from (\ref{eq:lin}) can be obtained as
follows. Let 
\begin{equation*} 
\Xi_\Gamma:= \begin{pmatrix}
 0 & - \frac1{\tau_m}\,\vec{N}_\Gamma^T \\
\vec{N}_\Gamma & \vec{A}_\Gamma 
\end{pmatrix} \,.
\end{equation*}
It is a simple matter to adapt the argument in the proof of
Theorem~\ref{thm:stab} in order to show that the matrix
$\Xi_\Gamma$ is nonsingular. 
Then (\ref{eq:lin}) can be reduced to
\begin{subequations}
\begin{equation} \label{eq:SchurkX}
\begin{pmatrix}
\vec B_\Omega + \gamma\,(\Nbulk \ 0)\,\Xi_\Gamma^{-1}\,
\binom{\NbulkT}{0} & \vec C_\Omega \\
\vec C_\Omega^T & 0 
\end{pmatrix}
\begin{pmatrix}
\vec U^{m+1} \\ P^{m+1} 
\end{pmatrix}
= \begin{pmatrix}
\vec g
-\gamma\,(\Nbulk \ 0)\, \Xi_\Gamma^{-1}\,
\binom{0}{\vec{A}_\Gamma\,\vec{X}^{m}} \\
0
\end{pmatrix}
\end{equation}
and
\begin{equation}
\binom{\kappa^{m+1}}{\delta\vec{X}^{m+1}} = \Xi_\Gamma^{-1}\,
\binom{-\NbulkT\,\vec U^{m+1}}{-\vec{A}_\Gamma\,\vec{X}^{m}}\,.
\label{eq:SchurkXb}
\end{equation}
\end{subequations}
In practice we solve (\ref{eq:SchurkX}) with a preconditioned BiCGSTAB
iteration, with the preconditioner (\ref{eq:ESW}). For 
$\vec{\mathcal{P}}_{\vec B}$ we choose $\vec B_\Omega$ in the case $d=2$, 
and 20 SSOR iteration steps for $\vec B_\Omega$ in the case $d=3$,
and for $\mathcal{P}_{S}$ we use (\ref{eq:BFBt_proj1}). 
Here we replace $\vec B_\Omega$ in (\ref{eq:BFBt_proj1}) with 
$\vec B_\Omega + \gamma\,(\Nbulk \ 0)\,\Xi_\Gamma^{-1}\,\binom{\NbulkT}{0}$.

\subsection{Solution of the linear system for \XFEMGAMMA} \label{sec:41}
The linear system for the approximation from Theorem~\ref{thm:stabXFEM1} in
Section~\ref{sec:31} is given by (\ref{eq:lin},b) with $K^m_\pspace$ 
replaced by $K^m_\pspace+1$. However, in order to highlight the changes needed
for the implementation of our \XFEMGAMMA\ method, we use an alternative 
formulation here that builds on the matrix definitions as in (\ref{eq:lin},b).

The linear system for (\ref{eq:HGa}--d) with $\pspace^m$ replaced by
$\pspace^m_{\rm XFEM}$ can be formulated as:
Find $(\vec U^{m+1},P^{m+1},\lambda^{m+1},
\kappa^{m+1},\delta\vec{X}^{m+1})$, where $\vec X^{m+1} = \vec X^m + 
\delta\vec X^{m+1}$, such that
\begin{equation}
\begin{pmatrix}
 \vec B_\Omega & \vec C_\Omega & \vec D_\Omega &  -\gamma\,\Nbulk & 0 \\
 \vec C^T_\Omega & 0 & 0 & 0 & 0 \\
 \vec D^T_\Omega & 0 & 0 & 0 & 0 \\
 \NbulkT & 0 & 0 & 0 & -\frac1{\tau_m}\,\vec{N}_\Gamma^T \\
0 & 0 & 0 & \vec{N}_\Gamma & \vec{A}_\Gamma 
\end{pmatrix} 
\begin{pmatrix} \vec U^{m+1} \\ P^{m+1} \\ \lambda^{m+1} \\ \kappa^{m+1} \\ 
\delta\vec{X}^{m+1} \end{pmatrix}
=
\begin{pmatrix} \vec g \\ 0 \\ 0 \\
0 \\ -\vec{A}_\Gamma\,\vec{X}^{m} \end{pmatrix} \,,
\label{eq:linXFEM1}
\end{equation}
where $(\vec U^{m+1},P^{m+1},\lambda^{m+1},\kappa^{m+1},\delta\vec{X}^{m+1})\in
(\R^d)^{K^m_\uspace}\times \R^{K^m_\pspace} \times \R \times
\R^{K^m_\Gamma}\,\times (\R^d)^{K^m_\Gamma}$.
Here all the matrices are as defined in (\ref{eq:mats}), and the entries of
$\vec D_\Omega$, for $i = 1 ,\ldots, K_\uspace^m$, are given by
\begin{equation*} 
[\vec D_\Omega]_{i,1} := - \left\langle \phi_i^{\uspace^m}, 
\vec \nu^m \right\rangle_{\Gamma^m} 
\end{equation*}
As before, the system (\ref{eq:linXFEM1}) can 
be solved with a Schur complement formulation similarly to 
(\ref{eq:SchurkX},b), which is now given by
\begin{equation*}
\begin{pmatrix}
\vec B_\Omega + \gamma\,(\Nbulk \ 0)\,\Xi_\Gamma^{-1}\,
\binom{\NbulkT}{0} & \vec C_\Omega & \vec D_\Omega \\
\vec C_\Omega^T & 0 & 0 \\
\vec D_\Omega^T & 0 & 0
\end{pmatrix}
\begin{pmatrix}
\vec U^{m+1} \\ P^{m+1} \\ \lambda^{m+1}
\end{pmatrix}
= \begin{pmatrix}
\vec g
-\gamma\,(\Nbulk \ 0)\, \Xi_\Gamma^{-1}\,
\binom{0}{\vec{A}_\Gamma\,\vec{X}^{m}} \\
0 \\ 0
\end{pmatrix}
\end{equation*}
and (\ref{eq:SchurkXb}). 

\subsection{Assembly of interface-bulk cross terms} \label{sec:43}
We note that the assembly of the matrices arising from (\ref{eq:HGa}--d) is 
mostly standard. For the cross terms between bulk mesh and parametric mesh one 
needs to compute contributions of the form
$\left\langle \phi_i^{\uspace^m}, \chi^m_j \right\rangle_{\Gamma^m}$,
where $\{\phi_i^{\uspace^m}\}_{i=1}^{K_\uspace^m}$ 
and $\{\chi_j^m\}_{j=1}^{K_\Gamma^m}$ are the canonical
basis functions of $S^m_2$ and $\Wh$, respectively.
We recall that in \cite[\S4.5]{dendritic} the calculation of such 
contributions has been considered when $S^m_2$ is replaced by $S^m_1$. We
now extend these techniques to the space of piecewise quadratic functions
$S^m_2$.
Firstly, we need to compute the
intersections between bulk elements $\sigmaO^m_l$ and surface mesh elements
$\sigma^m_j$. For notational convenience, we will drop the subscripts $l$ and
$j$ in the remainder of this subsection.

In two space dimensions, i.e.\ $d = 2$, the intersection of a segment
$\sigma^m$ of the polygonal curve $\Gamma^m$ and a bulk mesh element 
$\sigmaO^m \in \mathcal{T}^m$ is always given by a segment, say 
$\sigmaO^m \cap \sigma^m = [\vec{q}_1, \vec{q}_2]$. Then the contribution over
$[\vec{q}_1, \vec{q}_2]$ for 
$\left\langle \phi_i^{\uspace^m}, \chi^m_j \right\rangle_{\Gamma^m}$ is
\begin{equation} \label{eq:num2dtrue}
\langle \phi_i^{\uspace^m}, \chi^m_j \rangle_{[\vec{q}_1, \vec{q}_2]}=
\tfrac16\,|\vec{q}_1 - \vec{q}_2|\,\sum_{k=0}^2 \omega_k\,
\phi_i^{\uspace^m}(\vec{q}_k)\,\chi^m_j(\vec{q}_k)\,,
\end{equation}
where $\vec{q}_0:=\frac12\,\sum_{k=1}^2 \vec{q}_k$ and $\omega_0=\frac23$, 
$\omega_1=\omega_2=\frac16$ from Simpson's rule.

The natural generalization of (\ref{eq:num2dtrue}) to $d=3$ is given as 
follows.
Here the intersection of a triangular element $\sigma^m$ of the polyhedral
surface $\Gamma^m$ with a bulk mesh element $\sigmaO^m$ is a convex $l$-polygon
$\mathcal{P}$, with $3\leq l \leq 7$. 
Some example intersections are given in Figure~\ref{fig:ts}, and an algorithm
to compute $\mathcal{P} = \sigmaO^m \cap \sigma^m$ is stated 
in \cite[p.\ 6284]{dendritic}. 
\begin{figure}
\center
\includegraphics[angle=-0,width=0.19\textwidth]{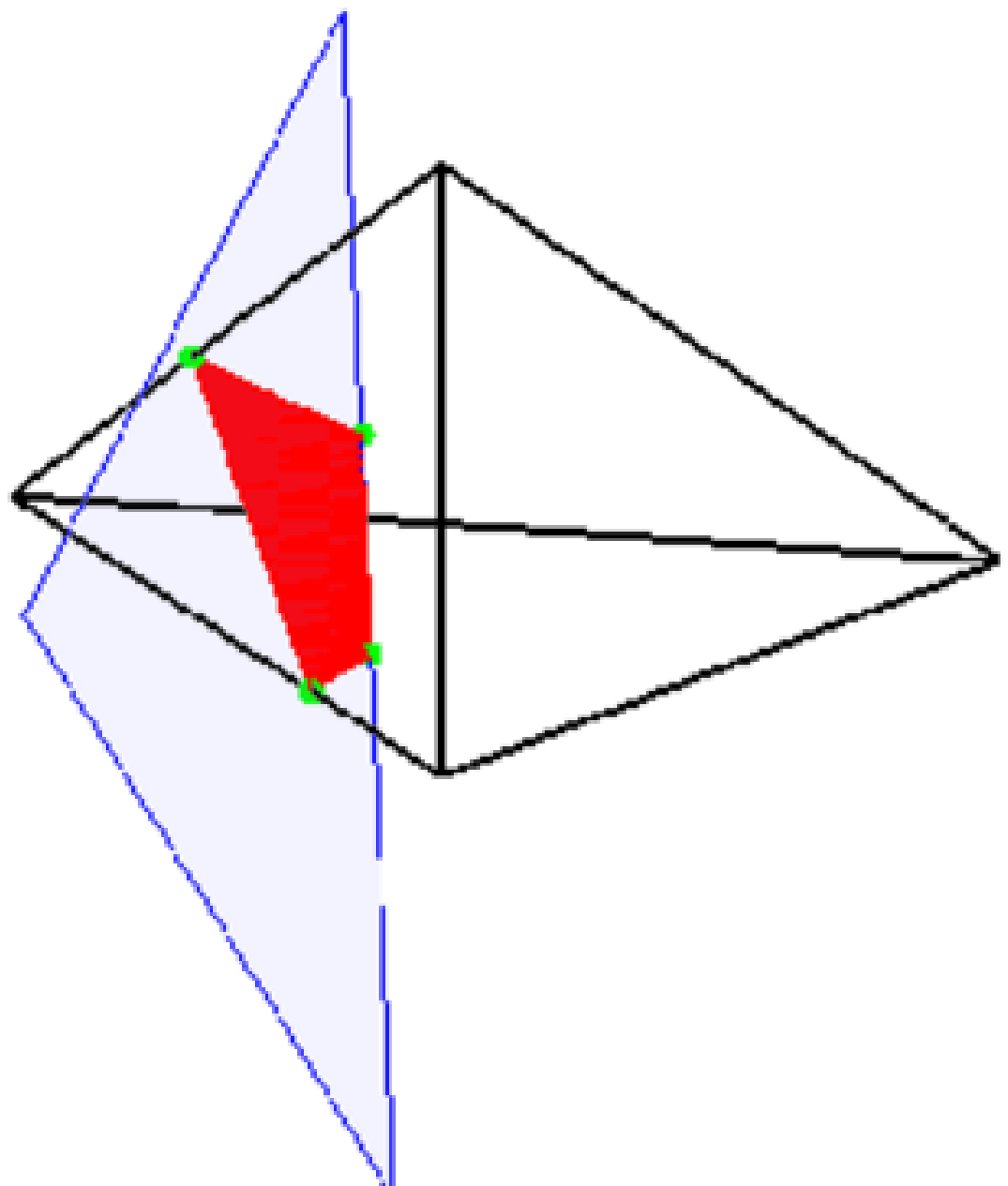}
\includegraphics[angle=-0,width=0.19\textwidth]{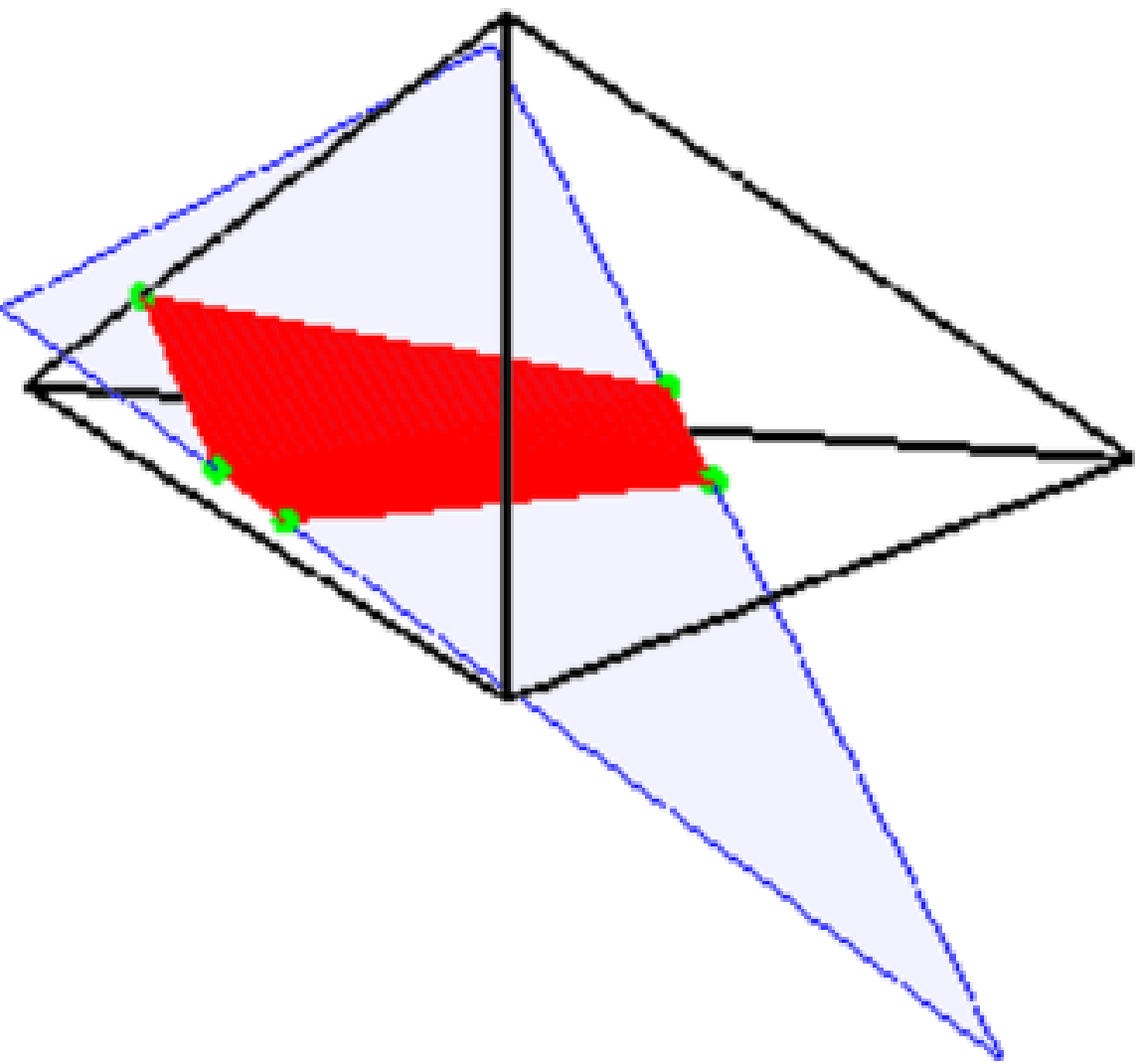}
\includegraphics[angle=-0,width=0.19\textwidth]{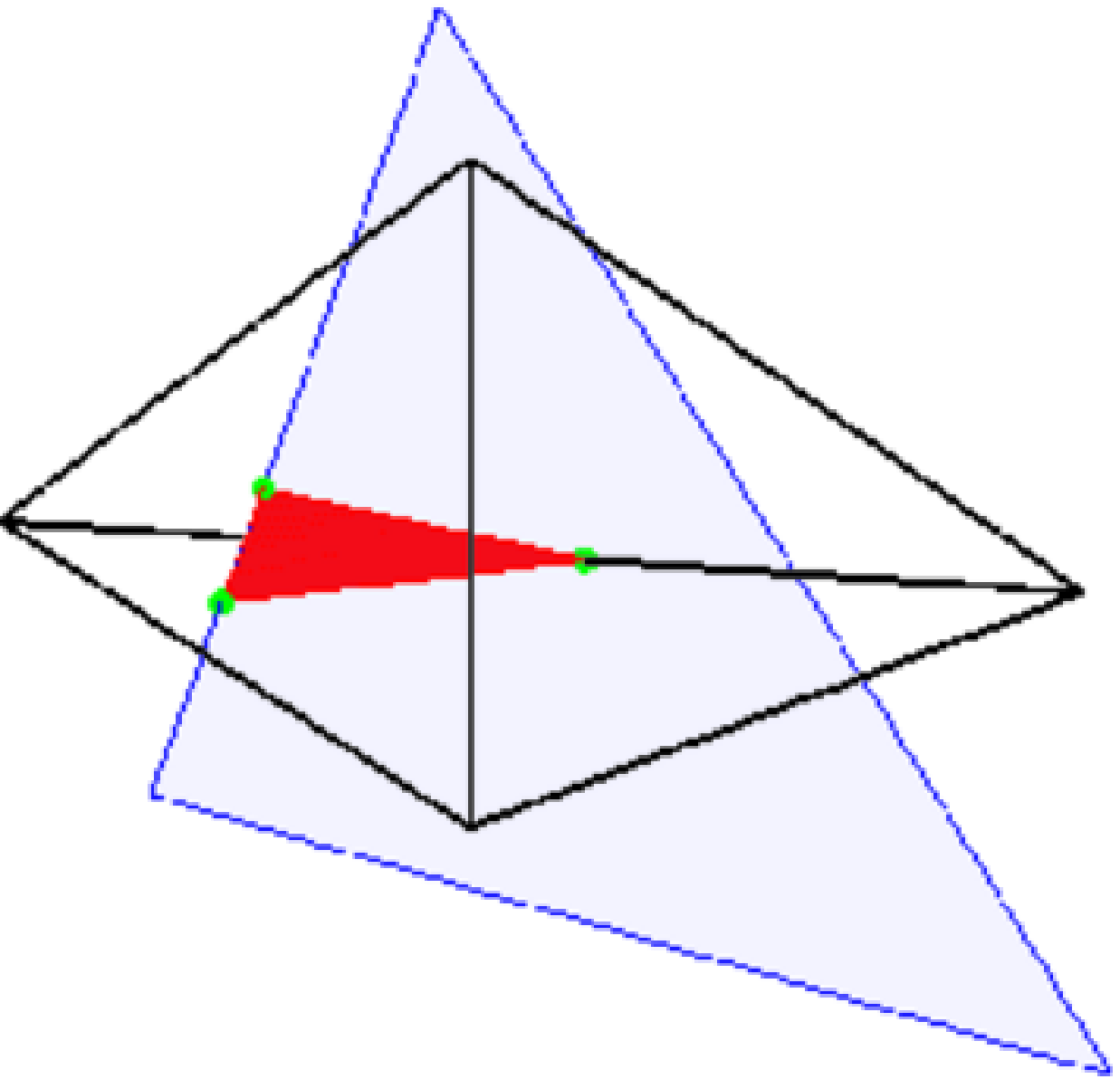}
\includegraphics[angle=-0,width=0.19\textwidth]{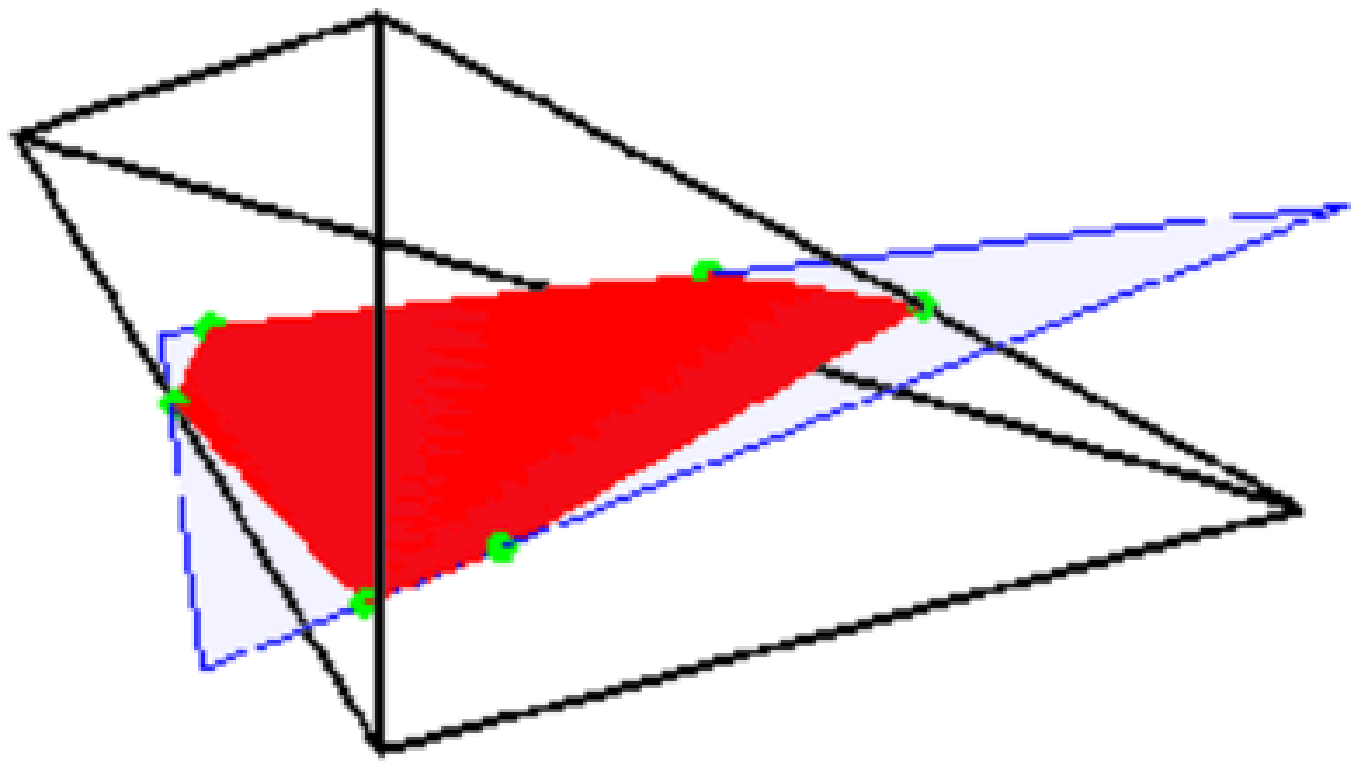}
\includegraphics[angle=-0,width=0.19\textwidth]{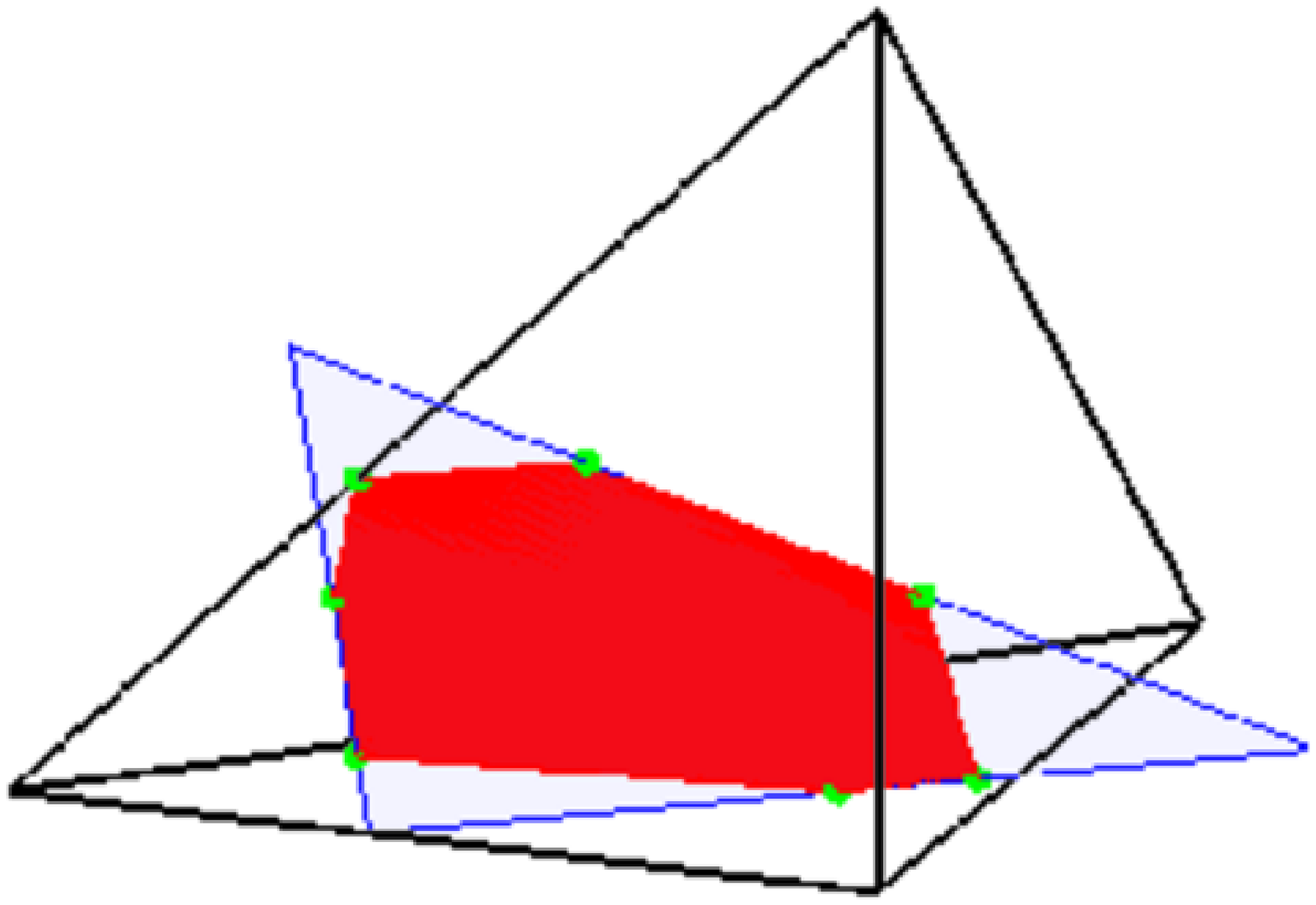}
\caption{Intersections of a triangle and a simplex in $\R^3$.}
\label{fig:ts}
\end{figure}%
Then the contribution over $\mathcal{P} \equiv \conv(\{\vec{q}_i\}_{i=1}^l)$ 
for $\left\langle \phi_i^{\uspace^m}, \chi^m_j \right\rangle_{\Gamma^m}$ can be
easy calculated by partitioning 
$\mathcal{P}$ into $l$ triangles with the help of
the centroid $\vec{q}_0:=\frac1l\,\sum_{k=1}^l \vec{q}_k$ of
$\mathcal{P}$, see Figure~\ref{fig:Psketch}. In particular, let
$\vec{p}_k:= \frac12\,[\vec{q}_0 + \vec{q}_k]$ and
$\vec{p}_{l+k} := \frac12\,[\vec{q}_k + \vec{q}_{k+1}]$ for $k=1,\ldots, l$, with
$\vec{q}_{l+1} := \vec{q}_1$, denote the edge midpoints of those triangles, and
let $\vec{c}_{k} := \frac13\,[\vec{q}_0 + \vec{q}_{k} + \vec{q}_{k+1}]$,
for $k=1,\ldots, l$, denote their barycentres. Then the contribution over 
$\mathcal{P}$ for 
$\left\langle \phi_i^{\uspace^m}, \chi^m_j \right\rangle_{\Gamma^m}$ is given
by
\begin{align} \label{eq:num3d3}
\left\langle \phi_i^{\uspace^m}, \chi^m_j \right\rangle_{\mathcal{P}}
& := \sum_{k=0}^l \omega_k^{\mathcal{P}}\,\phi_i^{\uspace^m}(\vec{q}_k)\,
\chi^m_j(\vec{q}_k)
+ \sum_{k=1}^{2\,l} \omega_{l+k}^{\mathcal{P}}\,\phi_i^{\uspace^m}(\vec{p}_k)\,
\chi^m_j(\vec{p}_k) \nonumber \\ &  \qquad
+ \sum_{k=1}^{l} \omega_{2\,l+k}^{\mathcal{P}}\,\phi_i^{\uspace^m}(\vec{c}_k)\,
\chi^m_j(\vec{c}_k)\,.
\end{align}
where the weights $\omega_k^{\mathcal{P}}$ in (\ref{eq:num3d3}) 
need to be defined such that the right hand side in (\ref{eq:num3d3}) 
is equal to $\int_{\mathcal{P}} \phi_i^{\uspace^m}\, \chi^m_j \dH{d-1}$.
Clearly here it suffices to find a quadrature rule that is exact for
cubics on $\mathcal{P}$. With the help of the above described partitioning of 
$\mathcal{P}$ into triangles this reduces to finding a quadrature rule that is
exact for cubics on triangles. Here we employ a quadrature rule
with sampling points at the vertices, at the edge midpoints and at the centroid
of the triangles. The weights for the sampling points are then given by
$\frac1{20}$, $\frac2{15}$ and $\frac9{20}$, respectively; see e.g.\
\cite{Stroud71}.

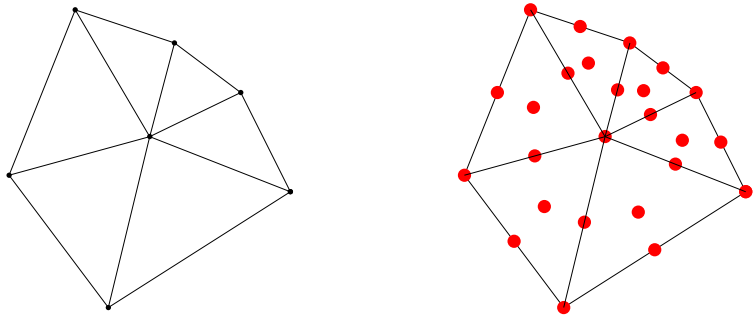
\begin{figure}
\center
\unitlength22mm
\psset{unit=\unitlength,linewidth=0.3pt}
\def\pointa{1,0}
\def\pointb{.7,.6}
\def\pointc{.3,.9}
\def\pointd{-.3,1.1}
\def\pointe{-.7,.1}
\def\pointf{-.1,-.7}
\def\pointq{.15,.33333333}
\def\pointab{.85,.3}
\def\pointbc{.5,.75}
\def\pointcd{0,1}
\def\pointde{-.5,.6}
\def\pointef{-.4,-.3}
\def\pointfa{.45,-.35}
\def\pointaq{.575,0.16666666666}
\def\pointbq{.425,.466666666666}
\def\pointcq{.225,.616666666666}
\def\pointdq{-.075,.716666666666}
\def\pointeq{-.275,.21666666666}
\def\pointfq{.025,-.183333333333}
\def\pointabq{0.616667, 0.311111}
\def\pointbcq{0.383333, 0.611111}
\def\pointcdq{0.05, 0.777778}
\def\pointdeq{-0.283333, 0.511111}
\def\pointefq{-0.216667, -0.0888889}
\def\pointfaq{0.35, -0.122222}
\begin{picture}(2,2) (-1,-0.7)
\psline (\pointa) (\pointb) (\pointc) (\pointd) (\pointe) (\pointf) (\pointa)
\psdots[linecolor=black,dotsize=2pt] 
(\pointa) (\pointb) (\pointc) (\pointd) (\pointe) (\pointf) (\pointa) (\pointq)
\psline (\pointq) (\pointa)
\psline (\pointq) (\pointb)
\psline (\pointq) (\pointc)
\psline (\pointq) (\pointd)
\psline (\pointq) (\pointe)
\psline (\pointq) (\pointf)
\end{picture}%
\qquad
\qquad
\begin{picture}(2,2) (-1,-0.7)
\psline (\pointa) (\pointb) (\pointc) (\pointd) (\pointe) (\pointf) (\pointa)
\psdots[linecolor=red,dotsize=5pt] 
(\pointa) (\pointb) (\pointc) (\pointd) (\pointe) (\pointf) (\pointa) (\pointq)
(\pointab) (\pointbc) (\pointcd) (\pointde) (\pointef) (\pointfa) 
(\pointaq) (\pointbq) (\pointcq) (\pointdq) (\pointeq) (\pointfq)
(\pointabq) (\pointbcq) (\pointcdq) (\pointdeq) (\pointefq) (\pointfaq) 
\psline (\pointq) (\pointa)
\psline (\pointq) (\pointb)
\psline (\pointq) (\pointc)
\psline (\pointq) (\pointd)
\psline (\pointq) (\pointe)
\psline (\pointq) (\pointf)
\end{picture}%
\caption{Sketch of the partitioning of $\mathcal{P}$ into triangles (left), 
and the sampling points for (\ref{eq:num3d3}) (right).}
\label{fig:Psketch}
\end{figure}

For the definitions of $\rho^m$ and $\mu^m$ in (\ref{eq:rhoma},b) the
disjoint partition $\mathcal{T}^m = \mathcal{T}^m_- \cup \mathcal{T}^m_+ \cup
\mathcal{T}^m_{\Gamma^m}$ of bulk elements is needed, and this can easily be
found with e.g.\ Algorithm~4.1 in \cite{crystal}.  
For the strategy (\ref{eq:rhomc}) we need in addition a procedure to compute
$\mathcal{L}^d(o^m \cap \Omega^m_-)$ for all elements  
$o^m \in \mathcal{T}^m_{\Gamma^m}$. 
Let $V:= o^m \cap \Omega^m_-$. Then the divergence theorem yields that
\begin{equation} \label{eq:volV}
\mathcal{L}^d(V) = \int_V 1 \dL{d} = 
\tfrac1d\,\int_{\partial V} (\vec\id - \vec z_0)\,.\,\vec\nu_V \dH{d-1} \,,
\end{equation}
where $\vec\id$ is the identity function on $\R^d$, 
$\vec z_0 \in \R^d$ is an arbitrarily fixed point, and where
$\vec\nu_V$ denotes the outer normal to $V$. 
Here we note that $\partial V$ is a union of flat facets with
$\vec\nu_V = \vec\nu^m$ on $o^m \cap \Gamma^m$ and
$\vec\nu_V = \vec\nu_{o^m}$, the outer normal of $o^m$, on 
$\partial o^m \cap \Omega^m_-$.
Hence the integral in
(\ref{eq:volV}) simplifies on noting that $\vec\id\,.\,\vec\nu_V$ is 
constant on each facet, and vanishes on each facet that contains $\vec z_0$.
Now assume that $o^m$ has an edge/face that is not intersected by 
$\Gamma^m$. Any such element we call {\em regularly cut}, and for the element
$o^m$ at hand we denote by $F$ the edge/face that is not cut by the interface. 
Now let $\vec z_0$ be the vertex opposite $F$.
Then it follows from (\ref{eq:volV}) that
\begin{subequations}
\begin{align} 
\mathcal{L}^d(V) & = 
\tfrac1d\,\int_{o^m \cap \Gamma^m} (\vec\id - \vec z_0  )\,.\,\vec\nu^m
\dH{d-1} \qquad\qquad \text{if $\vec z_0 \in {\Omega^m_-}$.} \label{eq:vol1} \\
\intertext{Similarly, it holds that}
\mathcal{L}^d(V) & = 
\tfrac1d\,\int_{F} (\vec\id - \vec z_0  )\,.\,\vec\nu_{o^m} \dH{d-1} 
+ \tfrac1d\,\int_{o^m \cap \Gamma^m} (\vec\id - \vec z_0  )\,.\,\vec\nu^m
\dH{d-1} \nonumber \\ & = 
\mathcal{L}^d(o^m) + 
\tfrac1d\,\int_{o^m \cap \Gamma^m} (\vec\id - \vec z_0  )\,.\,\vec\nu^m
\dH{d-1} \qquad\qquad \text{if $\vec z_0 \in {\Omega^m_+}$.} \label{eq:vol2}
\end{align}
\end{subequations}
Hence it follows from (\ref{eq:vol1},b) that for regularly cut elements we can
calculate $\mathcal{L}^d(V) = \mathcal{L}^d(o^m \cap \Omega^m_-)$ 
if we can decide for each vertex
of the bulk mesh whether it belongs to ${\Omega^m_-}$ or to
${\Omega^m_+}$. In practice this information can, for example, be obtained 
with the algorithm presented in \cite[Algorithm~4.2]{crystal}. 
On setting
\begin{equation*} 
\widetilde{\mathcal{L}}^d(o^m \cap \Omega^m_-) = \begin{cases} 
{\mathcal{L}}^d(o^m \cap \Omega^m_-) & \text{if $o^m$ is regularly cut}\,, \\
\frac12\,\mathcal{L}^d(o^m) & \text{else}\,,
\end{cases}
\end{equation*}
it holds that
\begin{equation} \label{eq:Lhat}
\widehat{\mathcal{L}}^d(\Omega^m_-) := 
\sum_{o \in \mathcal{T}^m_-} \mathcal{L}^d(o) + 
\sum_{o \in \mathcal{T}^m_{\Gamma^m}}
\widetilde{\mathcal{L}}^d(o \cap \Omega^m_-)
\end{equation}
is an approximation to $\mathcal{L}^d(\Omega^m_-)$ that is exact if all the
elements in $\mathcal{T}^m_{\Gamma^m}$ are regularly cut.
It remains to compute 
${\mathcal{L}}^d(o^m \cap \Omega^m_-)$ for all the elements in 
$\mathcal{T}^m_{\Gamma^m}$ that are not regularly cut. Let $o^m$ be such an
element, and assume that $o^m$ itself is partitioned into smaller elements.
Then it is straightforward to extend the definition (\ref{eq:Lhat}) to this
partitioning of $o^m$ to yield a definition for 
$\widehat{\mathcal{L}}^d(o^m \cap \Omega^m_-)$, where an analogue of the
decomposition $\mathcal{T}^m = \mathcal{T}^m_- \cup \mathcal{T}^m_+ \cup
\mathcal{T}^m_{\Gamma^m}$ needs to be defined for the local partitioning of
$o^m$.

Hence in order to compute ${\mathcal{L}}^d(o^m \cap \Omega^m_-)$ for a not
regularly cut element it is sufficient to locally refine it until
$\widehat{\mathcal{L}}^d(o^m \cap \Omega^m_-) = 
{\mathcal{L}}^d(o^m \cap \Omega^m_-)$. In practice we use an iterative
bisectioning procedure, where we stress that the refined partitionings are
only created for the purpose of computing the integral in (\ref{eq:volV}). In
particular, the refinements do not affect the approximation spaces
$\uspace^m$ and $\pspace^m$.
In order to avoid excessive refinement we stop the iterative bisectioning
procedure of the not regularly cut elements whenever
\begin{equation*}
\left| \mathcal{L}^d(\Omega^m_-) - \sum_{o^m \in \mathcal{T}^m_-}
\mathcal{L}^d(o^m) - \sum_{o^m \in \mathcal{T}^m_{\Gamma^m}}
\widehat{\mathcal{L}}^d(o^m \cap \Omega^m_-) \right| < {\rm tol}_V\,,
\end{equation*}
where ${\rm tol}_V$ is a small tolerance,
and where we note that 
$\mathcal{L}^d(\Omega^m_-) = \tfrac1d\,\int_{\Gamma^m} \vec X^m \,.\,\vec\nu^m
\dH{d-1}$ is known exactly. 
In practice we always use ${\rm tol}_V = 10^{-8}$.

Finally, we mention that determining if $\vec z_0 \in \Omega^m_\pm$ 
in (\ref{eq:vol1},b) in practice
is not very efficient. A better, and more robust, strategy makes use of the
fact that $\vec z_0 \in \Omega^m_\pm$ in (\ref{eq:vol1},b) is equivalent to the
integral in (\ref{eq:vol1}) 
being negative/positive, i.e.\
$$
\vec z_0 \in \Omega^m_\pm \quad \iff 
\quad \mp \int_{o^m \cap \Gamma^m} (\vec\id - \vec z_0)
\,.\,\vec\nu^m \dH{d-1} > 0\,.
$$
In practice it remains to robustly deal with the case that 
$\left|\int_{o^m \cap \Gamma^m} (\vec\id - \vec z_0)
\,.\,\vec\nu^m \dH{d-1}\right|$ is very small,
which means that numerical noise may influence the sign of the integral. 
But in the vast majority of the cases, the smallness (in magnitude) of the
integral will be caused by $\mathcal{H}^{d-1}(o^m \cap \Gamma^m)$ being small,
because $F$, the edge/face opposite $\vec z_0$, 
is not cut by $\Gamma^m$. Then the
sign of the integral can in general be robustly detected by inspecting the sign
of the (piecewise constant) integrand $(\vec\id - \vec z_0) \,.\,\vec\nu^m$.
In all instances where the sign of 
$\int_{o^m \cap \Gamma^m} (\vec\id - \vec z_0) \,.\,\vec\nu^m \dH{d-1}$ cannot
be robustly ascertained in practice, it is prudent to treat the element $o^m$
as if it is a not regularly cut element, and to proceed as outlined above. This
is the strategy that we use in all our computations for (\ref{eq:rhomc}). 
This works well for $d=2$, mainly because not regularly cut elements are very
rarely encountered. Unfortunately, this is different for $d=3$, where
not regularly cut elements are far more generic. On recalling 
Figure~\ref{fig:ts}, and especially the last two examples there, this may have
to do with the fact that for $d=3$ it is possible for a single element
$\sigma^m$ of $\Gamma^m$ to intersect all $d+1$ faces of a bulk element $o^m$,
something that is not possible for $d=2$. Unfortunately, this means that
(\ref{eq:rhomc}) at this stage is not practical for $d=3$.

\setcounter{equation}{0}
\section{Mesh adaptation} \label{sec:5}
We implemented our finite element approximation (\ref{eq:HGa}--d) within 
the framework of the finite element toolbox ALBERTA, see \cite{Alberta}. In
what follows we describe the mesh refinement strategies used for both bulk and
interface mesh. These are similar to the approach described in 
\cite{dendritic}.

\subsection{Bulk mesh adaptation} \label{sec:51}
Given a polyhedral approximation $\Gamma^m$, $m\geq0$, of the interface, 
we employ the following mesh adaptation strategy for the bulk mesh
triangulation $\mathcal{T}^m$. The strategy is inspired by a similar
refinement algorithm proposed in \cite{dendritic}, and it 
results in a fine mesh around 
$\Gamma^m$ and a coarse mesh further away from it.

In particular, given two integer parameters $N_f >  N_c$, we set 
$h_{f} = \frac{2\,H}{N_{f}}$, $h_{c} =  \frac{2\,H}{N_{c}}$, where for
simplicity we assume that $\Omega = \times_{i=1}^d (L_i,U_i)$
with $H = \frac12\,\min_{i=1,\ldots, d} (U_i - L_i)$. Then we set
\begin{equation} \label{eq:sec51}
 vol_{f} = \frac{h_{f}^{d}}{d!} \quad \mbox{and} \quad 
vol_{c} = \frac{h_{c}^{d}}{d!}\,,
\end{equation}
that is, for $d=3$, 
$vol_f$ denotes the volume of a tetrahedron with three right-angled 
and isosceles faces with side length $h_{f}$, while for $d=2$ it denotes the
area of a right-angled and isosceles triangle with side length $h_f$, and
similarly for $vol_c$.

Now starting with the triangulation $\mathcal{T}^{m-1}$ from the previous time
step, where here for convenience we define $\mathcal{T}^{-1}$ to be a uniform
partitioning of mesh size $h_c$, we obtain $\mathcal{T}^m$ as follows. First
any element $\sigmaO^{m-1} \in \mathcal{T}^{m-1}$ satisfying
$\mathcal{L}^d(\sigmaO^{m-1}) \geq 2\,vol_f$ and 
$\sigmaO^{m-1} \cap \Gamma^m \not= \emptyset$
is marked for refinement. In addition, any element satisfying
$\mathcal{L}^d(\sigmaO^{m-1}) \geq 2\,vol_f$, 
for which a direct neighbour intersects
$\Gamma^m$, is also marked for refinement.
Similarly, an element that is not marked for refinement is marked for 
coarsening if it satisfies 
$\mathcal{L}^d(\sigmaO^{m-1}) \leq \frac{1}{2}\,vol_{c}$ and 
$\sigmaO^{m-1} \cap \Gamma^m = \emptyset$.
Now all the elements marked for refinement are halved into two smaller
elements with the help of a simple bisectioning procedure, see
\cite{Alberta} for details. In order to avoid hanging
nodes, this will in general lead to refinements of elements that were not
originally marked for refinement. Similarly, an element that is marked for 
coarsening is coarsened only if all of its 
neighbouring elements are marked for coarsening as well. For more details on
the refining and coarsening itself we refer to \cite{Alberta}. 

This marking and refinement process is repeated until no more
elements are required to be refined or coarsened.
Thus we obtain the triangulation $\mathcal{T}^{m}$ on which, together with 
$\Gamma^m$, the new
solutions $(\vec U^{m+1}, P^{m+1}, \kappa^{m+1}, \vec X^{m+1})$ 
will be computed. 
In practice only at the
first time step, $m=0$, more than one of the described refinement cycles 
are needed.

\subsection{Parametric mesh adaptation} \label{sec:52}
As mentioned before, the equation (\ref{eq:HGd}) means that the vertices of the
parametric approximation $\Gamma^m$ are in general very well distributed, so
that mesh smoothing (redistribution) is not necessary in practice. 
Similarly, an adaptation 
of the parametric mesh is in general not necessary. However, in simulations
where the total surface area $\mathcal{H}^{d-1}(\Gamma^m)$ 
increases significantly over time,
it is beneficial to locally refine the triangulation where elements have become
too large. 

The mesh refinement strategy can be described as follows, where we assume that
an arbitrary polyhedral approximation $\Gamma^0$ of $\Gamma(0)$ is given. Let
\begin{equation*}
vol_{max} := \max_{j = 1 ,\ldots, J^0_\Gamma} \mathcal{H}^{d-1}(\sigma^0_\Gamma)\,.
\end{equation*}
Then for an arbitrary $m \geq 0$, given $\Gamma^m$ and the solution
$(\vec U^{m+1}, P^{m+1}, \kappa^{m+1}, \vec X^{m+1})$ to (\ref{eq:HGa}--d), 
we define $\Gamma^{m+1,\star} := \vec{X}^{m+1}(\Gamma^m)$. Clearly,
$\Gamma^{m+1,\star}=\bigcup_{j=1}^{J^m_\Gamma} 
\overline{\sigma^{m+1,\star}_j}$,
where $\sigma^{m+1,\star}_j:= \vec{X}^{m+1}(\sigma^m_j)$, $j=1,\ldots, J^m_\Gamma$.
We will now define a finer triangulation $\bigcup_{j=1}^{J^{m+1}_\Gamma} 
\overline{\sigma^{m+1}_j}$, with $J^{m+1}_\Gamma \geq J^{m}_\Gamma$, for the
same polyhedral surface $\Gamma^{m+1,\star} = \Gamma^{m+1}$.
To this end, we mark all elements $\sigma^{m+1,\star}_j$, that have become too
large due to the growth of the interface, for refinement. In particular, any
element with 
$\mathcal{H}^{d-1}(\sigma^{m+1,\star}_j) \geq \frac74\,vol_{max}$ is marked for
refinement. Then all refined elements are replaced with two smaller ones with
the help of a simple bisectioning procedure. 
Note that this bisection does not change the
polyhedral surface $\Gamma^{m+1,\star} = \Gamma^{m+1}$. Moreover, we note that
in order to prevent hanging nodes, in general more elements will be refined
than have initially been marked for refinement. The cycle of marking and
refining is repeated until no more refinements are required. In practice,
this was always the case after just one such refinement step.

In conclusion we stress that the given parametric mesh adaptation algorithm 
means that
Theorem~\ref{thm:stabstab} still holds. Moreover, apart from this simple mesh
refinement, no other changes were performed on the parametric mesh in any of 
our simulations. 
In particular, no mesh smoothing (redistribution) was required.

\setcounter{equation}{0}
\section{Numerical results}  \label{sec:6}
Throughout this section we use uniform time steps
$\tau_m=\tau$, $m=0,\ldots, M-1$.
In addition, we set $\vec U^0 = \vec 0$, $\beta= 0$ and, unless otherwise
stated, we employ (\ref{eq:rhoma}) and \XFEMGAMMA.

We will often present detailed discretization parameters and CPU times for our 
simulations. Here the CPU times, which we report in seconds, correspond to a 
single-thread computation on an
Intel Xeon E5-2643 (3.3 GHz) processor with 16 GB of main memory. To summarize
the discretization parameters we use the shorthand notation
$n\,{\rm adapt}_{k,l}^{(\star)}$, where the superscript
$(\star) \in \{ (1), (0), (1,0) \}$ indicates which of the elements 
(\ref{eq:P2P1}--c) is employed.
The subscripts refer to the fineness of the spatial discretizations, i.e.\
for the set
$n\,{\rm adapt}_{k, l}^{(\cdot)}$ it holds that 
$N_f = 2^k$ and $N_c = 2^l$, recall
(\ref{eq:sec51}). For the case $d=2$ we have in addition that 
$K^0_\Gamma = J^0_\Gamma = 2^k$, while for $d=3$ it holds that
$(K^0_\Gamma, J^0_\Gamma) = (770, 1536), (1538, 3072), (3074, 6144)$ for 
$k = 5,6,7$. 
Finally, the uniform time step size 
for the set $n\,{\rm adapt}_{k,l}^{(\cdot)}$ is given by $\tau = 10^{-3} / n$,
and if $n=1$ we write ${\rm adapt}_{k, l}^{(\cdot)}$.

\subsection{Numerical results in 2d} \label{sec:61}
In this subsection we present numerical results for our approximation
(\ref{eq:HGa}--d) for the case $d = 2$. In particular, we will present
benchmark computations for the two test cases proposed in
\cite[Table~I]{HysingTKPBGT09}.
To this end, we define the following benchmark
quantities for the continuous solution $(\vec u, p, \Gamma)$ of 
(\ref{eq:2a}--h). 
Let $y_c(t) = \int_{\Omega_-(t)} x_2 \dL2 / \mathcal{L}^2(\Omega_-(t))$ denote
the $x_2$-component of the bubble's centre of mass. Let $\strikec(t)$ denote 
the ``degree of circularity'' of $\Gamma(t)$, which is defined as the ratio of
the perimeter of an area-equivalent circle and $\mathcal{H}^{1}(\Gamma(t))$. 
Finally, let 
$V_c(t) = \int_{\Omega_-(t)} u_2(t) \dL2 / \mathcal{L}^2(\Omega_-(t))$ 
denote the bubble's rise velocity, where 
$\vec u(\cdot,t) = 
(u_1(\cdot,t),u_2(\cdot,t))^T$. 
In this paper, we use the following discrete approximations of
these benchmark quantities:
\begin{equation} \label{eq:benchmarkm}
y_c^m = \frac1{\mathcal{L}^2(\Omega_-^m)}\,\int_{\Omega_-^m} x_2 \dL2\,,
\quad 
\strikec^m = 2\,[\pi\,\mathcal{L}^2(\Omega_-^m)]^\frac12\,
[\mathcal{H}^{1}(\Gamma^m)]^{-1}\,, \quad 
V^m_c = \frac{(\rho^m_-\,U^m_2, 1)}{(\rho^m_-,1)}\,,
\end{equation}
where $\rho^m_-\in S^m_0$ is defined as in (\ref{eq:rhoma},b) 
but with $\rho_+$ replaced by zero.
Finally, we also define the relative overall area/volume loss as
$$
\Mloss = 
\frac{\mathcal{L}^d(\Omega^0_-) - \mathcal{L}^d(\Omega^M_-)}
{\mathcal{L}^d(\Omega^0_-)}\,,
$$
and as a measure of the mesh quality we introduce the element ratios
\begin{equation}
r^m := \frac{\max_{j=1,\ldots, J^m_\Gamma} \mathcal{L}^d(\sigma^m_j)}
{\min_{j=1,\ldots, J^m_\Gamma} \mathcal{L}^d(\sigma^m_j)}\,,
\qquad m = 0,\ldots,M\,.
\label{eq:r}
\end{equation}

\subsubsection{2d benchmark problem 1} \label{sec:613}
We use the setup described in \cite{HysingTKPBGT09}, see Figure~2 there;
i.e.\ $\Omega = (0,1) \times (0,2)$ with 
$\partial_1\Omega = [0,1] \times \{0,2\}$ and 
$\partial_2\Omega = \{0,1\} \times (0,2)$.
Moreover, $\Gamma_0 = \{\vec z \in \R^2 : |\vec z - (\frac12, \frac12)^T| =
\frac14\}$.
The physical parameters from the test case 1 in \cite[Table~I]{HysingTKPBGT09} 
are then given by
\begin{equation} \label{eq:Hysing1}
\rho_+ = 1000\,,\quad \rho_- = 100\,,\quad \mu_+ = 10\,,\quad \mu_- = 1\,,\quad
\gamma = 24.5\,,\quad \vec f_1 = -0.98\,\vec\ek_d\,,\quad \vec f_2 = \vec 0\,.
\end{equation}
The time interval chosen for the simulation is $[0,T]$ with $T=3$.

Some discretization parameters and CPU times for our approximation
(\ref{eq:HGa}--d) are shown in Table~\ref{tab:CPU}. Here and throughout 
the CPU times correspond to
computations for the simple strategy (\ref{eq:rhoma}), but the times for the 
more involved choice (\ref{eq:rhomc}) are very similar.
\begin{table}
\center
\begin{tabular}{l|r|r|r|r|r|r}
\hline
& $K^0_\Gamma$ & $J^0_\Omega$ & NDOF$_{\rm bulk}$ & $M$ & CPU &CPU \XFEMGAMMA\\
\hline 
adapt$_{5,2}^{(1)}$  &   32 &   536 &   2475 & 3000 &   78  &    94 \\
adapt$_{7,3}^{(1)}$  &  128 &  2320 &  10563 & 3000 &  610  &   790 \\
2\,adapt$_{9,4}^{(1)}$ &  512 &  9728 &  44019 & 6000 & 8400  &  20640 \\
5\,adapt$_{11,5}^{(1)}$ & 2048 & 39328 & 177459 & 15000 & 138100 & 327700 \\
\hline
adapt$_{5,2}^{(1,0)}$  &   32 &   536 &   3011 & 3000 &  229   &  401 \\
adapt$_{7,3}^{(1,0)}$  &  128 &  2320 &  12883 & 3000 & 2590   & 13900 \\
2\,adapt$_{9,4}^{(1,0)}$ &  512 &  9728 &  53747 & 6000 & 40900 & 111800 \\ 
\hline
\end{tabular}
\caption{Simulation statistics and timings for the test case 1 in 
\cite{HysingTKPBGT09}.}
\label{tab:CPU}
\end{table}%
Some quantitative values for computations with the P2-P1 element 
(\ref{eq:P2P1}) 
are given in Table~\ref{tab:dataa}. Here the observed relative area losses 
for the runs without \XFEMGAMMA\ were $32.1\%$, $8.2\%$, $2.1\%$ and $0.5\%$;
and so we do not present the remaining statistics
for these runs. Here we recall from \S\ref{sec:31} that for the semidiscrete
continuous-in-time variant of (\ref{eq:HGa}--d) with \XFEMGAMMA\ true volume
conservation (area conservation in 2d) holds. 
In general, we also observe excellent volume conservation for
the fully discrete scheme.
Similarly, in Table~\ref{tab:data10a} we present 
quantitative values for computations with the P2-(P1+P0) element 
(\ref{eq:P2P10}). Once
again we omit the results for the runs without \XFEMGAMMA, 
for which the observed relative area losses were 
$7.2\%$, $1.9\%$ and $0.5\%$. 
We observe that the results in Tables~\ref{tab:dataa} and \ref{tab:data10a}
are in very good agreement with the corresponding numbers from the finest
discretization run of group 3 in \cite{HysingTKPBGT09}, which are given by
0.9013, 1.9000, 0.2417, 0.9239 and 1.0817. Here we note that of the three
groups in \cite{HysingTKPBGT09}, group 3 shows the most accurate and the most
consistent results for the test case 1. Their method is based on the ALE
approach with a piecewise quadratic velocity space enriched with cubic bubble
functions, with a discontinuous piecewise linear pressure space and with a
second order, fractional step $\theta$-scheme in time.
We stress that our simulation results appear to be in better agreement with 
the results from group 3 than other recently published results on the same 
benchmark problem, see e.g.\ \cite[Tables~III-V]{AlandV12}, 
\cite[Table~XVI]{ZahediKK12} and \cite[Table~VII]{ChoCCY12}.
\begin{table}
\center
\begin{tabular}{l|r|r|r|r}
\hline
& adapt$_{5,2}^{(1)}$ & adapt$_{7,3}^{(1)}$ & 
2\,adapt$_{9,4}^{(1)}$ & 5\,adapt$_{11,5}^{(1)}$\\
\hline 
$\Mloss$ & 
  0.0\% &  0.0\% &  0.0\% &  0.0\% \\
$\strikec_{\min}$ & 
 0.9136 & 0.9068 & 0.9034 & 0.9019 \\
$t_{\strikec = \strikec_{\min}}$ & 
 2.0760 & 1.9430 & 1.9110 & 1.9028 \\
$V_{c,\max}$ & 
 0.2478 & 0.2415 & 0.2414 & 0.2416 \\
$t_{V_c = V_{c,\max}}$ & 
 0.9470 & 0.9360 & 0.9255 & 0.9200 \\
$y_c(t=3)$ & 
 1.0906 & 1.0823 & 1.0815 & 1.0817 \\
\hline
\end{tabular}
\begin{tabular}{l|r|r|r|r}
\hline
& adapt$_{5,2}^{(1)}$ & adapt$_{7,3}^{(1)}$ & 
2\,adapt$_{9,4}^{(1)}$ & 5\,adapt$_{11,5}^{(1)}$\\ 
\hline 
$\Mloss$ & 
  0.0\% &  0.0\% &  0.0\% &  0.0\% \\
$\strikec_{\min}$ & 
 0.9061 & 0.9034 & 0.9018 & 0.9014 \\
$t_{\strikec = \strikec_{\min}}$ & 
 1.9260 & 1.9040 & 1.8990 & 1.8988 \\
$V_{c,\max}$ & 
 0.2430 & 0.2422 & 0.2418 & 0.2417 \\
$t_{V_c = V_{c,\max}}$ & 
 0.9210 & 0.9270 & 0.9250 & 0.9210 \\
$y_c(t=3)$ & 
 1.0894 & 1.0832 & 1.0821 & 1.0818 \\
\hline
\end{tabular}
\caption{Some quantitative results for the test case 1 in 
\cite{HysingTKPBGT09}. Here we use the P2-P1 element (\ref{eq:P2P1}) with 
(\ref{eq:rhoma}) and \XFEMGAMMA. 
The bottom table is for (\ref{eq:rhomc}) and \XFEMGAMMA.}
\label{tab:dataa}
\end{table}%
\begin{table}
\center
\begin{tabular}{l|r|r|r}
\hline
& adapt$_{5,2}^{(1,0)}$ & adapt$_{7,3}^{(1,0)}$ & 
2\,adapt$_{9,4}^{(1,0)}$ \\
\hline 
$\Mloss$ & 
  0.0\% &  0.0\% &  0.0\% \\
$\strikec_{\min}$ &
 0.9055 & 0.9044 & 0.9027 \\
$t_{\strikec = \strikec_{\min}}$ & 
 2.0950 & 1.9470 & 1.9105 \\
$V_{c,\max}$ & 
 0.2483 & 0.2414 & 0.2413 \\
$t_{V_c = V_{c,\max}}$ & 
 0.9490 & 0.9480 & 0.9255 \\
$y_c(t=3)$ & 
 1.0837 & 1.0806 & 1.0811 \\
\hline
\end{tabular}
\begin{tabular}{l|r|r|r}
\hline
& adapt$_{5,2}^{(1,0)}$ & adapt$_{7,3}^{(1,0)}$ & 
2\,adapt$_{9,4}^{(1,0)}$ \\     
\hline 
$\Mloss$ & 
  0.0\% &  0.0\% &  0.0\% \\
$\strikec_{\min}$ & 
 0.9007 & 0.9015 & 0.9014 \\
$t_{\strikec = \strikec_{\min}}$ & 
 1.9410 & 1.9040 & 1.9000 \\
$V_{c,\max}$ & 
 0.2417 & 0.2418 & 0.2417 \\
$t_{V_c = V_{c,\max}}$ & 
 0.9250 & 0.9250 & 0.9230 \\
$y_c(t=3)$ & 
 1.0873 & 1.0824 & 1.0819 \\
\hline
\end{tabular}
\caption{Some quantitative results for the test case 1 in 
\cite{HysingTKPBGT09}. Here we use the P2-(P1+P0) element (\ref{eq:P2P10}) with
(\ref{eq:rhoma}) and \XFEMGAMMA.
The bottom table is for (\ref{eq:rhomc}) and \XFEMGAMMA.}
\label{tab:data10a}
\end{table}%

Overall our results that appear to agree most closely with the results from
group 3 in \cite{HysingTKPBGT09} are the ones for the finest run with the 
P2-(P1+P0) element and the strategy (\ref{eq:rhomc});
see Table~\ref{tab:data10a}. 
In what follows we
present some visualizations of the numerical results for that run.
Here we recall from Table~\ref{tab:CPU} that this run took less CPU time then
the run for 5\,adapt$_{11,5}^{(1)}$, i.e.\ the finest discretization for the
P2-P1 element.
A plot of $\Gamma^M$ can be seen in Figure~\ref{fig:bubble}, while the time
evolution of the circularity, the centre of mass and the rise velocity
are shown in Figures~\ref{fig:circularity} and \ref{fig:comrise}. 
A plot of the discrete energy 
as well as a plot of 
the mesh quality of $\Gamma^m$ over
time, recall (\ref{eq:r}), are shown in Figure~\ref{fig:enratio}. 
\begin{figure}
\center
\includegraphics[angle=-90,width=0.6\textwidth]{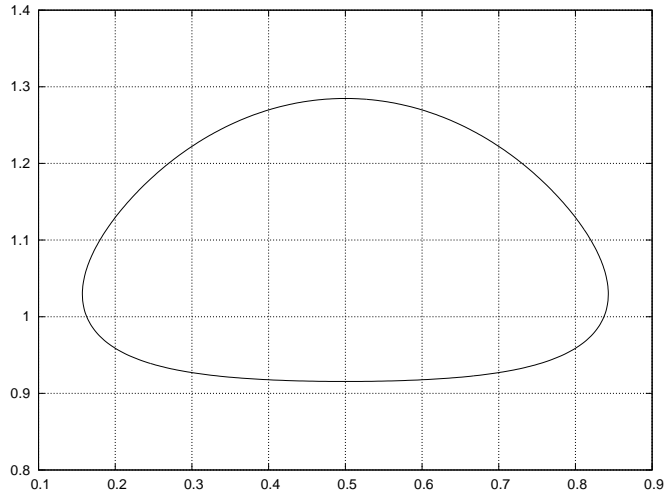}
\caption{(2\,adapt$_{9,4}^{(1,0)}$)
The final bubble for the test case 1 at time $T=3$.}
\label{fig:bubble}
\end{figure}%
\begin{figure}
\center
\includegraphics[angle=-90,width=0.45\textwidth]{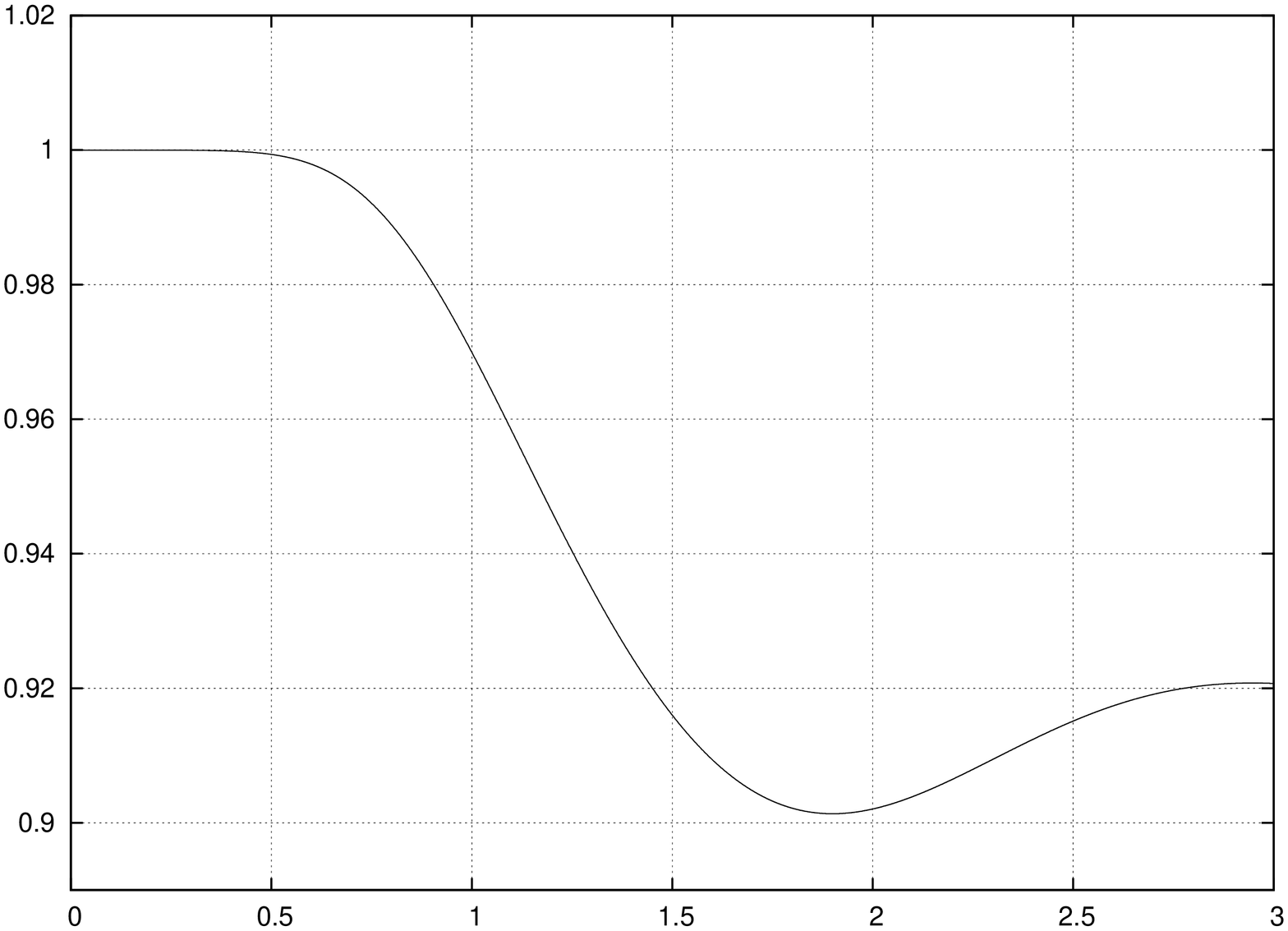}
\includegraphics[angle=-90,width=0.45\textwidth]{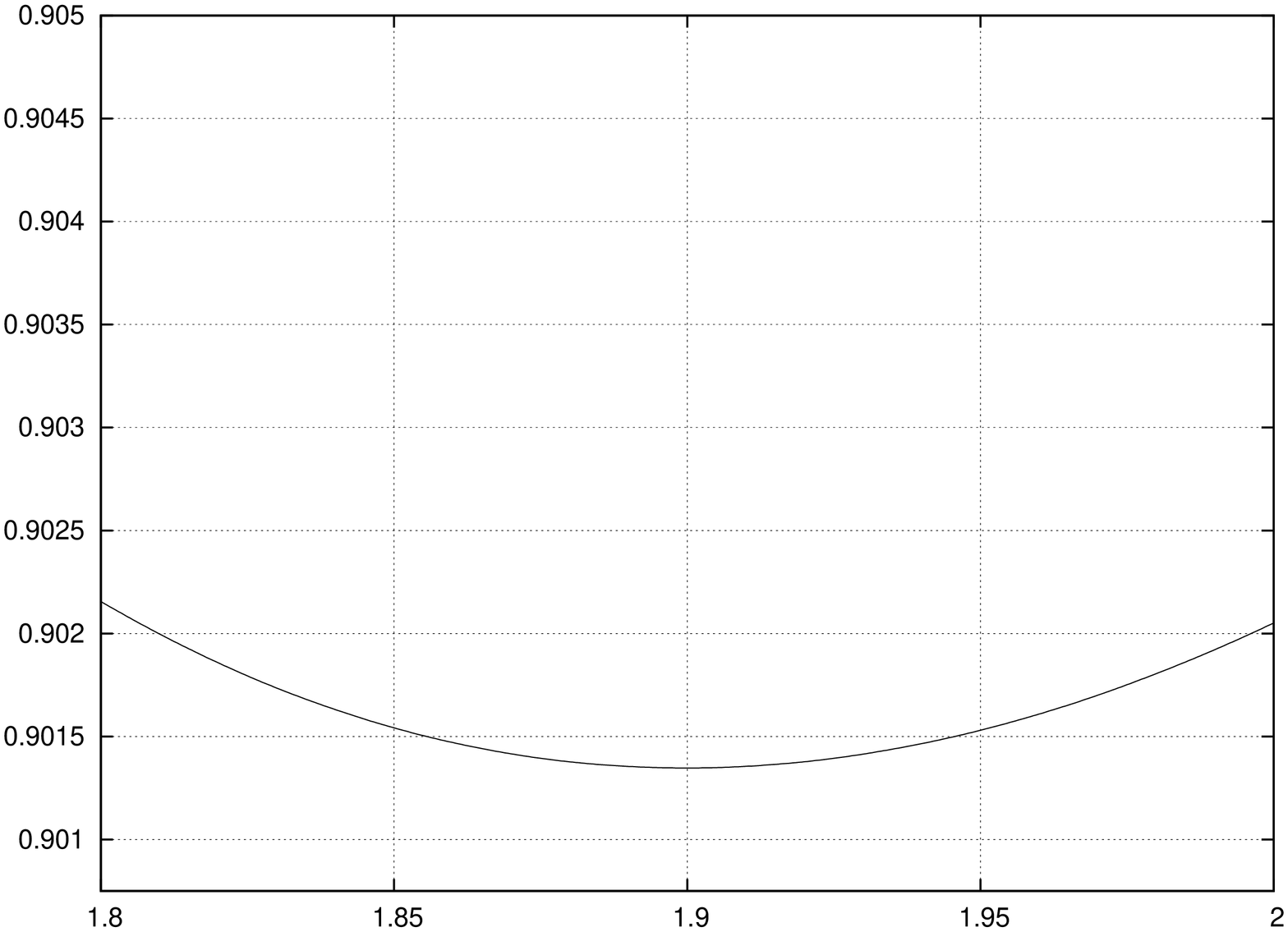}
\caption{(2\,adapt$_{9,4}^{(1,0)}$)
Circularity for the test case 1.}
\label{fig:circularity}
\end{figure}%
\begin{figure}
\center
\includegraphics[angle=-90,width=0.45\textwidth]{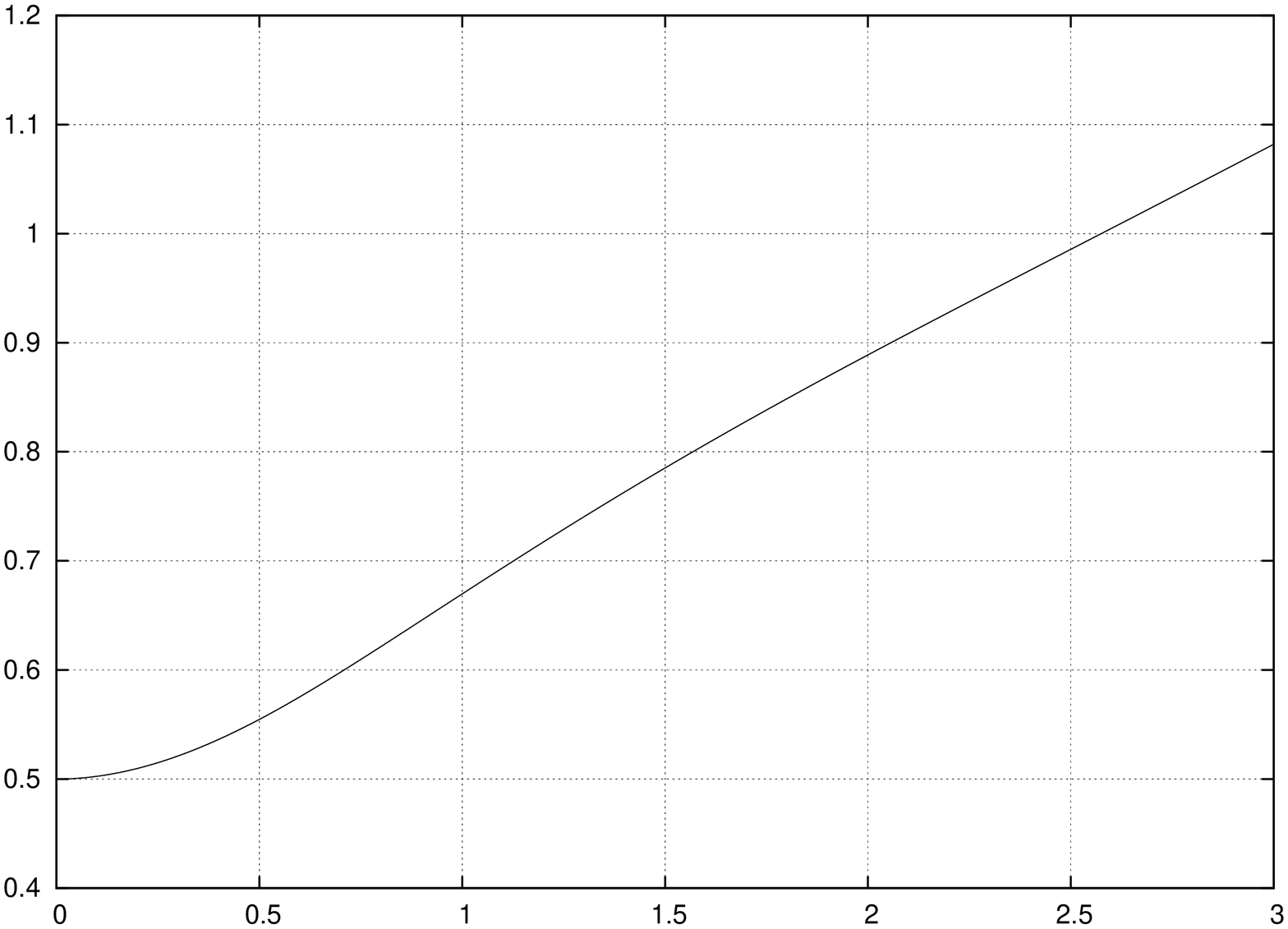}
\includegraphics[angle=-90,width=0.45\textwidth]{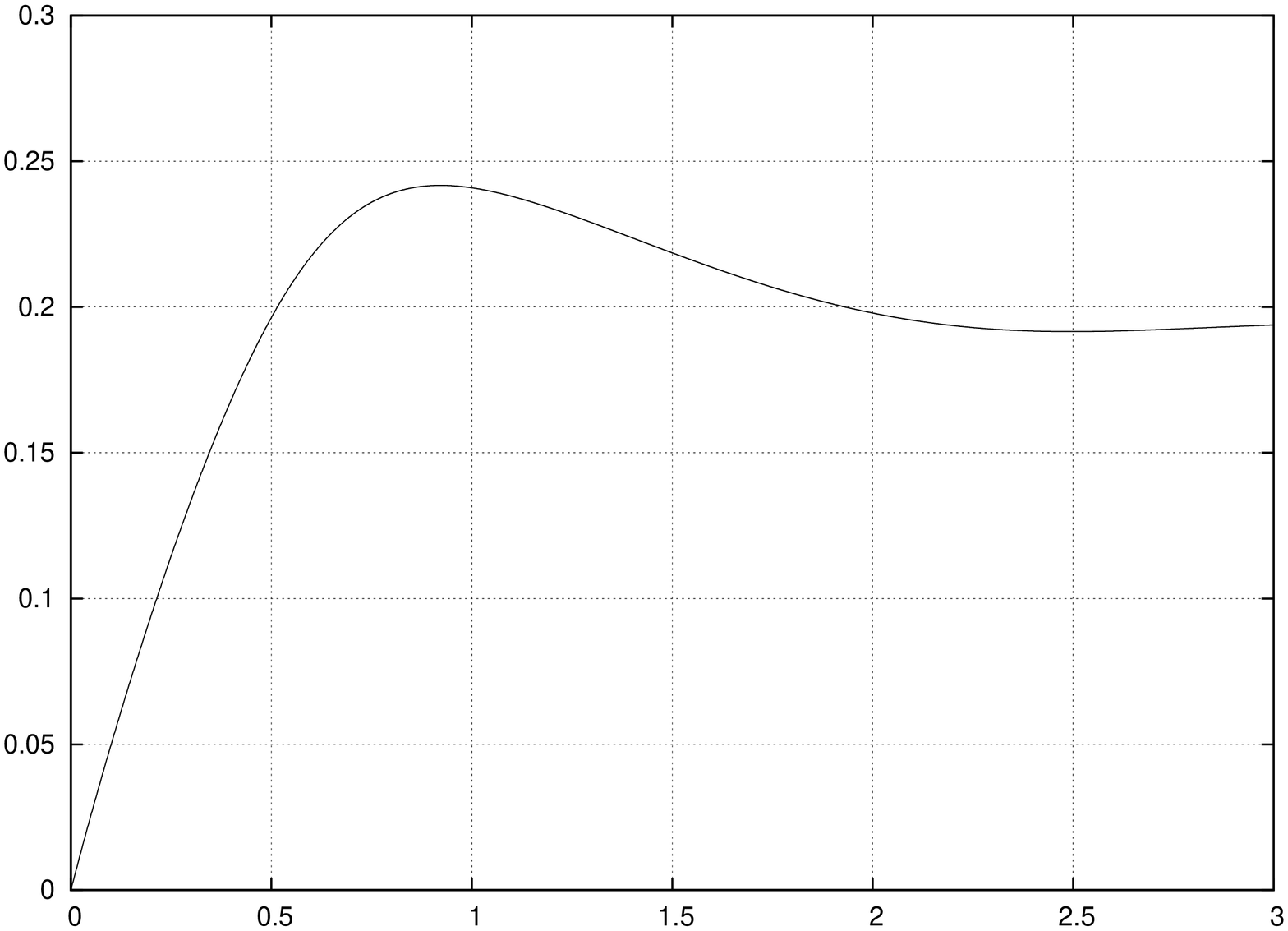}
\caption{(2\,adapt$_{9,4}^{(1,0)}$)
Centre of mass and rise velocity for the test case 1.}
\label{fig:comrise}
\end{figure}%
\begin{figure}
\center
\includegraphics[angle=-90,width=0.45\textwidth]{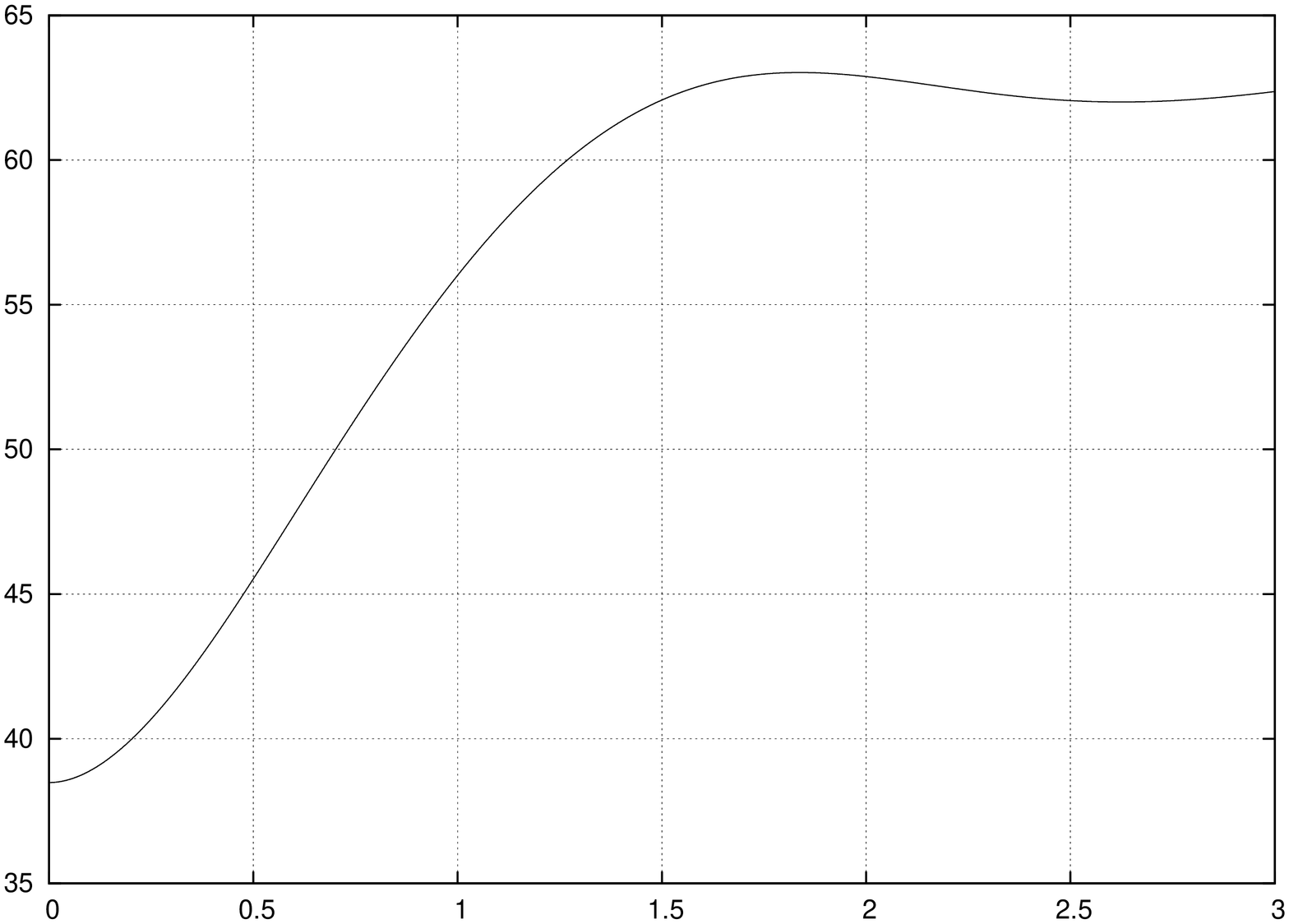}
\includegraphics[angle=-90,width=0.45\textwidth]{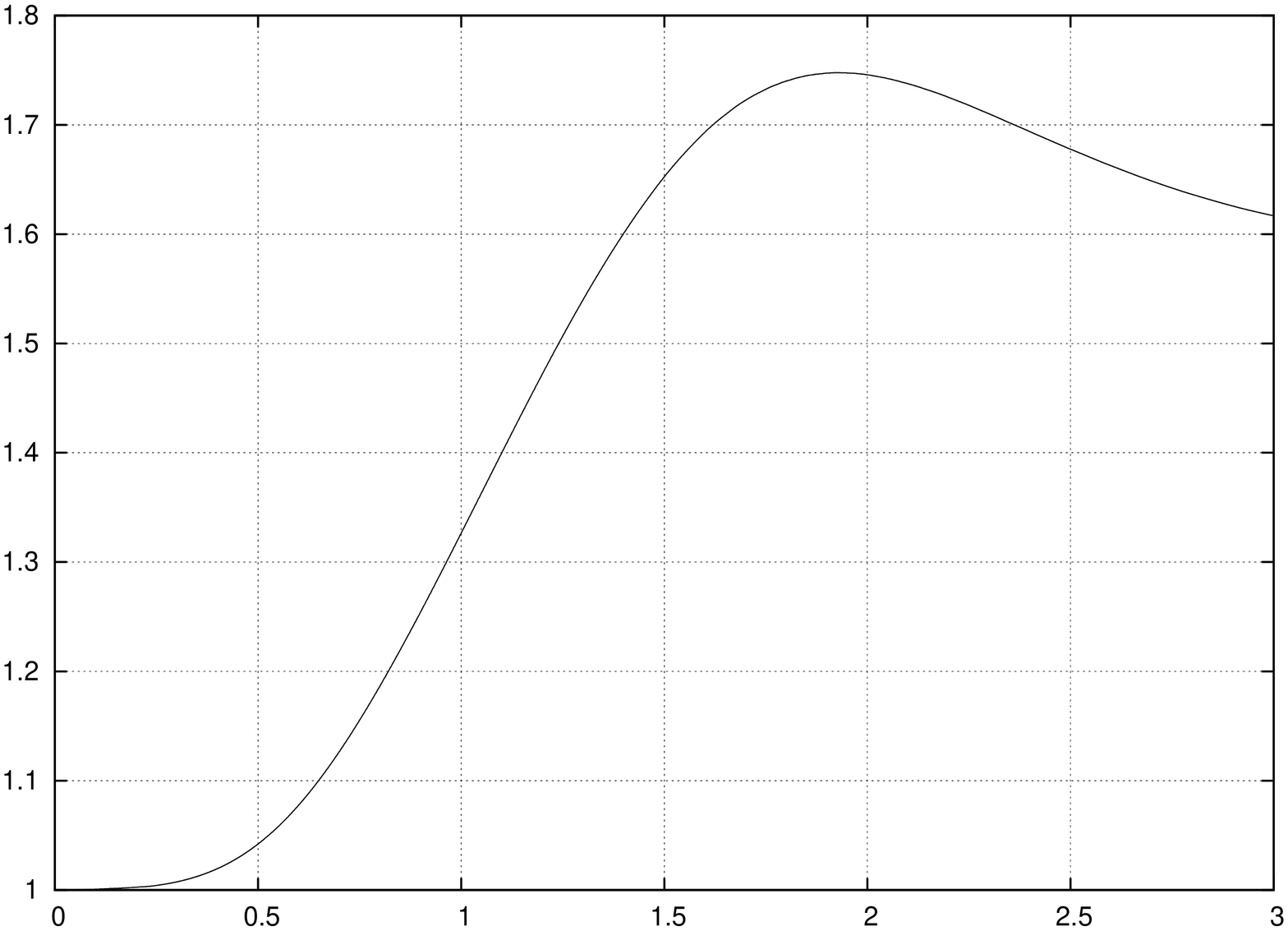}
\caption{(2\,adapt$_{9,4}^{(1,0)}$)
Discrete energy $\mathcal{E}(\rho^m, \vec U^{m+1}, \Gamma^{m+1})$ 
and the mesh quality $r^m$ for the test case 1.}
\label{fig:enratio}
\end{figure}%

\subsubsection{2d benchmark problem 2} \label{sec:614}
We use the same setup as in Section~\ref{sec:613}, and take the physical 
parameters from the test case 2 in \cite[Table~I]{HysingTKPBGT09}, which are
given by
\begin{equation} \label{eq:Hysing2}
\rho_+ = 1000\,,\quad \rho_- = 1\,,\quad \mu_+ = 10\,,\quad \mu_- = 0.1\,,\quad
\gamma = 1.96\,,\quad \vec f_1 = -0.98\,\vec\ek_d\,,\quad \vec f_2 = \vec 0\,.
\end{equation}
The time interval chosen for the simulation is $[0,T]$ with $T=3$, as before.

Some discretization parameters and CPU times for (\ref{eq:HGa}--d) 
are shown in Table~\ref{tab:CPU2}.
\begin{table}
\center
\begin{tabular}{l|r|r|r|r|r|r}
\hline
& $K^0_\Gamma$ & $J^0_\Omega$ & NDOF$_{\rm bulk}$ & $M$ & CPU &CPU \XFEMGAMMA\\
\hline 
adapt$_{5,2}^{(1)}$   &   32 &   536 &   2475 & 3000 &  195  &   277  \\
adapt$_{7,3}^{(1)}$   &  128 &  2320 &  10563 & 3000 &  2600 & 11270 \\
2\,adapt$_{9,4}^{(1)}$ & 512 &  9728 &  44019 & 6000 & 33000 & 76200 \\ 
\hline
adapt$_{5,2}^{(1,0)}$   &   32 &   536 &   3011 & 3000 &    770 &   2520 \\
adapt$_{7,3}^{(1,0)}$   &  128 &  2320 &  12883 & 3000 &  13600 &  53600 \\
2\,adapt$_{9,4}^{(1,0)}$ & 512 &  9728 &  53747 & 6000 & 517900 & 611400 \\
\hline
\end{tabular}
\caption{Simulation statistics and timings for the test case 2 in 
\cite{HysingTKPBGT09}.}
\label{tab:CPU2}
\end{table}%
Selected benchmark quantities are shown in Tables~\ref{tab:dataa2}
and \ref{tab:data10a2} for the elements (\ref{eq:P2P1}) and (\ref{eq:P2P10}),
respectively.
We observe that the results in these two tables
are in good agreement with the corresponding numbers from the finest
discretization run of group 3 in \cite{HysingTKPBGT09}, which are given by
0.5144, 3.0000, 0.2502, 0.7317, 0.2393, 2.0600 and 1.1376. 
Here we note that there is little agreement on these results between the three
groups in \cite{HysingTKPBGT09}, but we believe the numbers of
group 3 to be the most reliable ones.
\begin{table}
\center
\begin{tabular}{l|r|r|r}
\hline
& adapt$_{5,2}^{(1)}$ & adapt$_{7,3}^{(1)}$ & 
2\,adapt$_{9,4}^{(1)}$ \\
\hline 
$\Mloss$ & 
  0.0\% &  0.0\% &  0.0\% \\
$\strikec_{\min}$ & 
 0.5892 & 0.5192 & 0.5167 \\
$t_{\strikec = \strikec_{\min}}$ & 
 3.0000 & 3.0000 & 3.0000 \\
$V_{c,\max1}$ & 
 0.2584 & 0.2480 & 0.2488 \\
$t_{V_c = V_{c,\max1}}$ & 
 0.8800 & 0.7610 & 0.7310 \\
$V_{c,\max2}$ &
 0.2283 & 0.2305 & 0.2356 \\
$t_{V_c = V_{c,\max2}}$ & 
 2.0000 & 1.9500 & 2.0490 \\
$y_c(t=3)$ & 
 1.1275 & 1.1238 & 1.1319 \\
\hline
\end{tabular}
\begin{tabular}{l|r|r|r}
\hline
& adapt$_{5,2}^{(1)}$ & adapt$_{7,3}^{(1)}$ & 
2\,adapt$_{9,4}^{(1)}$ \\         
\hline 
$\Mloss$ & 
 -0.4\% &  0.0\% &  0.0\% \\
$\strikec_{\min}$ & 
 0.5671 & 0.5086 & 0.5077 \\
$t_{\strikec = \strikec_{\min}}$ & 
 3.0000 & 3.0000 & 3.0000 \\
$V_{c,\max1}$ & 
 0.2517 & 0.2507 & 0.2507 \\
$t_{V_c = V_{c,\max1}}$ & 
 0.7370 & 0.7270 & 0.7395 \\
$V_{c,\max2}$ & 
 0.2310 & 0.2376 & 0.2390 \\
$t_{V_c = V_{c,\max2}}$ & 
 1.9270 & 2.0250 & 2.0425 \\
$y_c(t=3)$ & 
 1.1162 & 1.1296 & 1.1350 \\
\hline
\end{tabular}
\caption{Some quantitative results for the test case 2 in 
\cite{HysingTKPBGT09}. Here we use the P2-P1 element (\ref{eq:P2P1}) with 
(\ref{eq:rhoma}) and \XFEMGAMMA. 
The bottom table is for (\ref{eq:rhomc}) and \XFEMGAMMA.
}
\label{tab:dataa2}
\end{table}%
\begin{table}
\center
\begin{tabular}{l|r|r|r}
\hline
& adapt$_{5,2}^{(1,0)}$ & adapt$_{7,3}^{(1,0)}$ & 
2\,adapt$_{9,4}^{(1,0)}$ \\
\hline 
$\Mloss$ & 
 -0.2\% &  0.0\% &  0.0\% \\
$\strikec_{\min}$ & 
 0.5701 & 0.5172 & 0.5131 \\
$t_{\strikec = \strikec_{\min}}$ & 
 3.0000 & 3.0000 & 3.0000 \\
$V_{c,\max1}$ & 
 0.2580 & 0.2480 & 0.2489 \\
$t_{V_c = V_{c,\max1}}$ & 
 0.8790 & 0.7620 & 0.7295 \\
$V_{c,\max2}$ & 
 0.2295 & 0.2311 & 0.2356 \\
$t_{V_c = V_{c,\max2}}$ & 
 1.9640 & 1.9400 & 2.0445 \\
$y_c(t=3)$ & 
 1.1226 & 1.1234 & 1.1318 \\
\hline
\end{tabular}
\begin{tabular}{l|r|r|r}
\hline
& adapt$_{5,2}^{(1,0)}$ & adapt$_{7,3}^{(1,0)}$ & 
2\,adapt$_{9,4}^{(1,0)}$ \\   
\hline 
$\Mloss$ & 
 -0.4\% &  0.0\% &  0.0\% \\
$\strikec_{\min}$ & 
 0.5542 & 0.5055 & 0.5067 \\
$t_{\strikec = \strikec_{\min}}$ & 
 3.0000 & 3.0000 & 3.0000 \\
$V_{c,\max1}$ & 
 0.2503 & 0.2507 & 0.2503 \\
$t_{V_c = V_{c,\max1}}$ & 
 0.7150 & 0.7150 & 0.7155 \\
$V_{c,\max2}$ & 
 0.2300 & 0.2372 & 0.2388 \\
$t_{V_c = V_{c,\max2}}$ & 
 1.9327 & 2.0290 & 2.0425 \\
$y_c(t=3)$ & 
 1.1224 & 1.1316 & 1.1356 \\
\hline
\end{tabular}
\caption{Some quantitative results for the test case 2 in 
\cite{HysingTKPBGT09}. Here we use the P2-(P1+P0) element (\ref{eq:P2P10}) with 
(\ref{eq:rhoma}) and \XFEMGAMMA. 
The bottom table is for (\ref{eq:rhomc}) and \XFEMGAMMA.
}
\label{tab:data10a2}
\end{table}%

We again visualize the numerical results for our simulation with the
finest discretization parameters for the 
P2-(P1+P0) element and the strategy (\ref{eq:rhomc}).
A plot of $\Gamma^M$ can be seen in Figure~\ref{fig:bubble2}, 
where we observe that no self intersections have occured, in line with the
results of group 3 in \cite{HysingTKPBGT09}.
Some quantative statistics are shown in Figures~\ref{fig:circularity2} and
\ref{fig:comrise2}. 
Plots of the discrete energy 
and of the mesh quality of $\Gamma^m$ 
are shown in Figure~\ref{fig:enratio2}. Here the discontinuities in 
$r^m$, recall (\ref{eq:r}), are caused by the local refinement of $\Gamma^m$
as described in \S\ref{sec:52}.
For the displayed run
we start with $J^0_\Gamma = 512$ elements and finish with $J^M_\Gamma = 662$
elements.

\begin{figure}
\center
\includegraphics[angle=-90,width=0.6\textwidth]{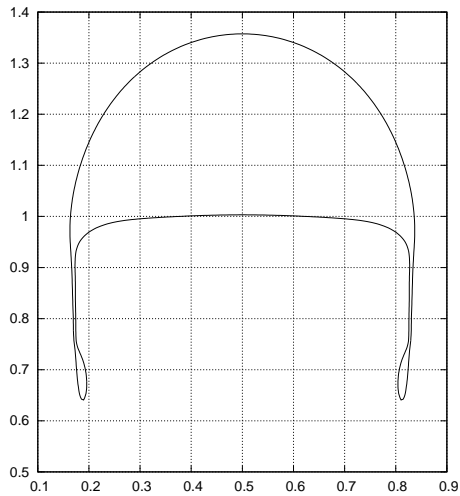}
\caption{(2\,adapt$_{9,4}^{(1,0)}$)
The final bubble for the test case 2 at time $T=3$.}
\label{fig:bubble2}
\end{figure}%
\begin{figure}
\center
\includegraphics[angle=-90,width=0.45\textwidth]{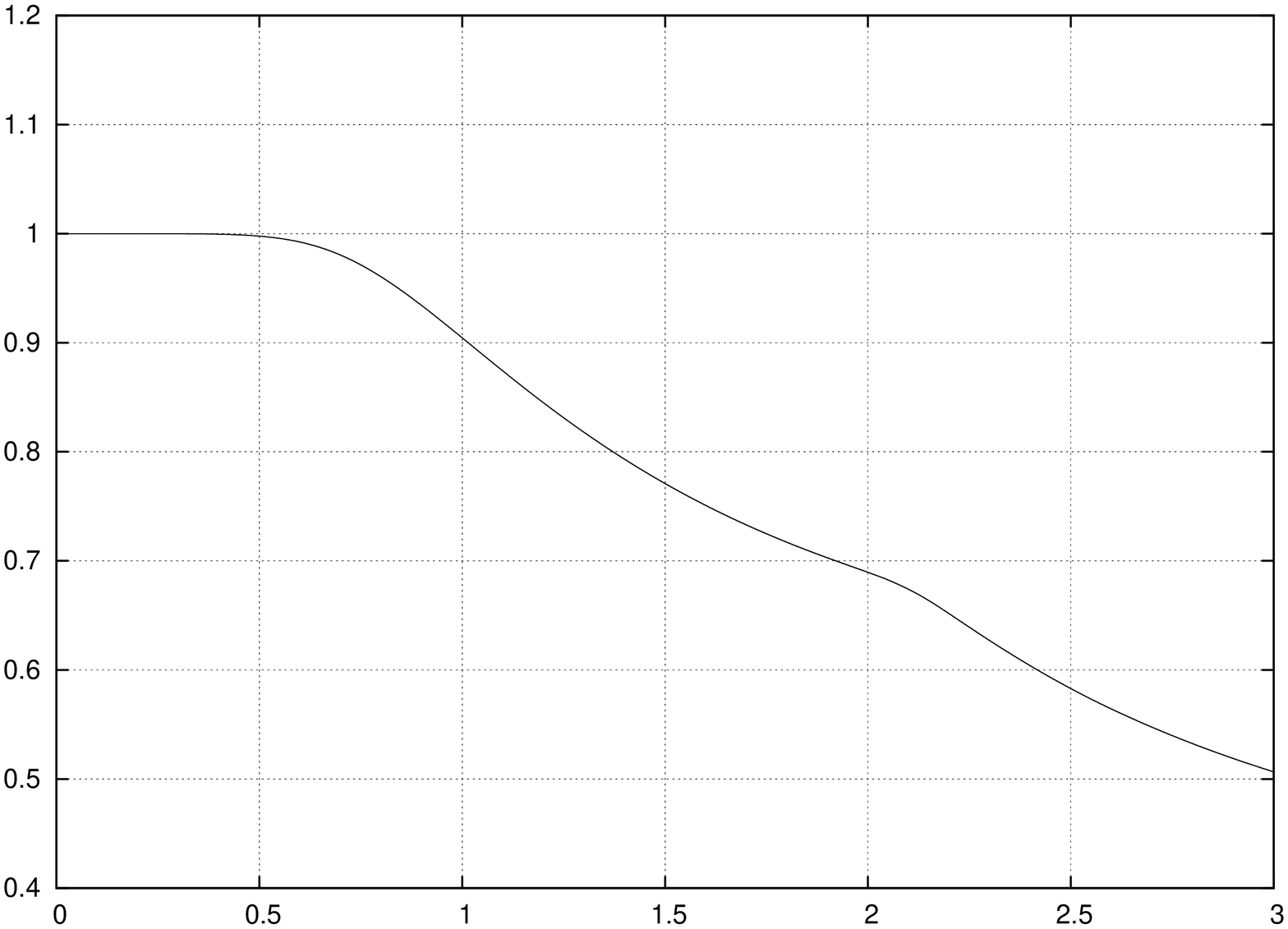}
\includegraphics[angle=-90,width=0.45\textwidth]{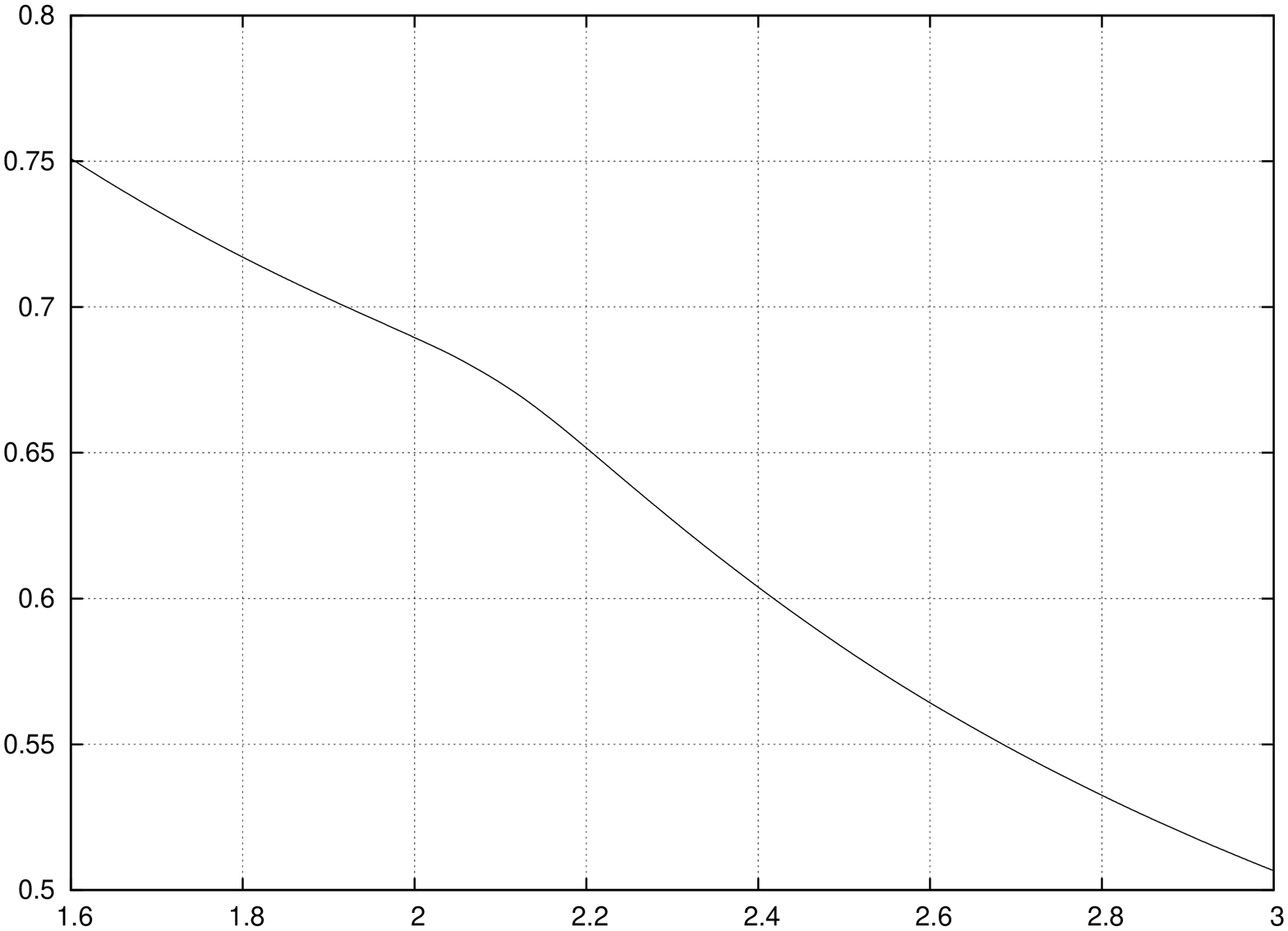}
\caption{(2\,adapt$_{9,4}^{(1,0)}$)
Circularity for the test case 2.}
\label{fig:circularity2}
\end{figure}%
\begin{figure}
\center
\includegraphics[angle=-90,width=0.45\textwidth]{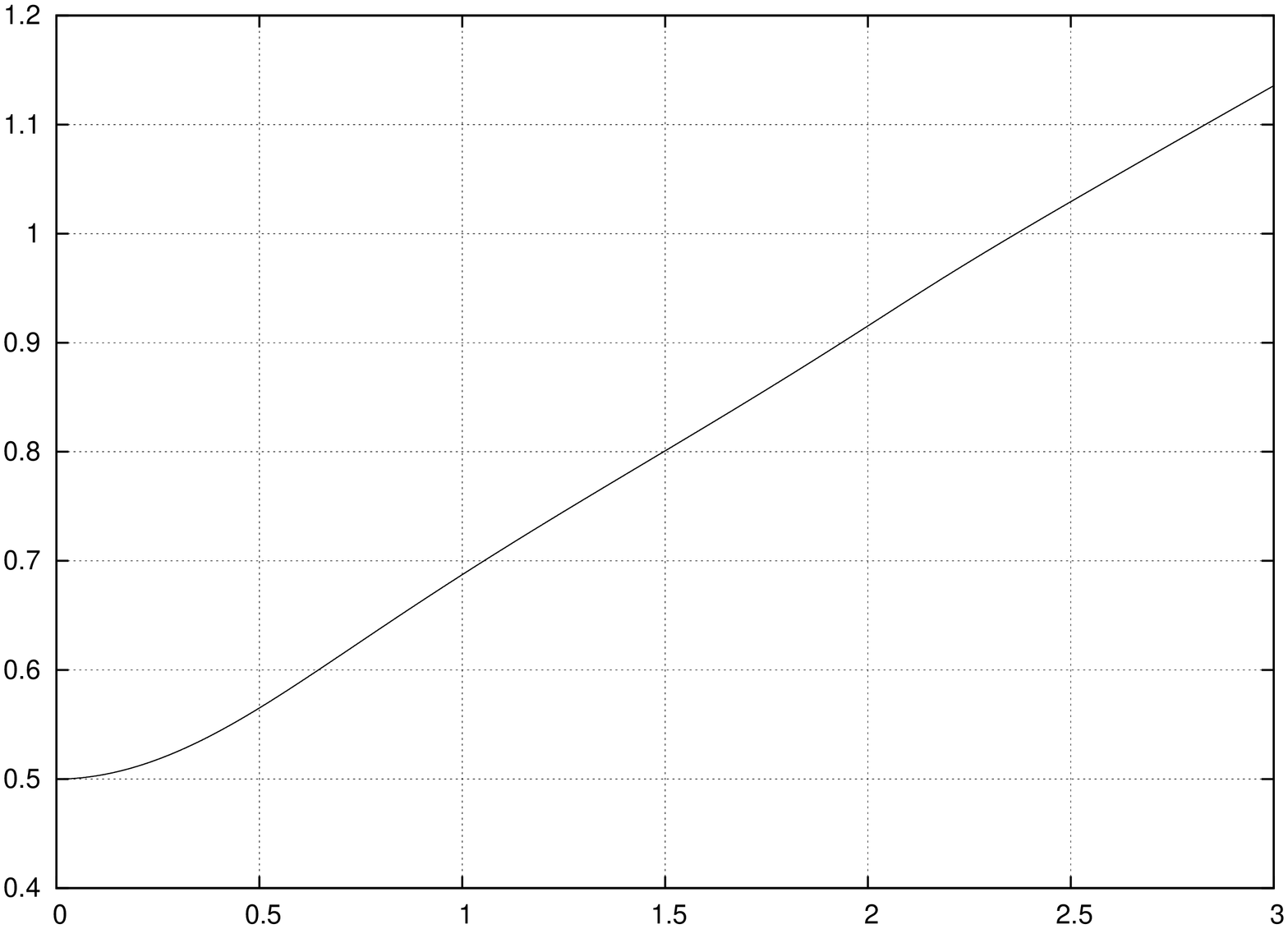}
\includegraphics[angle=-90,width=0.45\textwidth]{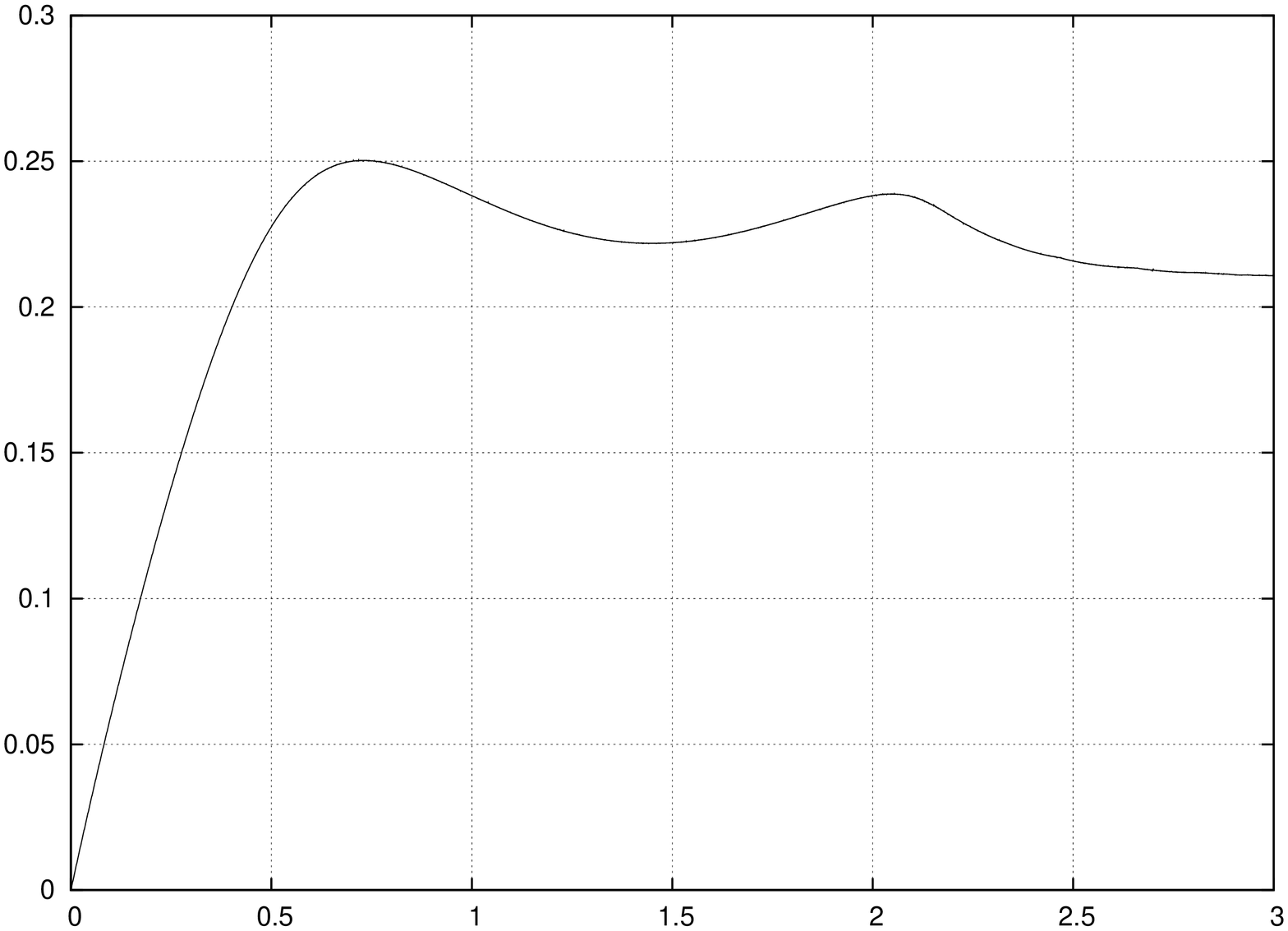}
\caption{(2\,adapt$_{9,4}^{(1,0)}$)
Centre of mass and rise velocity for the test case 2.}
\label{fig:comrise2}
\end{figure}%
\begin{figure}
\center
\includegraphics[angle=-90,width=0.45\textwidth]{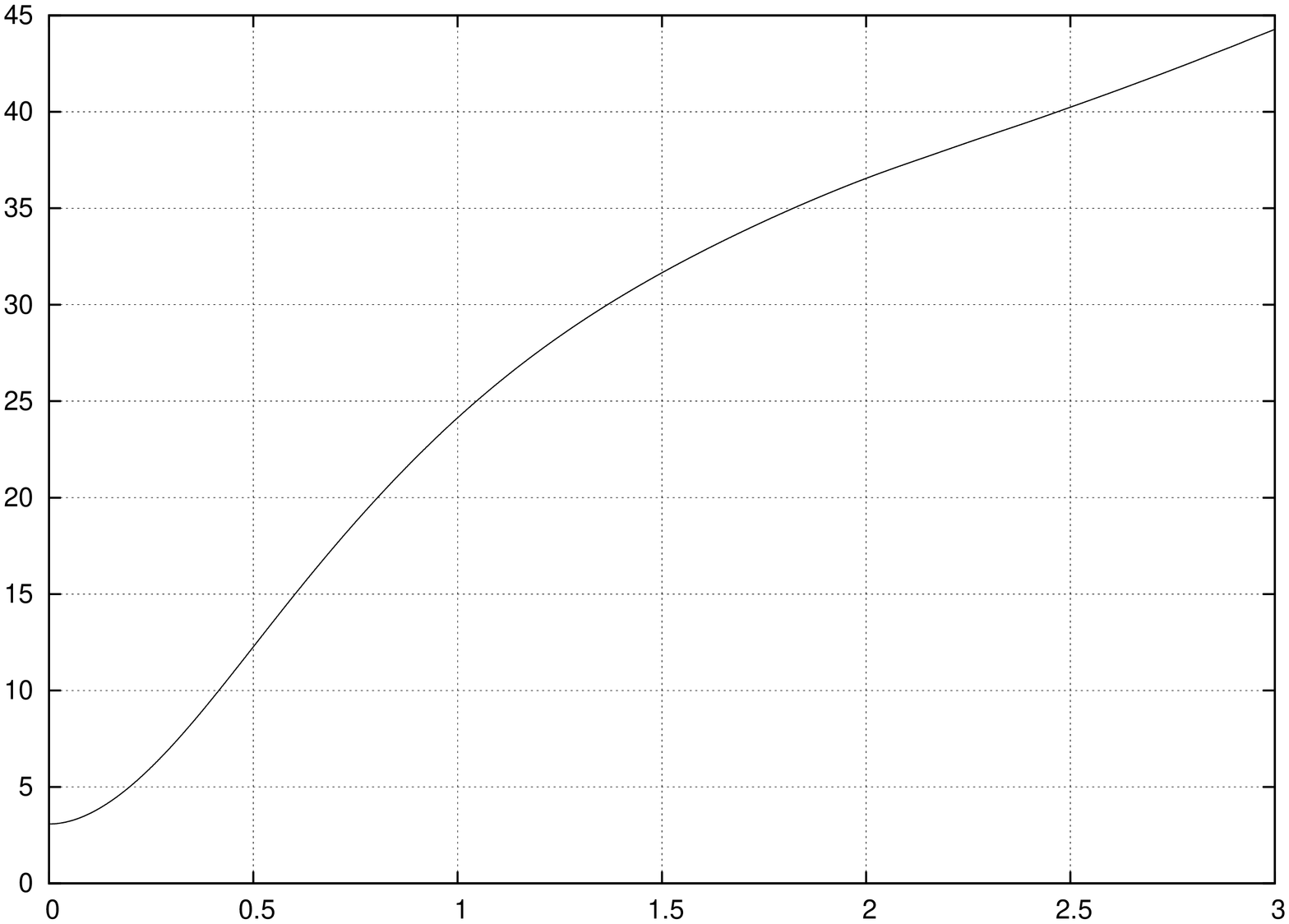}
\includegraphics[angle=-90,width=0.45\textwidth]{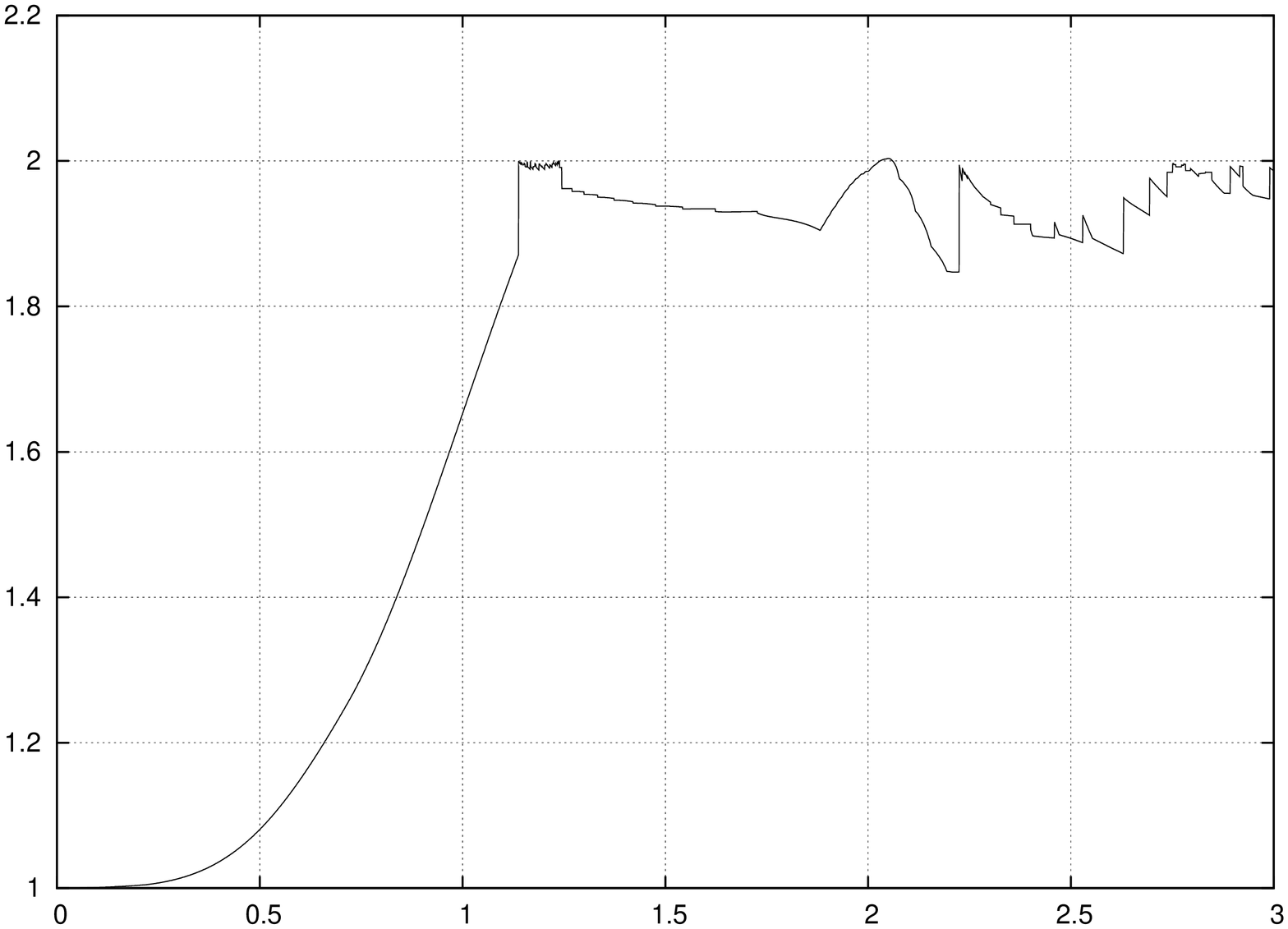}
\caption{(2\,adapt$_{9,4}^{(1,0)}$)
Discrete energy $\mathcal{E}(\rho^m, \vec U^{m+1}, \Gamma^{m+1})$ 
and the mesh quality $r^m$ for the test case 2.}
\label{fig:enratio2}
\end{figure}%

\subsection{Numerical results in 3d} \label{sec:62}
In this subsection we present numerical results for our approximation
(\ref{eq:HGa}--d) for the case $d = 3$. In particular, we will present
benchmark computations for the natural 3d analogues of the two-dimensional test
cases proposed in \cite[Table~I]{HysingTKPBGT09}. Moreover, we will present
some rising droplet simulations that are based on real physical parameters
suggested in \cite{GrossR11}.

For the 3d benchmark computations, we introduce the natural extensions of the 
quantities defined in (\ref{eq:benchmarkm}). That is,
the discrete approximations of the
$x_3$-component of the bubble's centre of mass, the ``degree of sphericity''
and the bubble's rise velocity are defined by
\begin{align*}
z_c^m & = \frac1{\mathcal{L}^3(\Omega_-^m)}\,\int_{\Omega_-^m} x_3 \dL3
= \frac3{\int_{\Gamma^m} \vec X^m\,.\,\vec \nu^m \dH{2}}
\int_{\Gamma^m} \tfrac12\,(\vec X^m \,.\,\vec\ek_3)^2\,
(\vec\nu^m \,.\,\vec\ek_3) \dH{2} \,,
\quad \nonumber \\ 
\strikes^m & = \pi^\frac13\,[6\,\mathcal{L}^3(\Omega_-^m)]^\frac23\,
[\mathcal{H}^{2}(\Gamma^m)]^{-1}\,, \quad 
V^m_c = \frac{(\rho^m_-\,U^m_3, 1)}{(\rho^m_-,1)}\,.
\end{align*}

\subsubsection{3d benchmark problem 1} \label{sec:621}
Here we consider the natural 3d analogue of the problem in \S\ref{sec:613},
i.e.\ of test case 1 in \cite{HysingTKPBGT09}, where only benchmark problems 
in two space dimensions are presented. To this end, we let $\Omega =
(0,1) \times (0,1) \times (0.2)$ with 
$\partial_1\Omega = [0,1] \times [0,1] \times \{0,2\}$ and 
$\partial_2\Omega = \partial\Omega \setminus \partial_1\Omega$.
Moreover, we set $T=3$, $\Gamma_0 = \{ \vec z \in \R^3 : |\vec z -
(\frac12, \frac12, \frac12)^T| = \frac14\}$, and choose the physical 
parameters as in (\ref{eq:Hysing1}). 

Our discretization parameters and CPU times are shown in Table~\ref{tab:3dCPU}.
The quantitative values for the evolution are given in Table~\ref{tab:3ddataa}.
\begin{table}
\center
\begin{tabular}{l|r|r|r|r|r}
\hline
& $K^0_\Gamma$ & $J^0_\Omega$ & NDOF$_{\rm bulk}$ & $M$ & CPU \XFEMGAMMA\\
\hline 
adapt$_{5,2}^{(1)}$  &  770 &  22320 &   95542 & 3000 & 68500 \\ 
adapt$_{6,3}^{(1)}$  & 1538 &  89616 &  383206 & 3000 & 456600 \\ 
\hline
\end{tabular}
\caption{Simulation statistics and timings for the 3d benchmark problem 1.} 
\label{tab:3dCPU}
\end{table}%
\begin{table}
\center
\begin{tabular}{l|r|r}
\hline
& adapt$_{5,2}^{(1)}$ & adapt$_{6,3}^{(1)}$ \\
\hline 
$\Mloss$ & 
  0.0\% & 0.0\% \\
$\strikes_{\min}$ & 
 0.9570 & 0.9508 \\
$t_{\strikes = \strikes_{\min}}$ & 
 3.0000 & 3.0000 \\
$V_{c,\max}$ & 
 0.3823 & 0.3845 \\
$t_{V_c = V_{c,\max}}$ & 
 1.1930 & 1.0800 \\
$z_c(t=3)$ & 
 1.5516 & 1.5555 \\
\hline
\end{tabular}
\caption{Some quantitative results for the 3d benchmark problem 1.}
\label{tab:3ddataa}
\end{table}%
In what follows we
present some visualizations of the numerical results for the run with
adapt$_{6,3}^{(1)}$.
Plots of $\Gamma^M$ can be seen in Figure~\ref{fig:3dbubble}, while the time
evolution of the sphericity, the centre of mass and the rise velocity
are shown in Figures~\ref{fig:sphericity} and \ref{fig:3dcomrise}. 
\begin{figure}
\center
\hspace*{-2.1cm}
\includegraphics[angle=-90,width=0.5\textwidth]{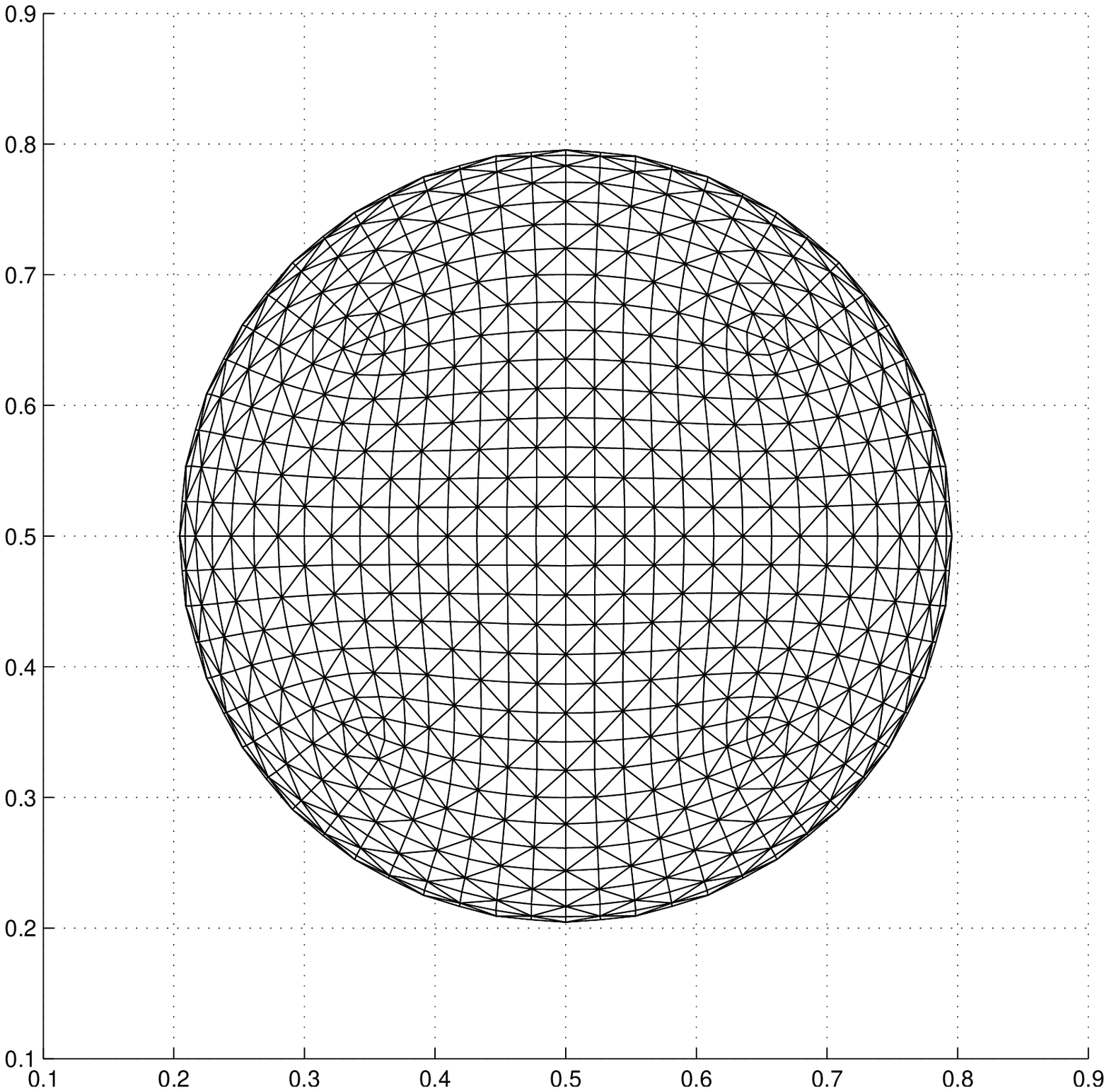}
\includegraphics[angle=-90,width=0.5\textwidth]{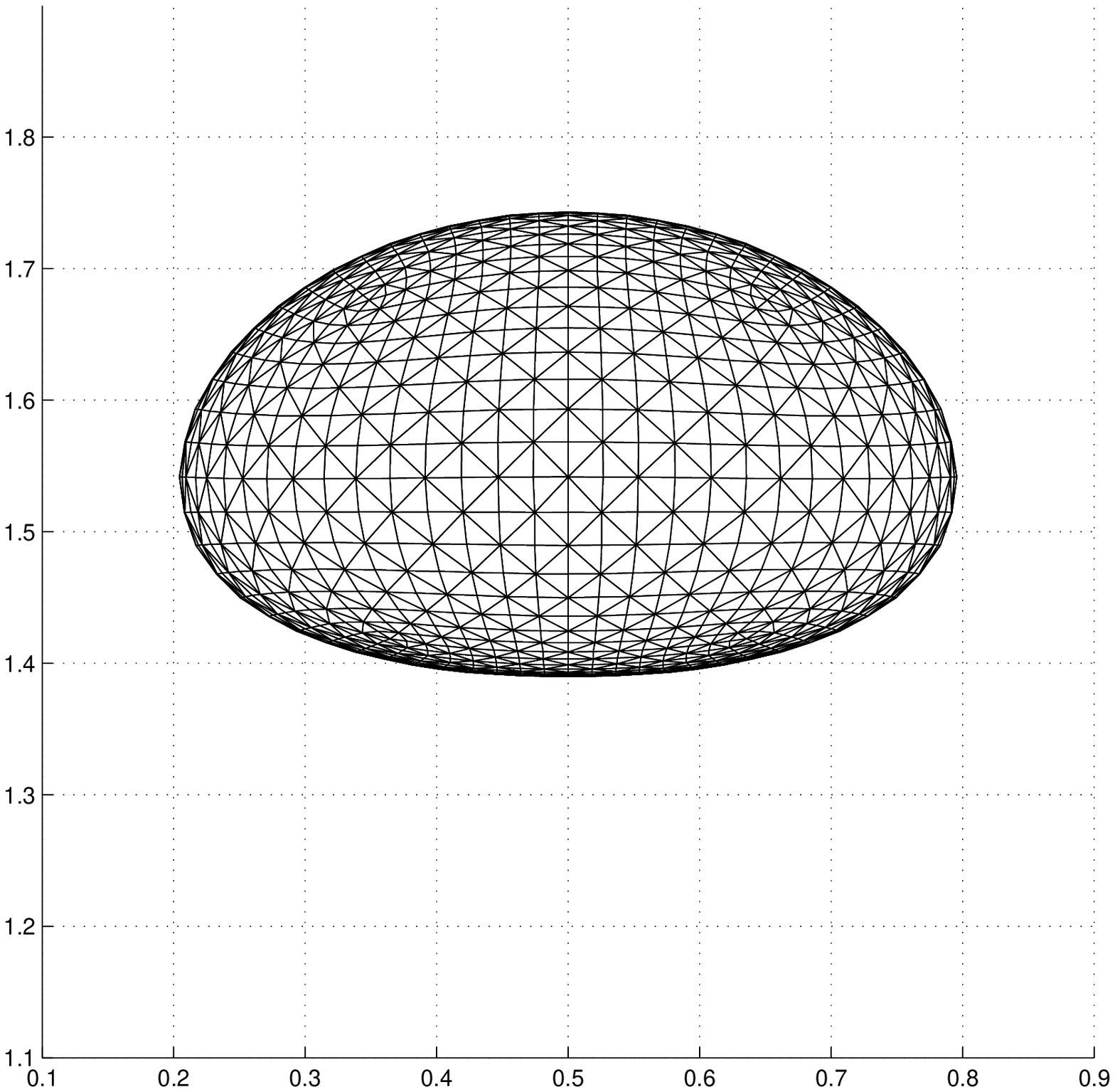}
\caption{(adapt$_{6,3}^{(1)}$)
The final bubble for the 3d benchmark problem 1 at time $T=3$. 
Views from the top (left) and from the front (right).}
\label{fig:3dbubble}
\end{figure}%
\begin{figure}
\center
\includegraphics[angle=-90,width=0.45\textwidth]{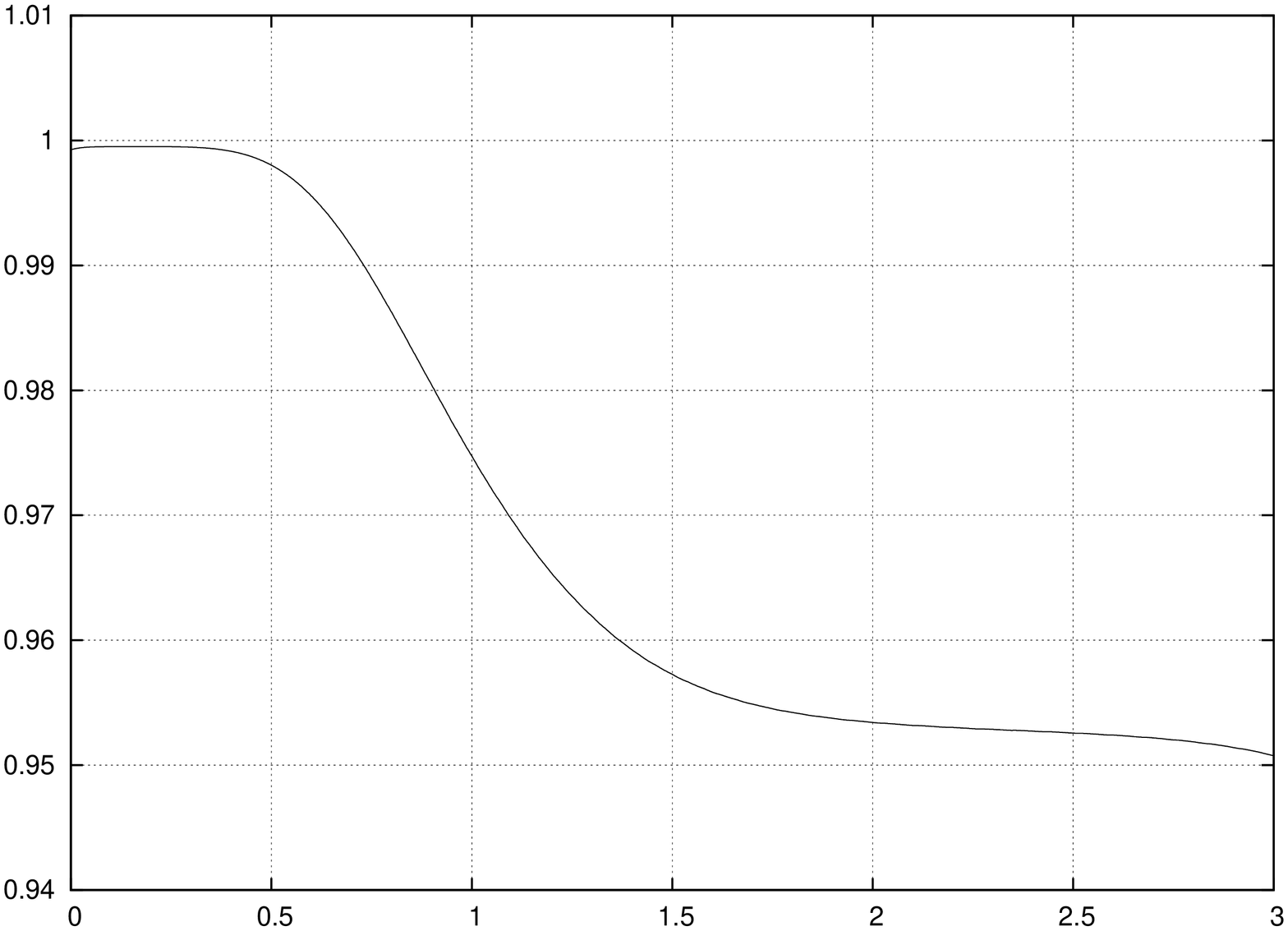}
\caption{(adapt$_{6,3}^{(1)}$)
Sphericity for the 3d benchmark problem 1.}
\label{fig:sphericity}
\end{figure}%
\begin{figure}
\center
\includegraphics[angle=-90,width=0.45\textwidth]{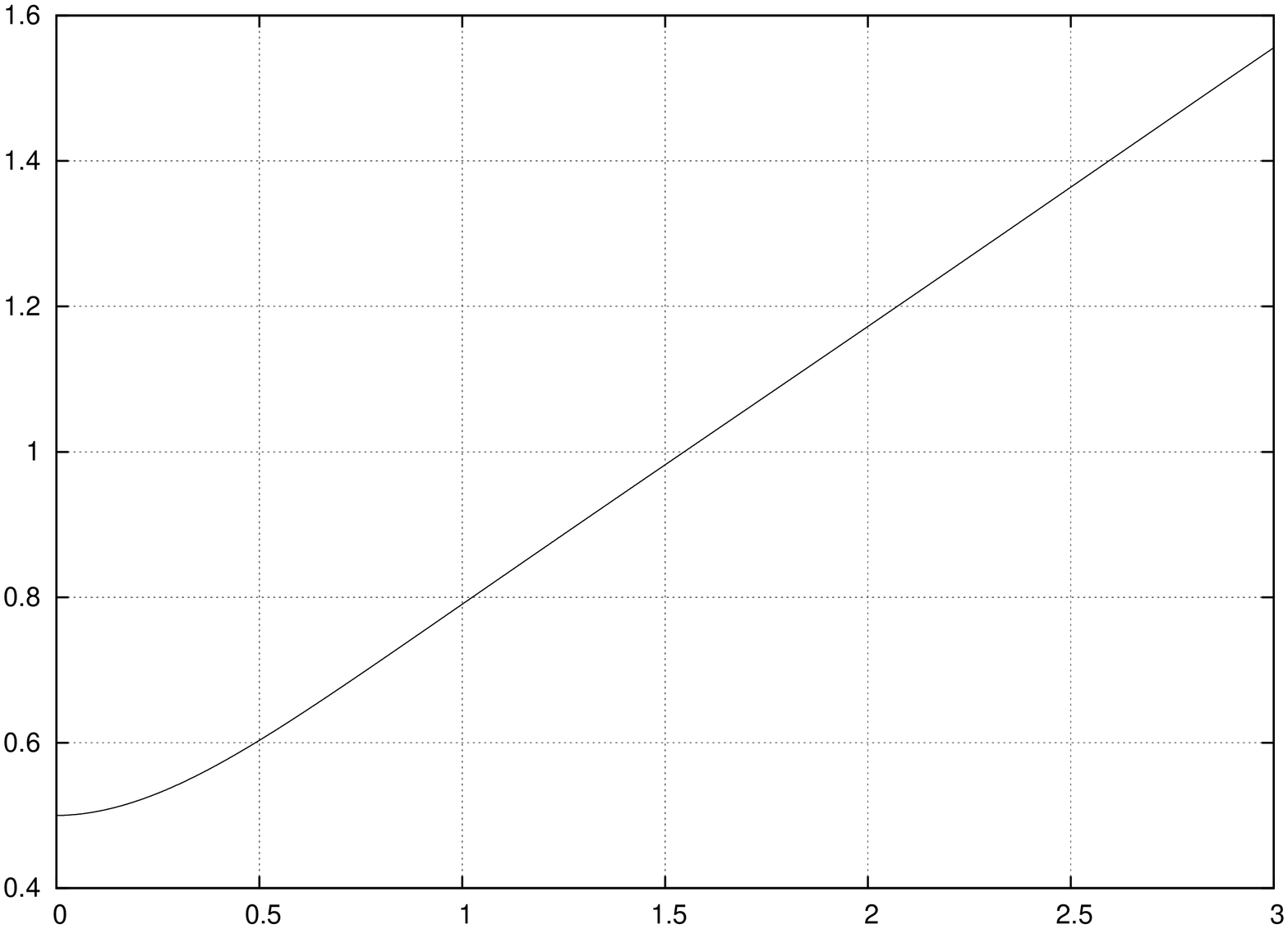}
\includegraphics[angle=-90,width=0.45\textwidth]{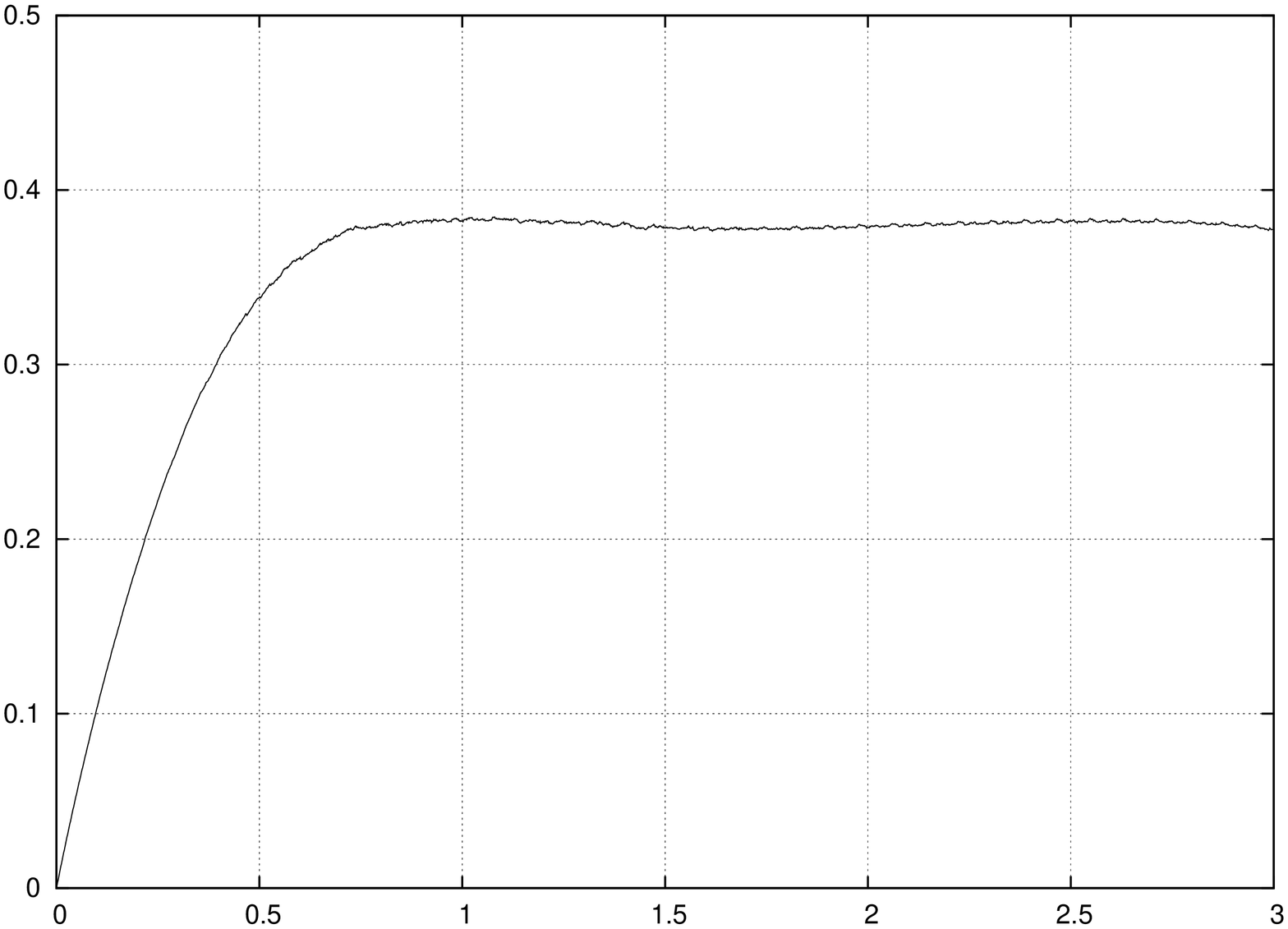}
\caption{(adapt$_{6,3}^{(1)}$)
Centre of mass and rise velocity for the 3d benchmark problem 1.}
\label{fig:3dcomrise}
\end{figure}%

\subsubsection{3d benchmark problem 2} \label{sec:622}
We use the same setup as in Section~\ref{sec:621}, and take the physical 
parameters as in (\ref{eq:Hysing2}).
The time interval chosen for the simulation is $[0,T]$ with $T=1.5$.
Some discretization parameters and CPU times for (\ref{eq:HGa}--d) 
are shown in Table~\ref{tab:3dCPU2_t15}.
The quantitative values for the evolution are given in 
Table~\ref{tab:3ddataa2_t15}.
\begin{table}
\center
\begin{tabular}{l|r|r|r|r|r}
\hline
& $K^0_\Gamma$ & $J^0_\Omega$ & NDOF$_{\rm bulk}$ & $M$ &CPU \XFEMGAMMA\\
\hline 
adapt$_{5,2}^{(1)}$  &  770 &  22320 &   95542 & 1500 & 231800\\ 
adapt$_{6,3}^{(1)}$  & 1538 &  89616 &  383206 & 1500 & 1494500 \\ 
\hline
\end{tabular}
\caption{Simulation statistics and timings for the 3d benchmark problem 2.}
\label{tab:3dCPU2_t15}
\end{table}%
\begin{table}
\center
\begin{tabular}{l|r|r}
\hline
& adapt$_{5,2}^{(1)}$ & adapt$_{6,3}^{(1)}$ \\
\hline 
$\Mloss$ & 
   0.0\% &  0.0\% \\
$\strikes_{\min}$ & 
  0.8144 & 0.7977 \\
$t_{\strikes = \strikes_{\min}}$ & 
  1.5000 & 1.5000 \\
$V_{c,\max}$ & 
  0.3819 & 0.3862 \\
$t_{V_c = V_{c,\max}}$ & 
  0.5670 & 0.5860 \\
$z_c(t=3)$ & 
  0.9831 & 0.9867 \\
\hline
\end{tabular}
\caption{Some quantitative results for the 3d benchmark problem 2.}
\label{tab:3ddataa2_t15}
\end{table}%
We again visualize the numerical results for our simulation with the
parameters adapt$_{6,3}^{(1)}$.
Plots of $\Gamma^M$ can be seen in Figure~\ref{fig:3dbubble2_t15}, where we
note the very nonconvex shape of the final bubble.
Some quantative statistics are shown in Figures~\ref{fig:sphericity2_t15} 
and \ref{fig:3dcomrise2_t15}. 
\begin{figure}
\center
\hspace*{-2.1cm}
\includegraphics[angle=-90,width=0.5\textwidth]{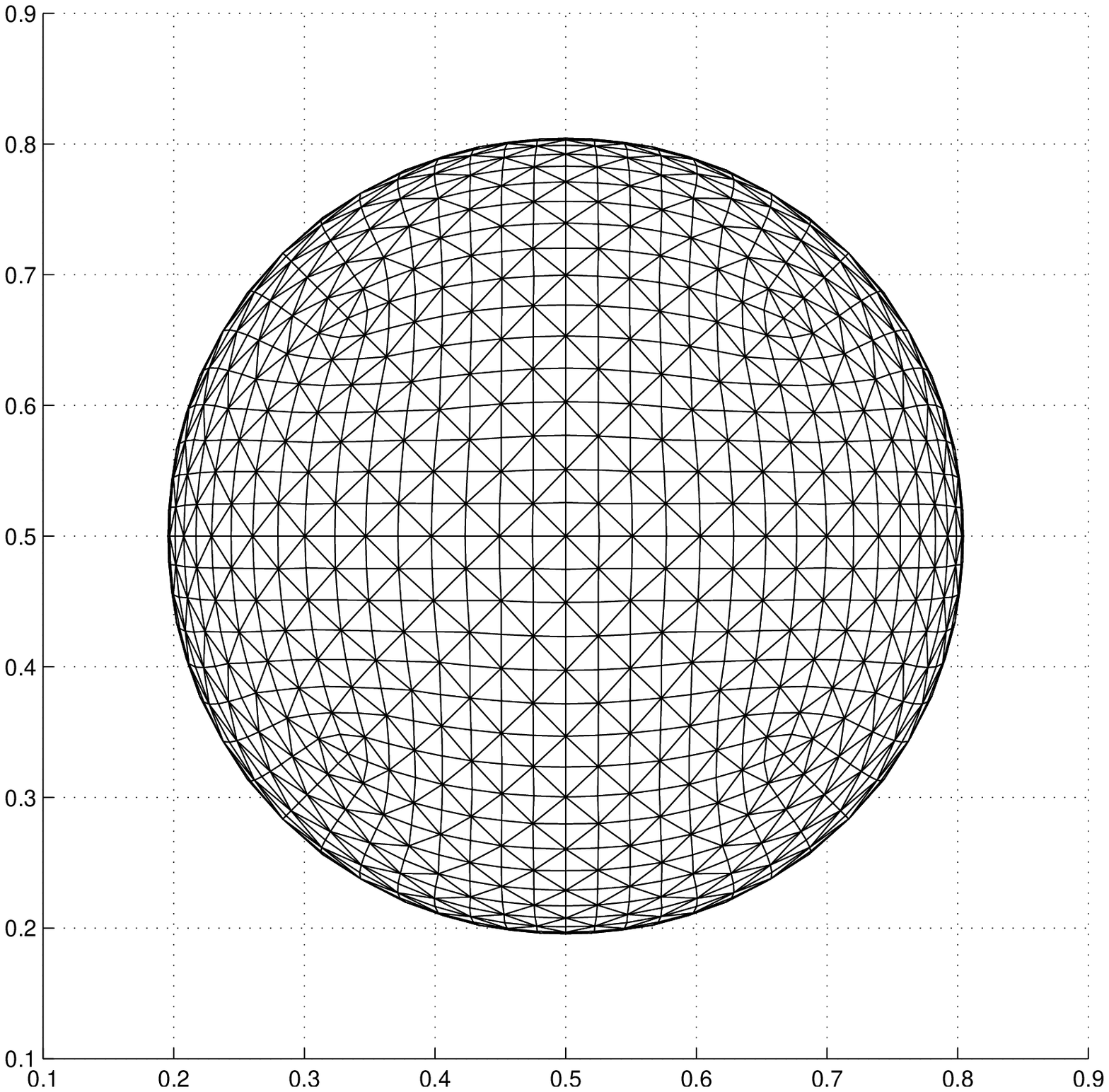}
\includegraphics[angle=-90,width=0.5\textwidth]{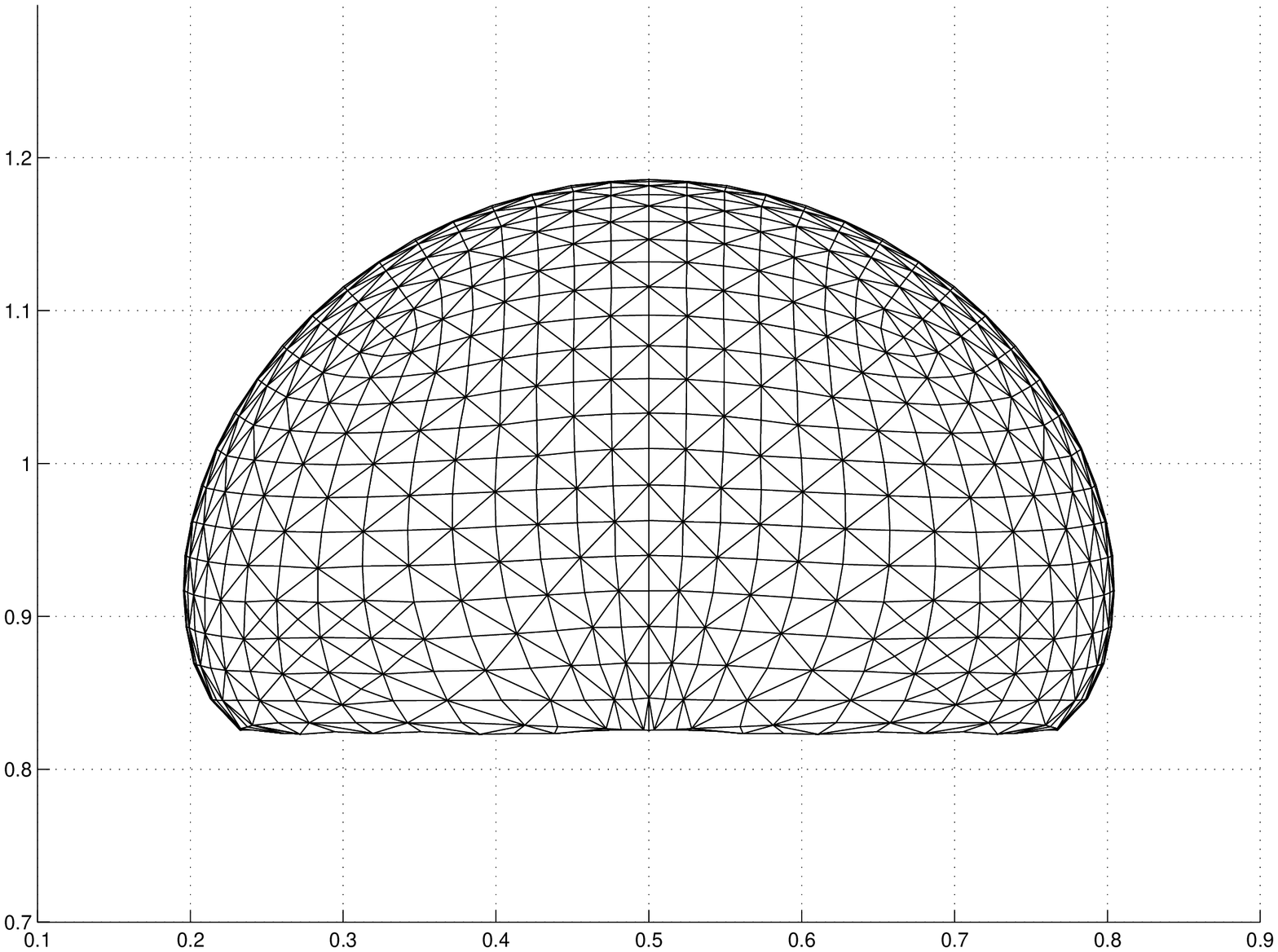}
\includegraphics[angle=-90,width=0.4\textwidth]{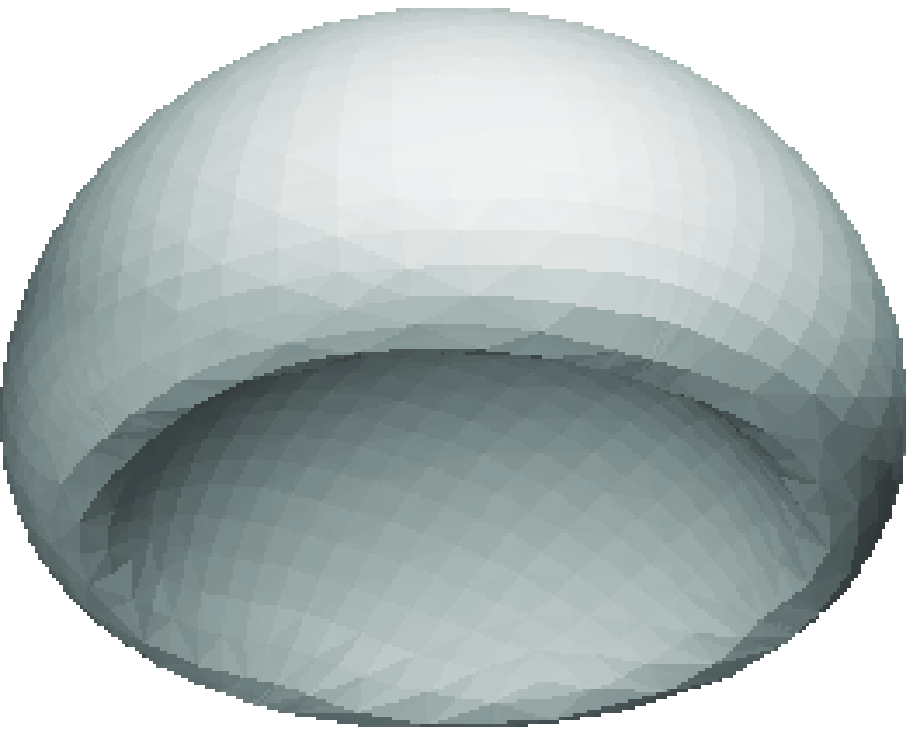}
\caption{(adapt$_{6,3}^{(1)}$)
The final bubble for the 3d benchmark problem 2 at time $T=1.5$. 
Views from the top (left) and from the front (right). Below an additional view
of the final bubble.}
\label{fig:3dbubble2_t15}
\end{figure}%
\begin{figure}
\center
\includegraphics[angle=-90,width=0.45\textwidth]{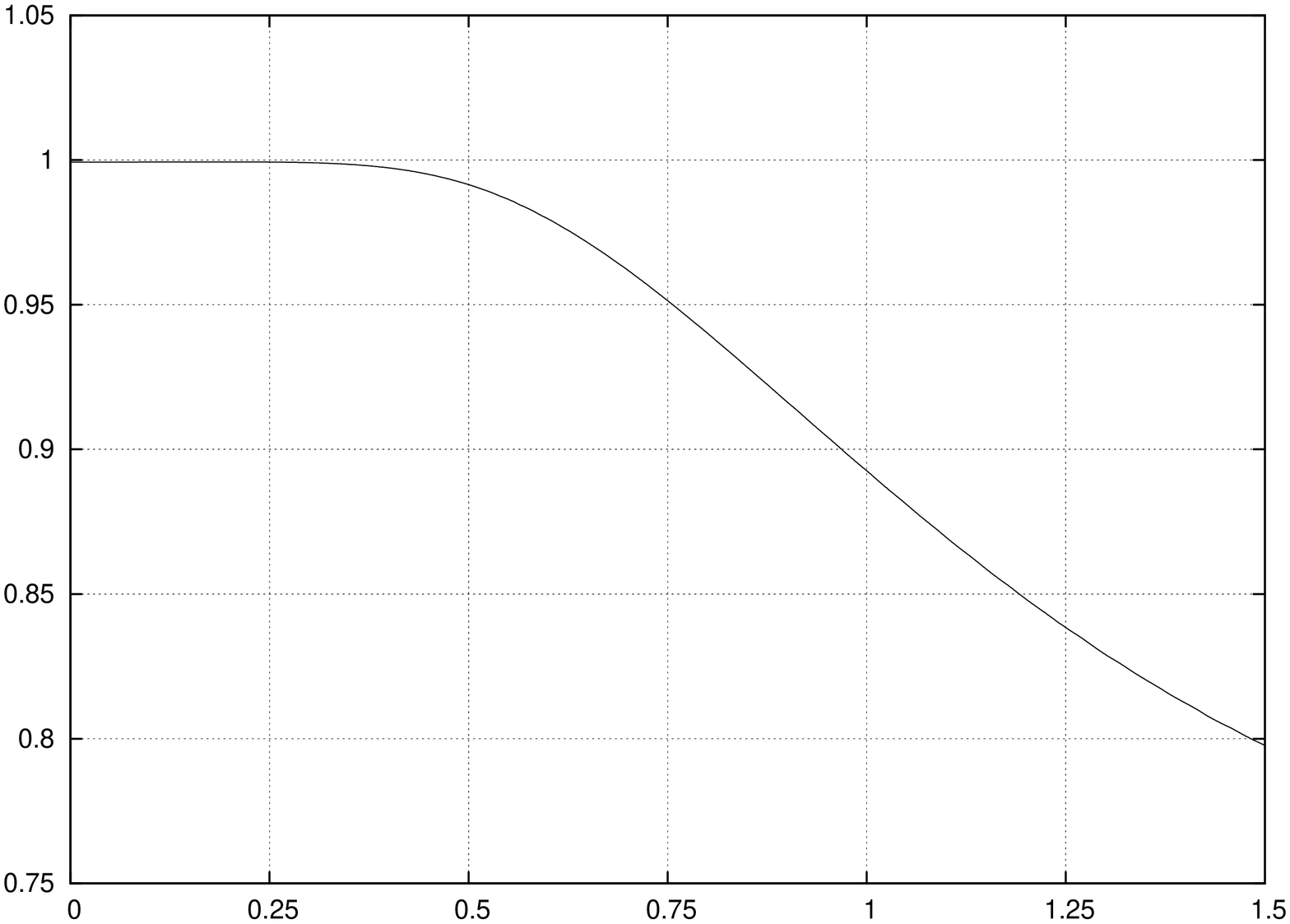}
\caption{(adapt$_{6,3}^{(1)}$)
Sphericity for the 3d benchmark problem 2.}
\label{fig:sphericity2_t15}
\end{figure}%
\begin{figure}
\center
\includegraphics[angle=-90,width=0.45\textwidth]{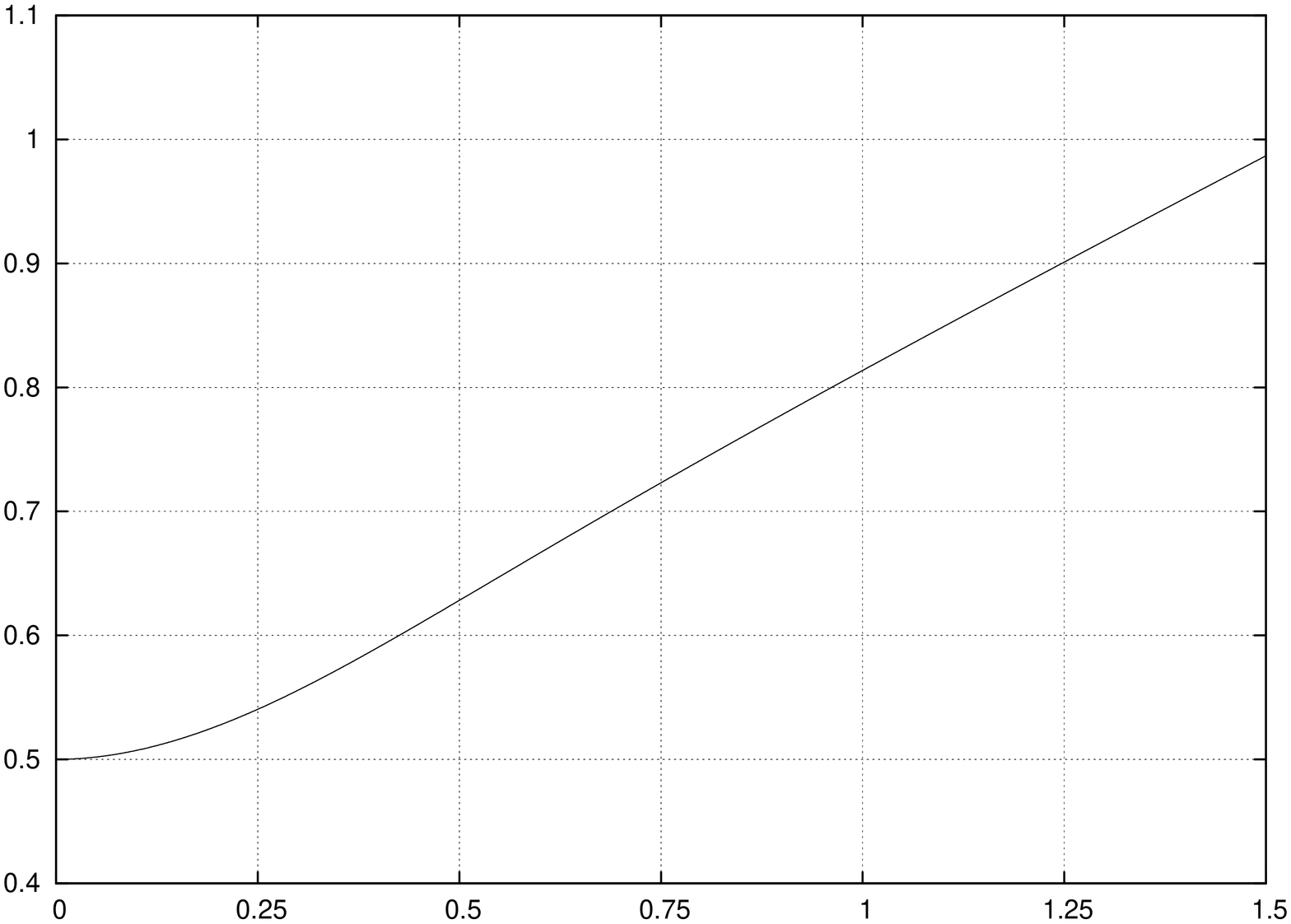}
\includegraphics[angle=-90,width=0.45\textwidth]{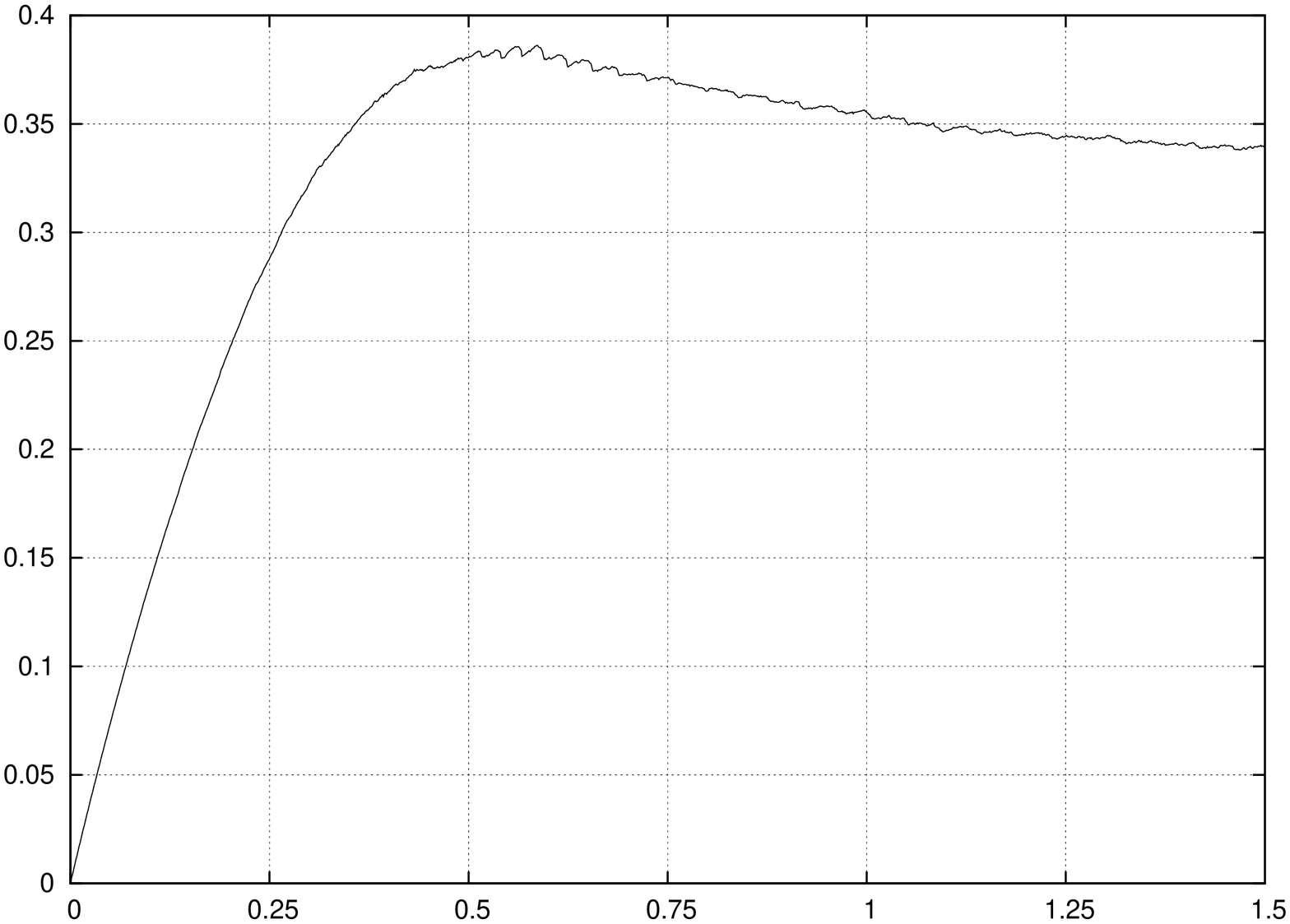}
\caption{(adapt$_{6,3}^{(1)}$)
Centre of mass and rise velocity for the 3d benchmark problem 2.}
\label{fig:3dcomrise2_t15}
\end{figure}%

\subsubsection{Rising butanol droplets in 3d} \label{sec:623}
Here we consider a problem that is inspired from \S1.3.1 in \cite{GrossR11}, 
where we apply a convenient rescaling in space. 
In particular,
we let $\Omega = (0, 1.2) \times (0, 1.2) \times (0, 3)$ with
$\partial_1\Omega = [0,1.2] \times [0,1.2] \times \{0,3\}$ and 
$\partial_2\Omega = \partial\Omega \setminus \partial_1\Omega$.
Moreover, we set
$\Gamma_0 = \{ \vec z \in \R^3: |\vec z - (0.6, 0.6, 0.3)| = R_0\}$
with $R_0 = 0.1$.
Finally,
\begin{align} \label{eq:Reusken1}
\rho_+ & = 9.865 \times 10^{-4},\quad \rho_- = 8.454 \times 10^{-4}\,,\quad 
\mu_+ = 1.388 \times 10^{-5}\,,\quad \mu_- = 3.281\times 10^{-5}\,,\quad
\nonumber \\
\gamma & = 1.63\times 10^{-3}\,,
\end{align}
with $\vec f_1 = -(0, 0, 981)^T$ and $\vec f_2 = \vec 0$. These parameters
model the evolution of a rising butanol droplet in a tank filled with water.
In \cite[\S1.3.1]{GrossR11} a terminal rise velocity of $V^M_c = 5.3$ 
(rescaled to our applied transformation in space) 
is reported at time $T=0.5$, while the final position of the bubble is at about 
$z_c^M = 2.7$. 
The discretization parameters for our approximation (\ref{eq:HGa}--d) are shown
in Table~\ref{tab:3dCPUReusken1}. For our finest run
we obtain $V^M_c = 5.33$ and $z_c^M = 2.69$. 
See Figure~\ref{fig:3dbubble_reusken1} for plots of $\Gamma^M$, and 
Figure~\ref{fig:reusken1} for plots of sphericity, centre of mass and rise
velocity over time.
\begin{table}
\center
\begin{tabular}{l|r|r|r|r|r}
\hline
& $K^0_\Gamma$ & $J^0_\Omega$ & NDOF$_{\rm bulk}$ & $M$ & CPU \XFEMGAMMA\\
\hline 
2\,adapt$_{5,2}^{(1)}$  &  770 &  4608 &  21056 & 1000 & 7860 \\
2\,adapt$_{6,3}^{(1)}$  & 1538 & 19536 &  87530 & 1000 & 30830 \\
2\,adapt$_{7,3}^{(1)}$  & 3074 & 47712 & 206652 & 1000 & 113300 \\ 
\hline
\end{tabular}
\caption{Simulation statistics and timings for the rising butanol droplet with
$R_0=0.1$.}
\label{tab:3dCPUReusken1}
\end{table}%
\begin{figure}
\center
\hspace*{-2.1cm}
\includegraphics[angle=-90,width=0.5\textwidth]{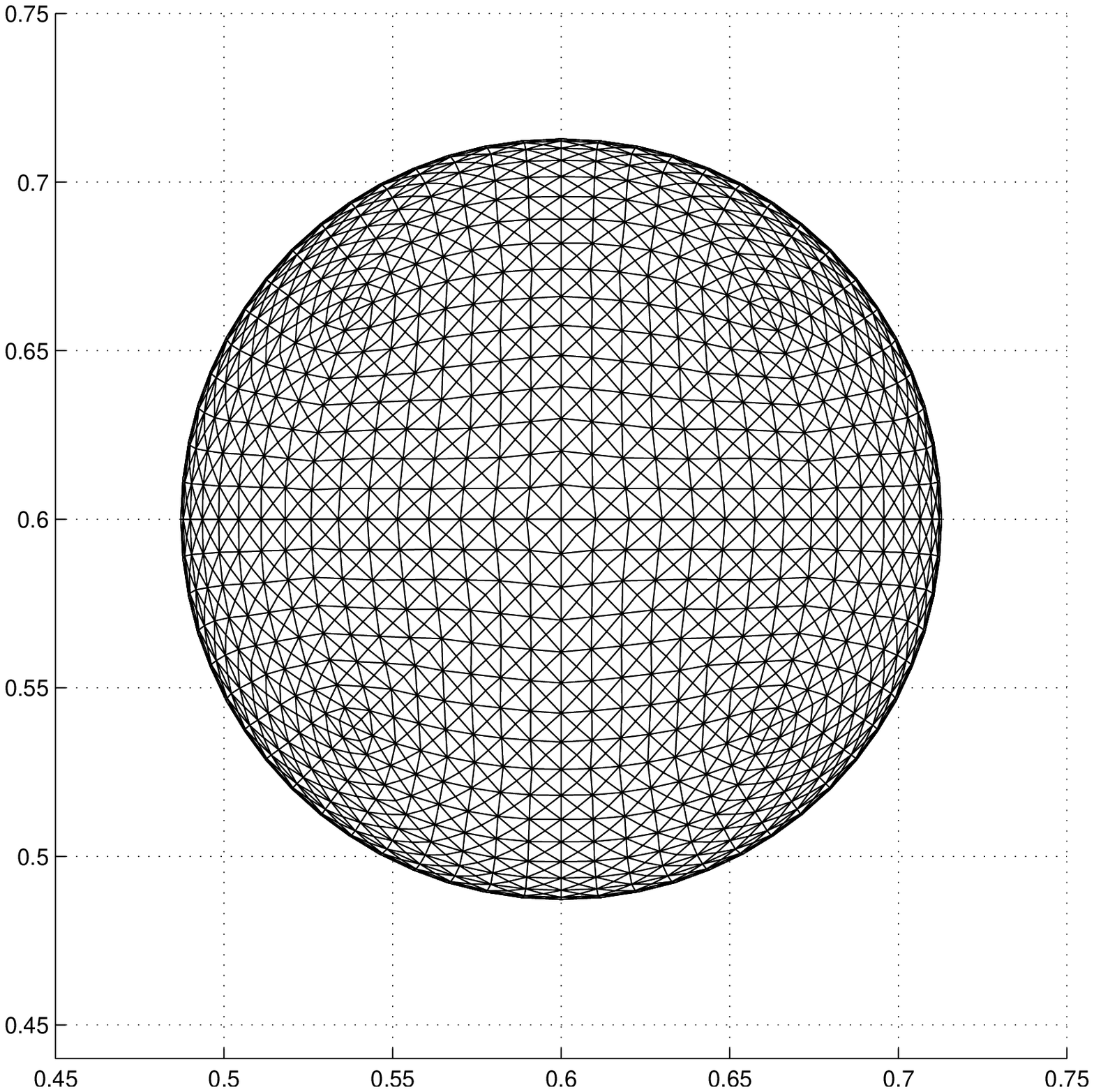}
\includegraphics[angle=-90,width=0.5\textwidth]{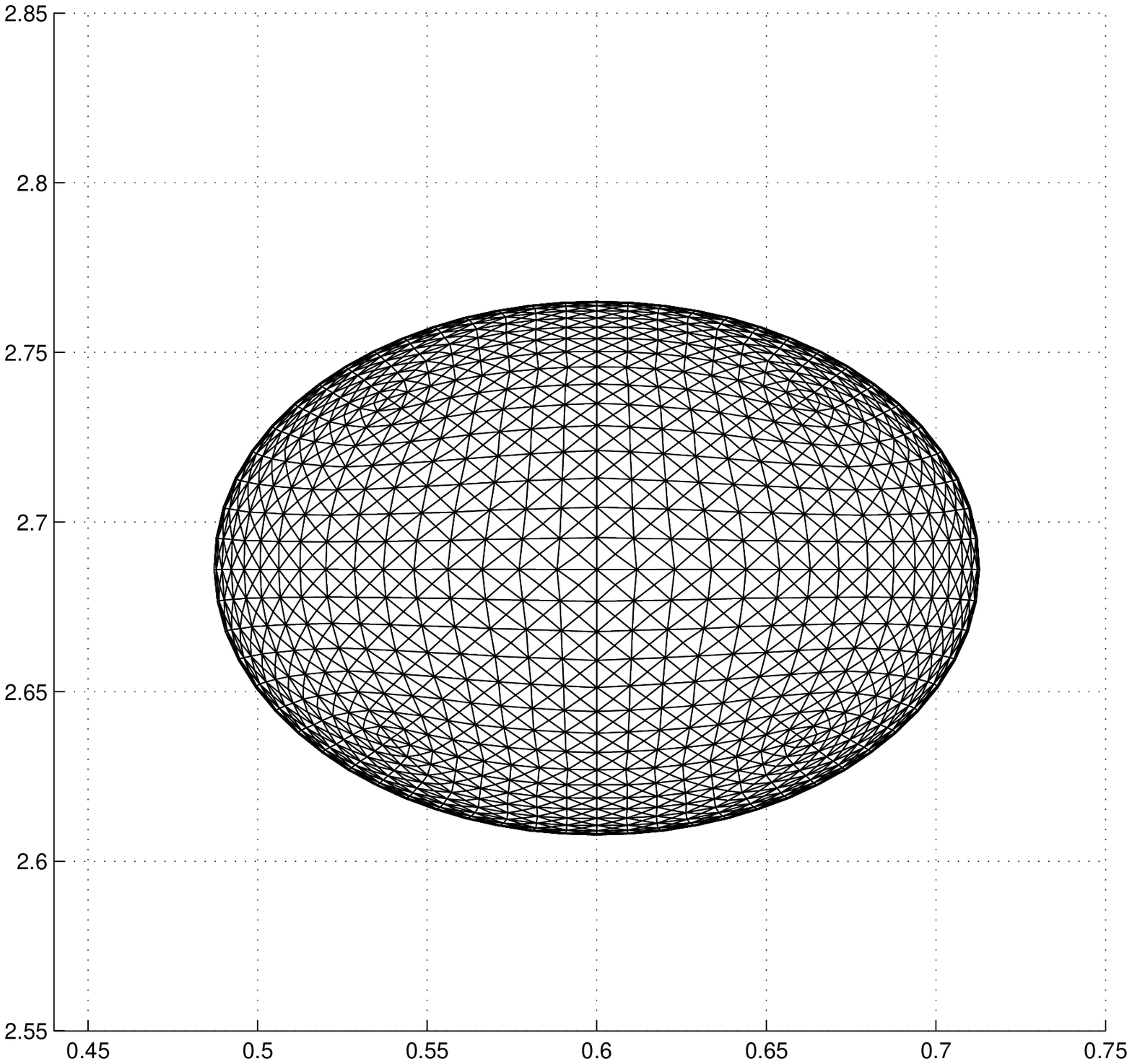}
\caption{(2\,adapt$_{7,3}^{(1)}$, $R_0 = 0.1$)
The final butanol droplet at time $T=0.5$.
Views from the top (left) and from the front (right).}
\label{fig:3dbubble_reusken1}
\end{figure}%
\begin{figure}
\center
\includegraphics[angle=-90,width=0.4\textwidth]{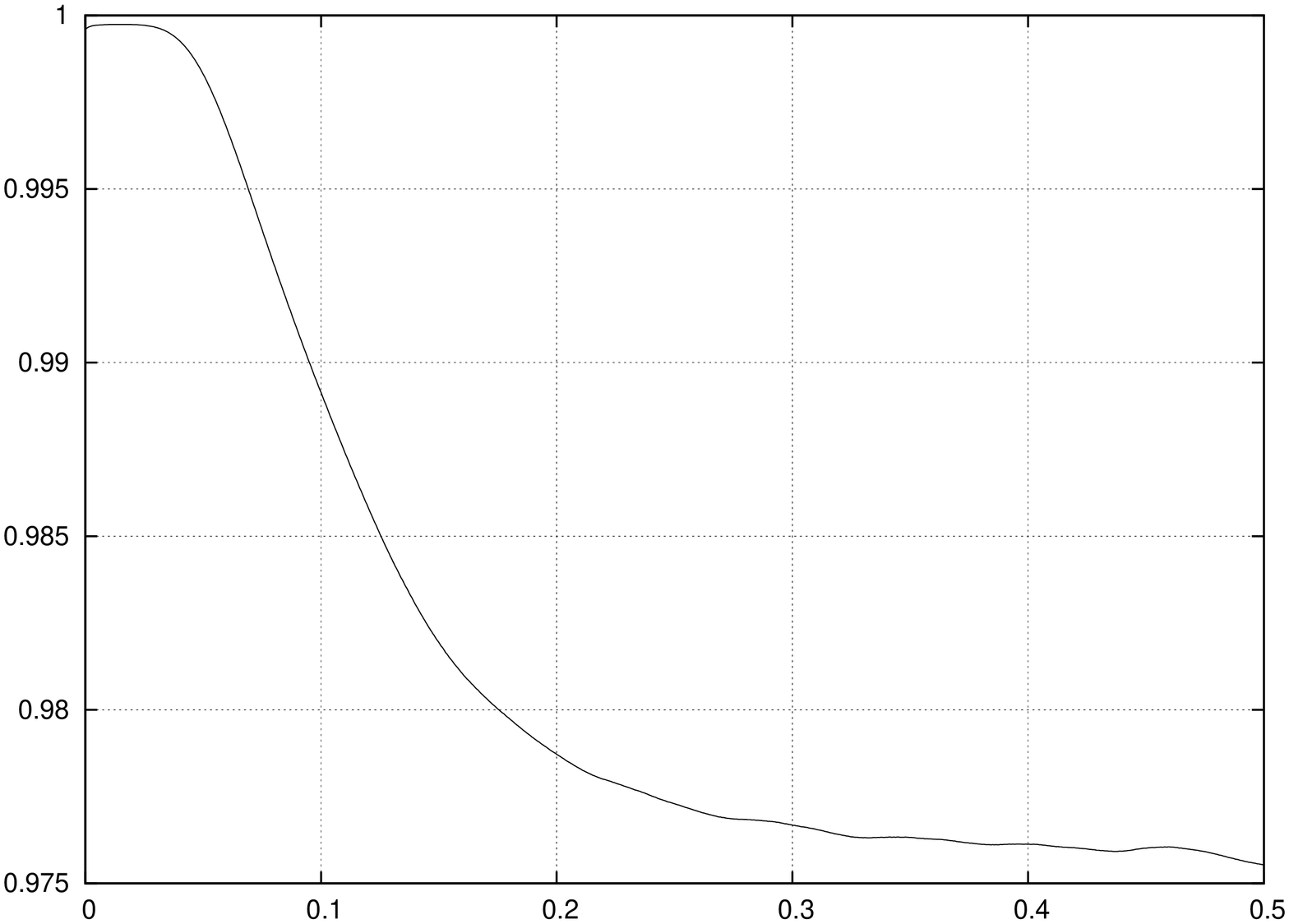}\\
\includegraphics[angle=-90,width=0.4\textwidth]{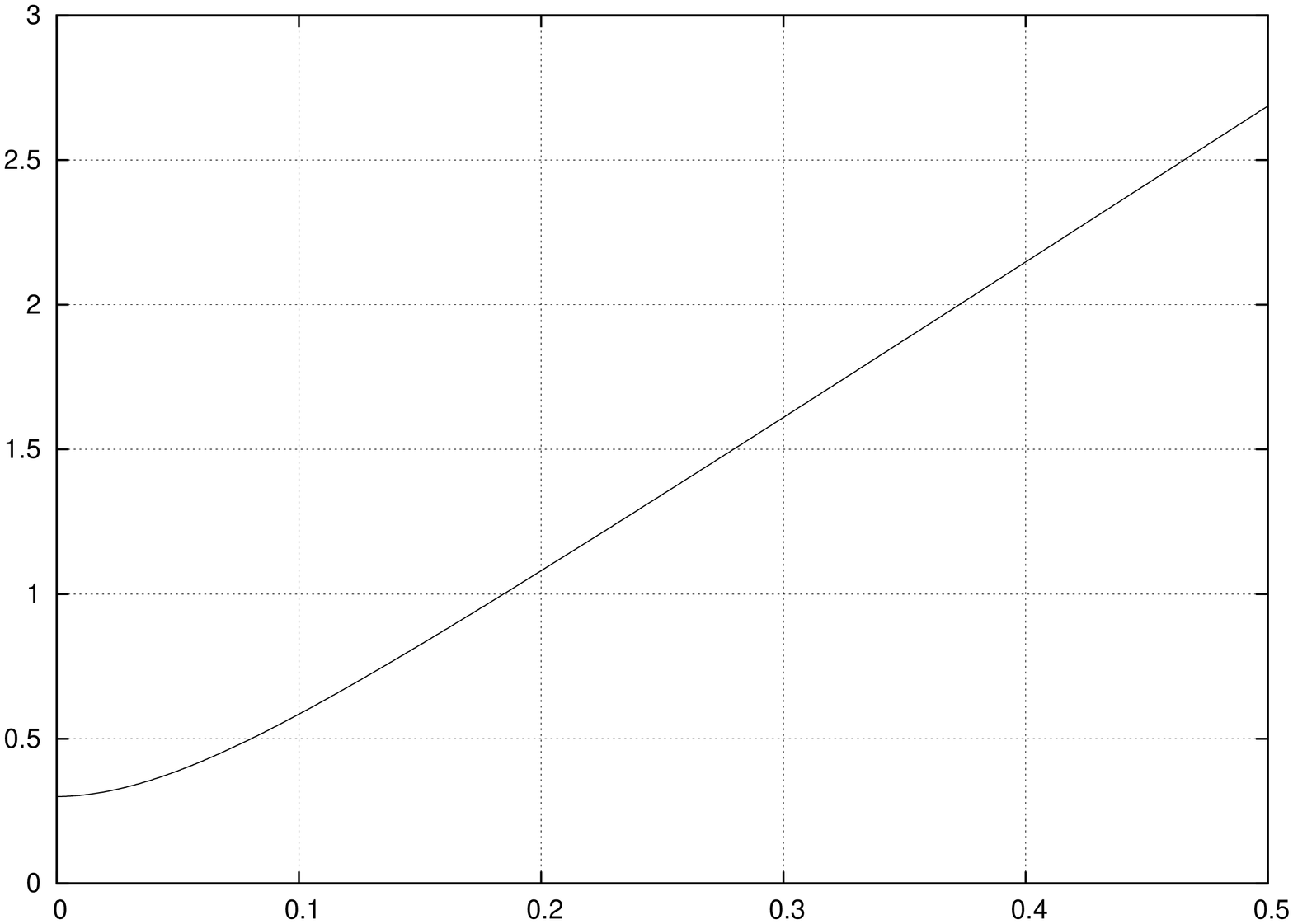}\qquad
\includegraphics[angle=-90,width=0.4\textwidth]{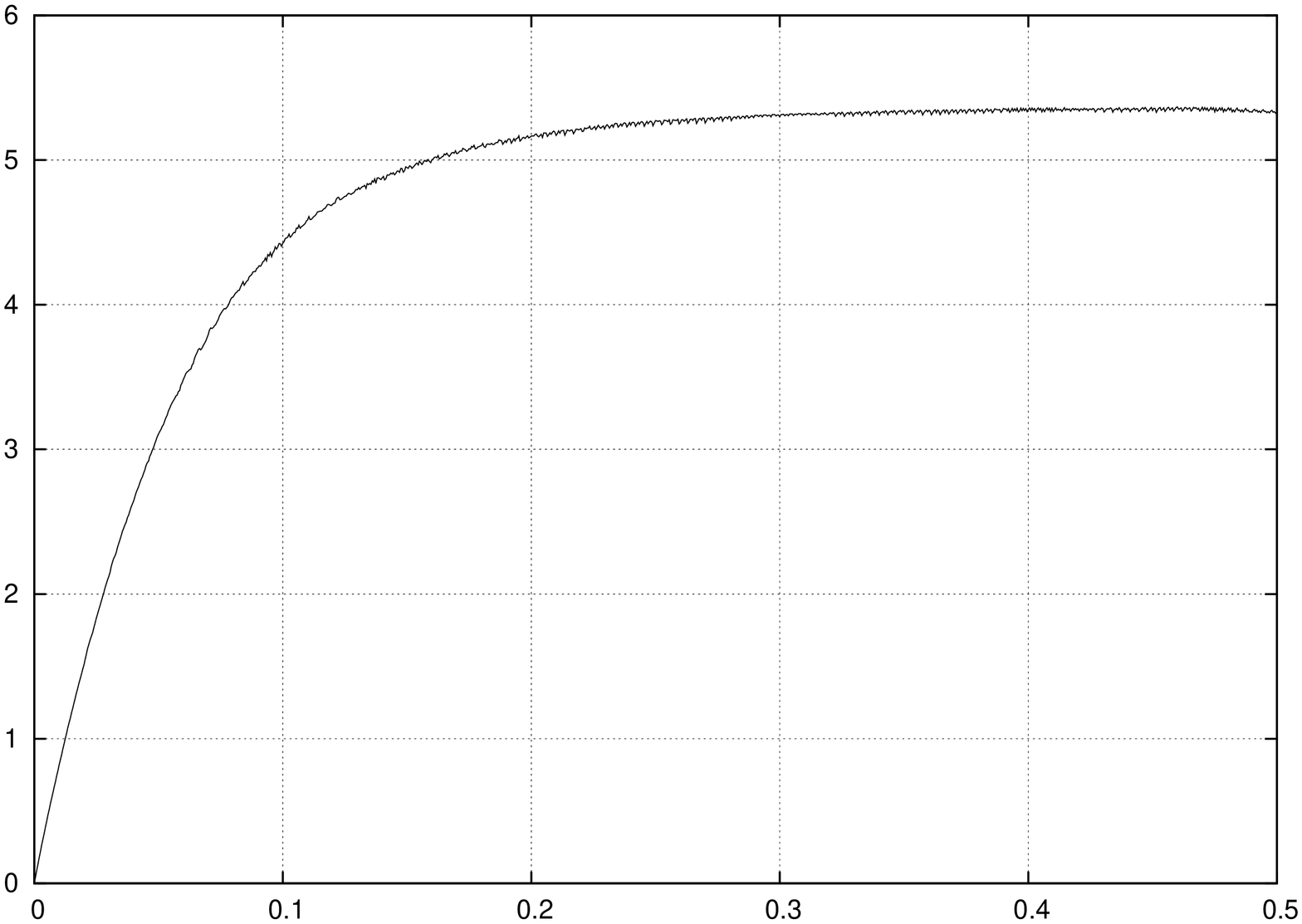}
\caption{(2\,adapt$_{7,3}^{(1)}$, $R_0 = 0.1$)
Sphericity (top), as well as 
centre of mass and rise velocity (bottom) for the rising butanol droplet.}
\label{fig:reusken1}
\end{figure}%

We recall from \cite[Fig. 1.13]{GrossR11} that the shape of the rising butanol
droplet changes dramatically if the initial droplet is chosen larger. To
illustrate this, we repeat the previous simulation but now choose the radius of
the initial sphere to be $R_0=0.2$, 
i.e.\ the droplet is twice as large
as before. As the larger droplet is rising faster, we stop the simulation at
time $T=0.4$.
The discretization parameters for our approximation (\ref{eq:HGa}--d) are shown
in Table~\ref{tab:3dCPUReusken12}, and we visualize the simulation with the
finest parameters in Figures~\ref{fig:3dbubble_reusken12} and
\ref{fig:reusken12}.
\begin{table}
\center
\begin{tabular}{l|r|r|r|r|r}
\hline
& $K^0_\Gamma$ & $J^0_\Omega$ & NDOF$_{\rm bulk}$ & $M$ & CPU \XFEMGAMMA\\
\hline 
2\,adapt$_{5,2}^{(1)}$  &  770 &  4608 &  21056 & 800 &  11820 \\
2\,adapt$_{6,3}^{(1)}$  & 1538 & 19536 &  87530 & 800 &  68880 \\
2\,adapt$_{7,3}^{(1)}$  & 3074 & 47712 & 206652 & 800 & 450800 \\ 
\hline
\end{tabular}
\caption{Simulation statistics and timings for the rising butanol droplet with
$R_0=0.2$.}
\label{tab:3dCPUReusken12}
\end{table}%
\begin{figure}
\center
\hspace*{-2.1cm}
\includegraphics[angle=-90,width=0.5\textwidth]{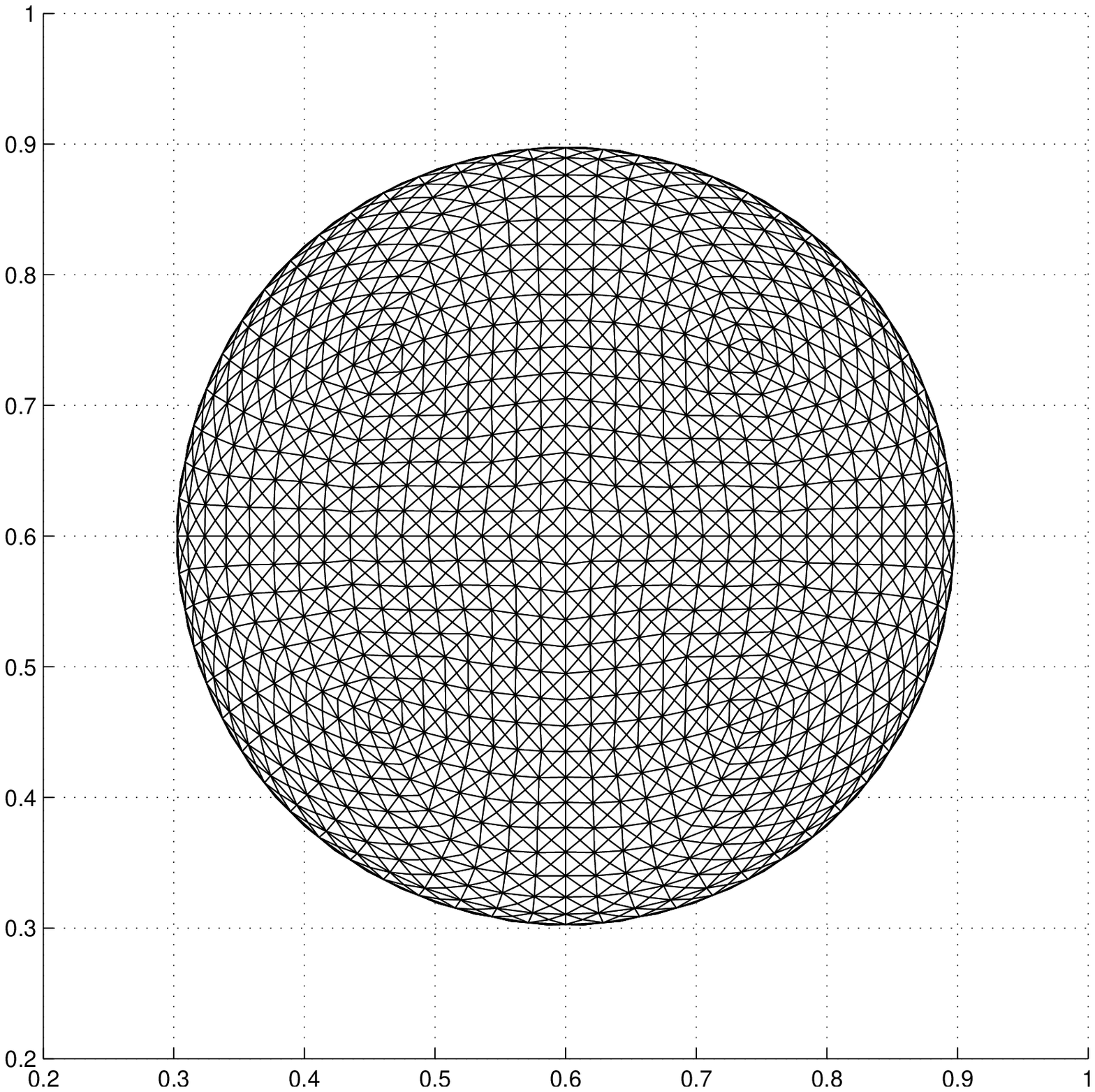}
\includegraphics[angle=-90,width=0.5\textwidth]{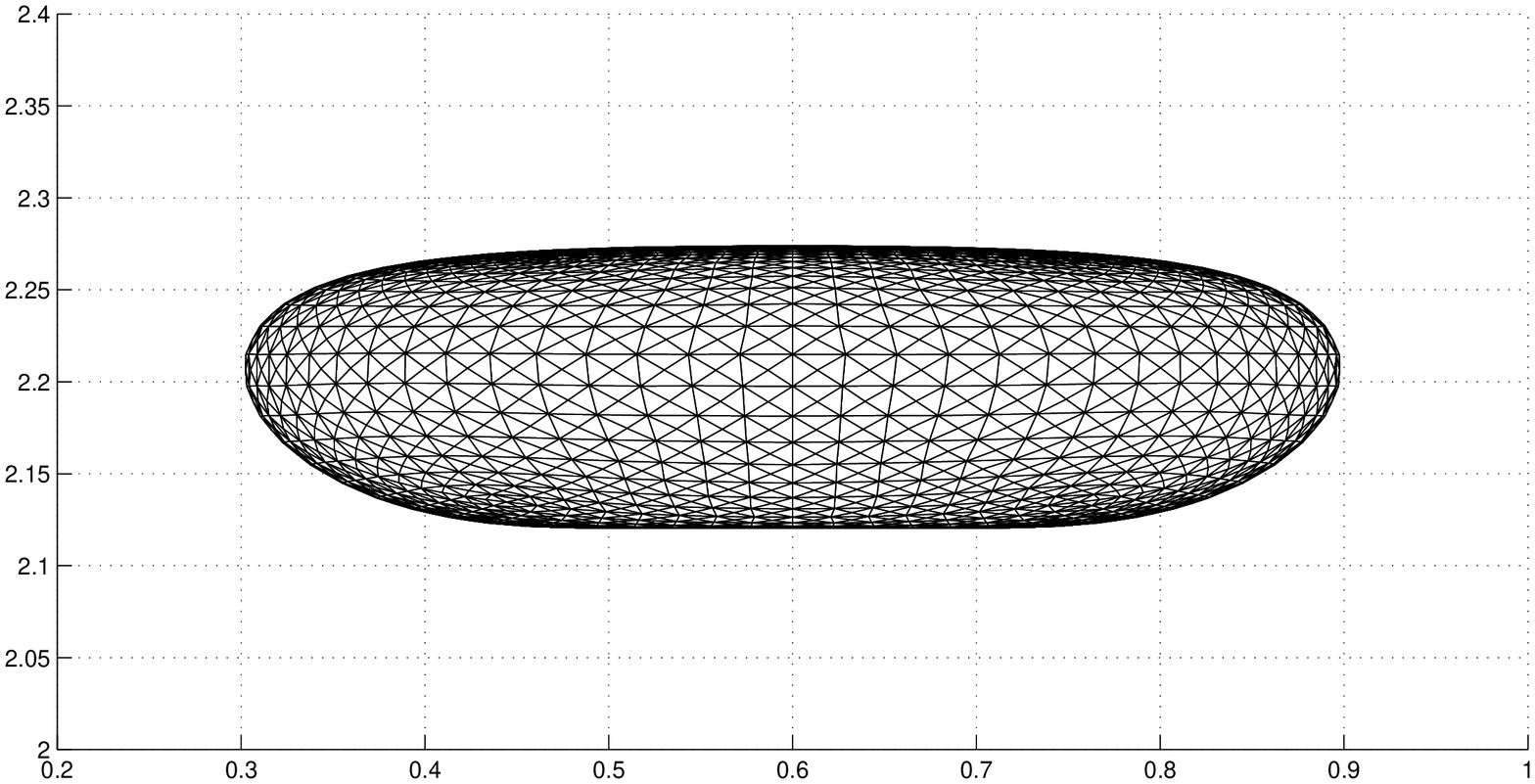}
\caption{(2\,adapt$_{6,3}^{(1)}$, $R_0 = 0.2$)
The final butanol droplet at time $T=0.4$.
Views from the top (left) and from the front (right).}
\label{fig:3dbubble_reusken12}
\end{figure}%
\begin{figure}
\center
\includegraphics[angle=-90,width=0.4\textwidth]{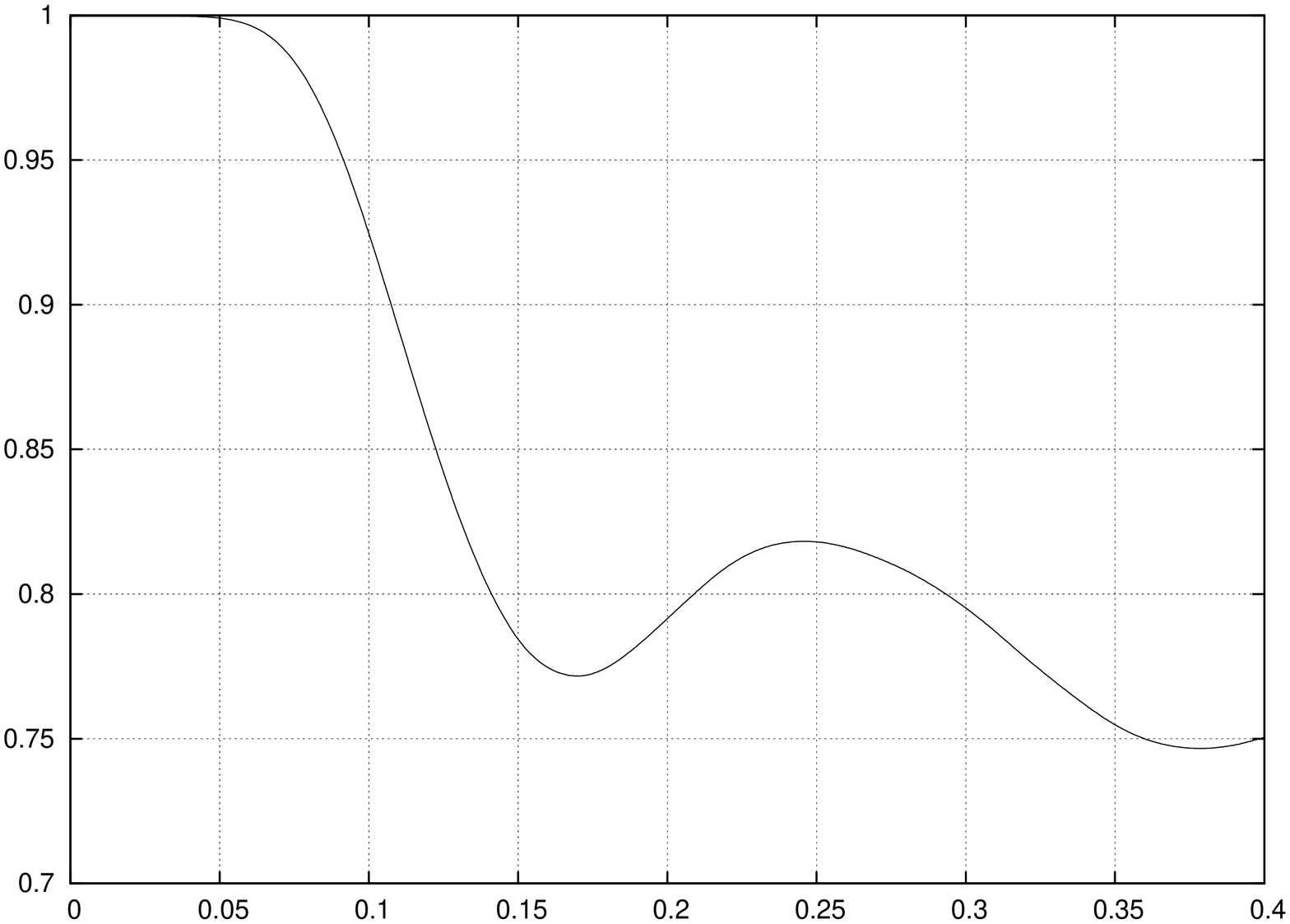}\\
\includegraphics[angle=-90,width=0.4\textwidth]{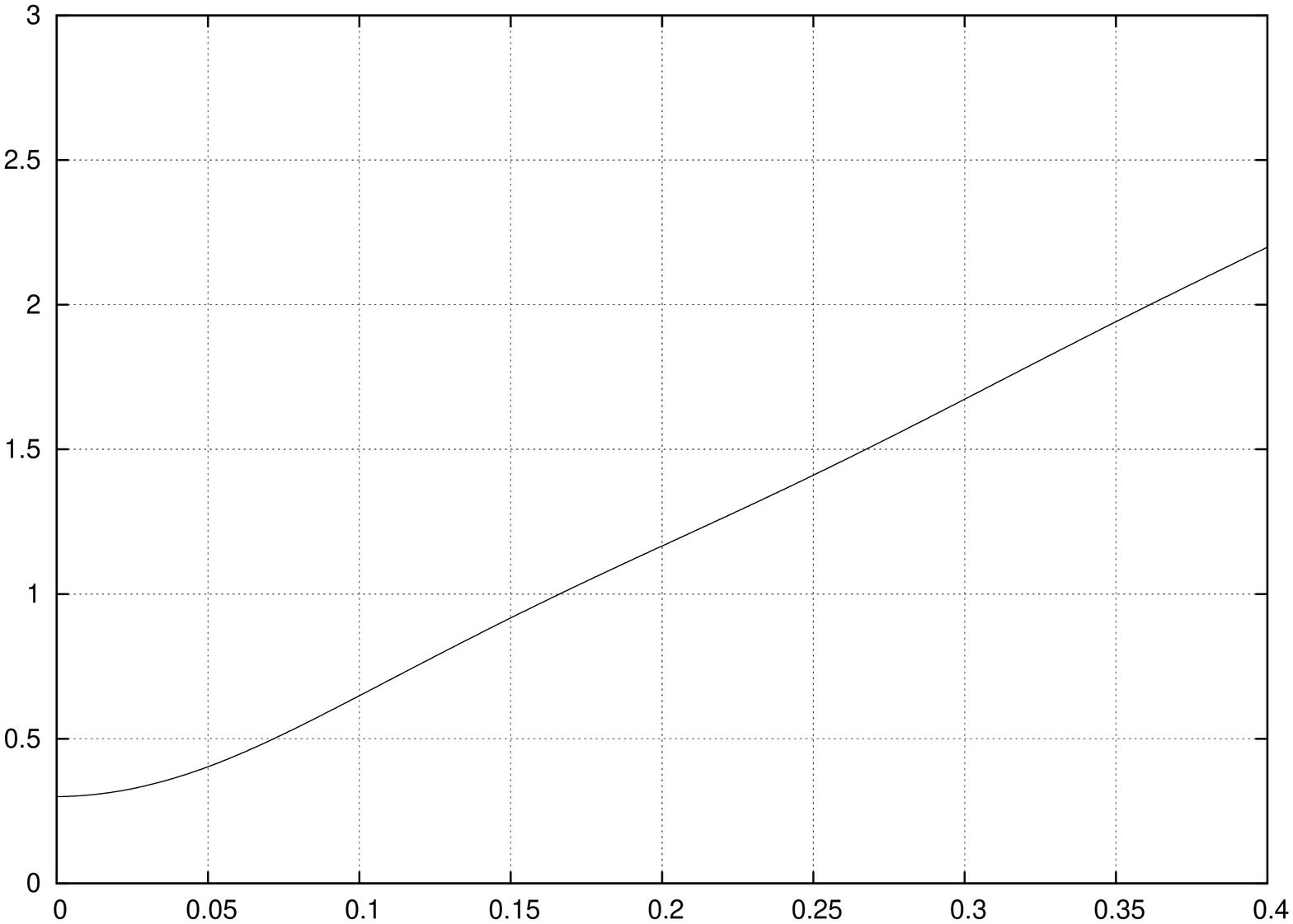}\qquad
\includegraphics[angle=-90,width=0.4\textwidth]{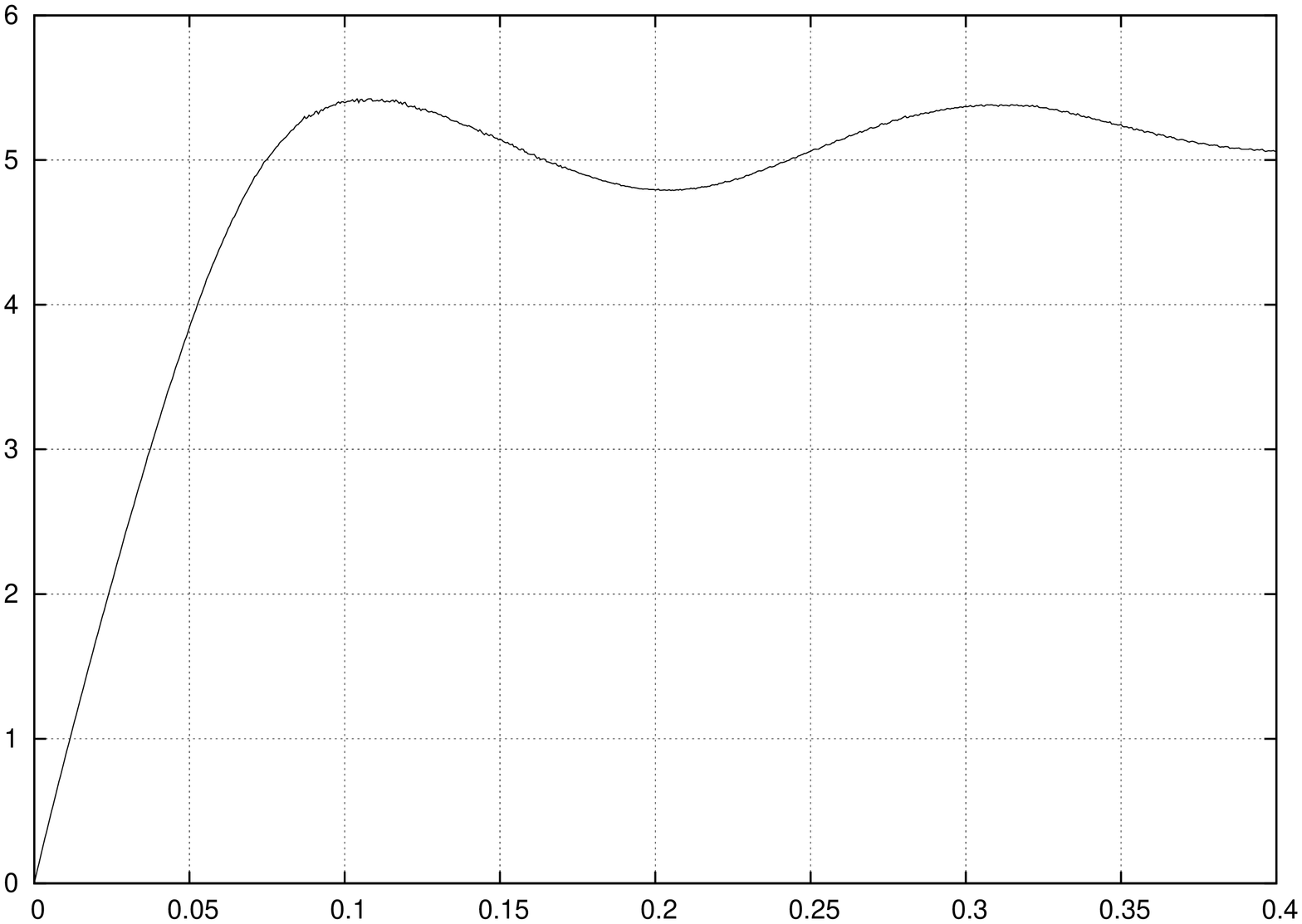}
\caption{(2\,adapt$_{6,3}^{(1)}$, $R_0 = 0.2$)
Sphericity (top), as well as 
centre of mass and rise velocity (bottom) for the rising butanol droplet.}
\label{fig:reusken12}
\end{figure}%

\subsubsection{Rising toluene droplet in 3d} \label{sec:624}
Here we consider a problem that is inspired from \S7.11.2 in \cite{GrossR11}. 
In particular, 
we use all the parameters from the first simulation in \S\ref{sec:623}, 
with the exceptions of
\begin{align*} 
\rho_+ & = 9.988 \times 10^{-4},\quad \rho_- = 8.675 \times 10^{-4}\,,\quad 
\mu_+ = 1.029 \times 10^{-5}\,,\quad \mu_- = 5.96\times 10^{-6}\,,\quad \\
\gamma & = 3.431\times 10^{-2}\,.
\end{align*}
These parameters
model the evolution of a rising toluene droplet in a tank filled with water.
Here the properties of the outer phase (water) slightly differ from the ones in
(\ref{eq:Reusken1}), which models the fact that some saturation with
toluene has taken place to avoid any mass transfer between the two phases.
The main difference to the rising butanol droplet in \S\ref{sec:623} is the
higher surface tension $\gamma$, which the authors in
\cite{GrossR11} state makes this simulation computationally much more
challenging.
In \cite[\S7.11.2]{GrossR11} ten time steps with $\tau = 5\times10^{-4}$ are
performed for this experiment. We continue this simulation until 
time $T=0.4$, 
and observe a terminal rise velocity of $V^M_c = 7.5$, as well
as $z_c^M = 2.78$ for our finest simulation; 
see Table~\ref{tab:3dCPUReusken2} for the precise discretization parameters.
The evolution of the sphericity, the centre of mass as well as the rise
velocity can be seen in Figure~\ref{fig:reusken2}, 
where we note that the droplet
stays almost perfectly spherical throughout the evolution.
\begin{table}
\center
\begin{tabular}{l|r|r|r|r|r}
\hline
& $K^0_\Gamma$ & $J^0_\Omega$ & NDOF$_{\rm bulk}$ & $M$ & CPU \XFEMGAMMA\\
\hline 
2\,adapt$_{5,2}^{(1)}$ &  770 &  4608 &  21056 & 800 & 14760 \\
2\,adapt$_{6,3}^{(1)}$ & 1538 & 19536 &  87530 & 800 & 68650 \\
2\,adapt$_{7,3}^{(1)}$ & 3074 & 47712 & 206652 & 800 & 340600 \\
\hline
\end{tabular}
\caption{Simulation statistics and timings for rising toluene droplet with
$R_0=0.1$.}
\label{tab:3dCPUReusken2}
\end{table}%
\begin{figure}
\center
\includegraphics[angle=-90,width=0.4\textwidth]{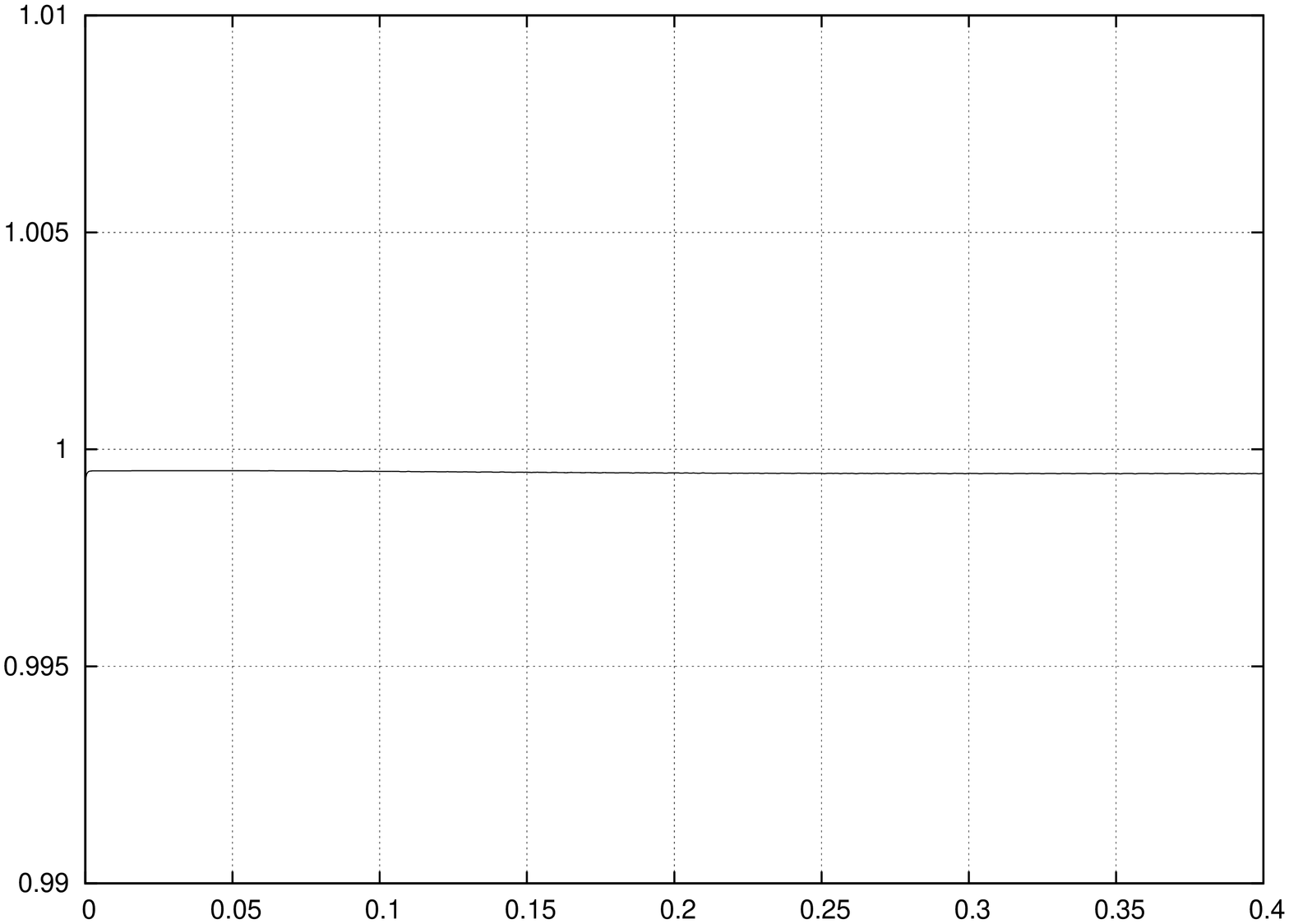}\\
\includegraphics[angle=-90,width=0.4\textwidth]{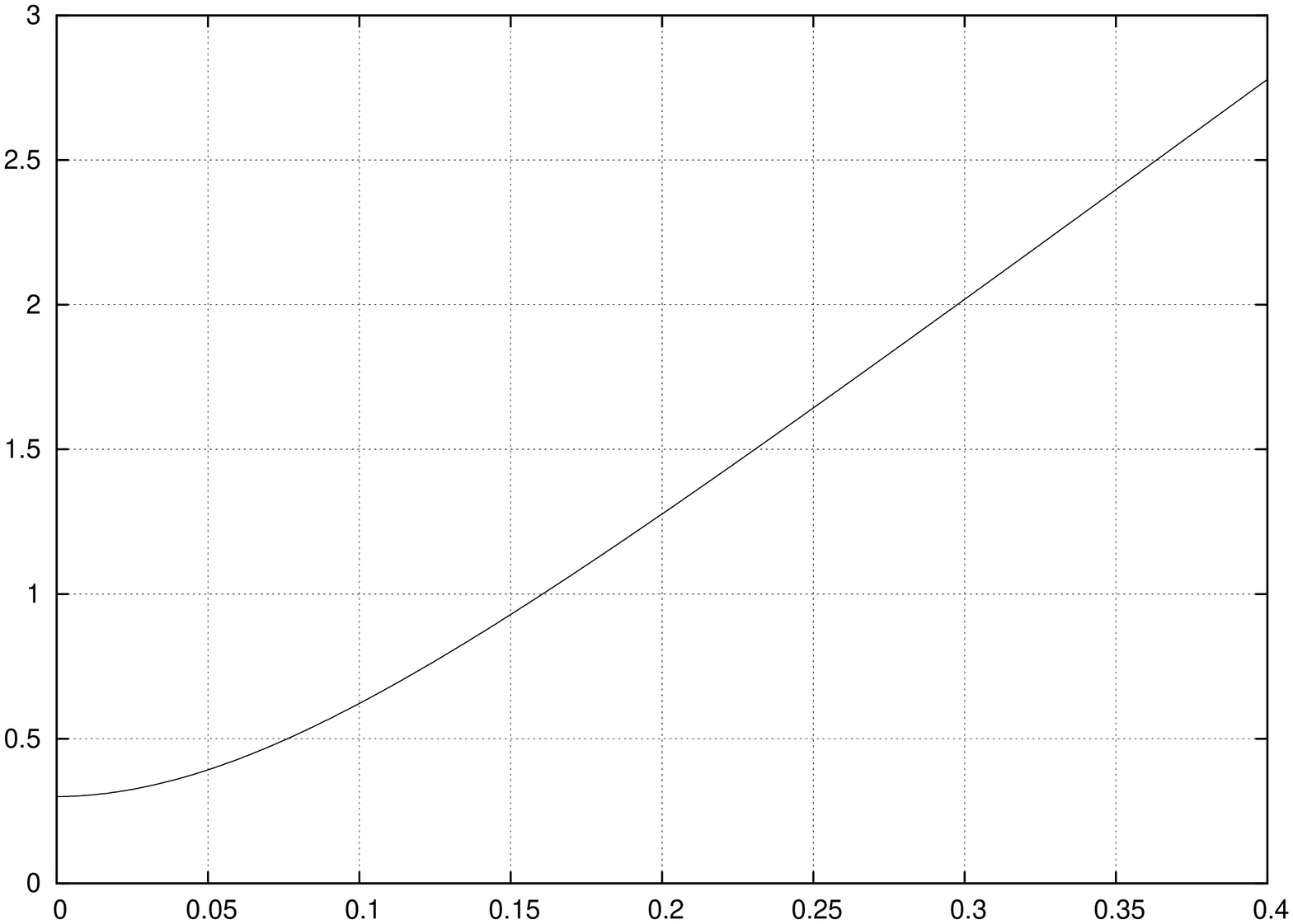}\qquad
\includegraphics[angle=-90,width=0.4\textwidth]{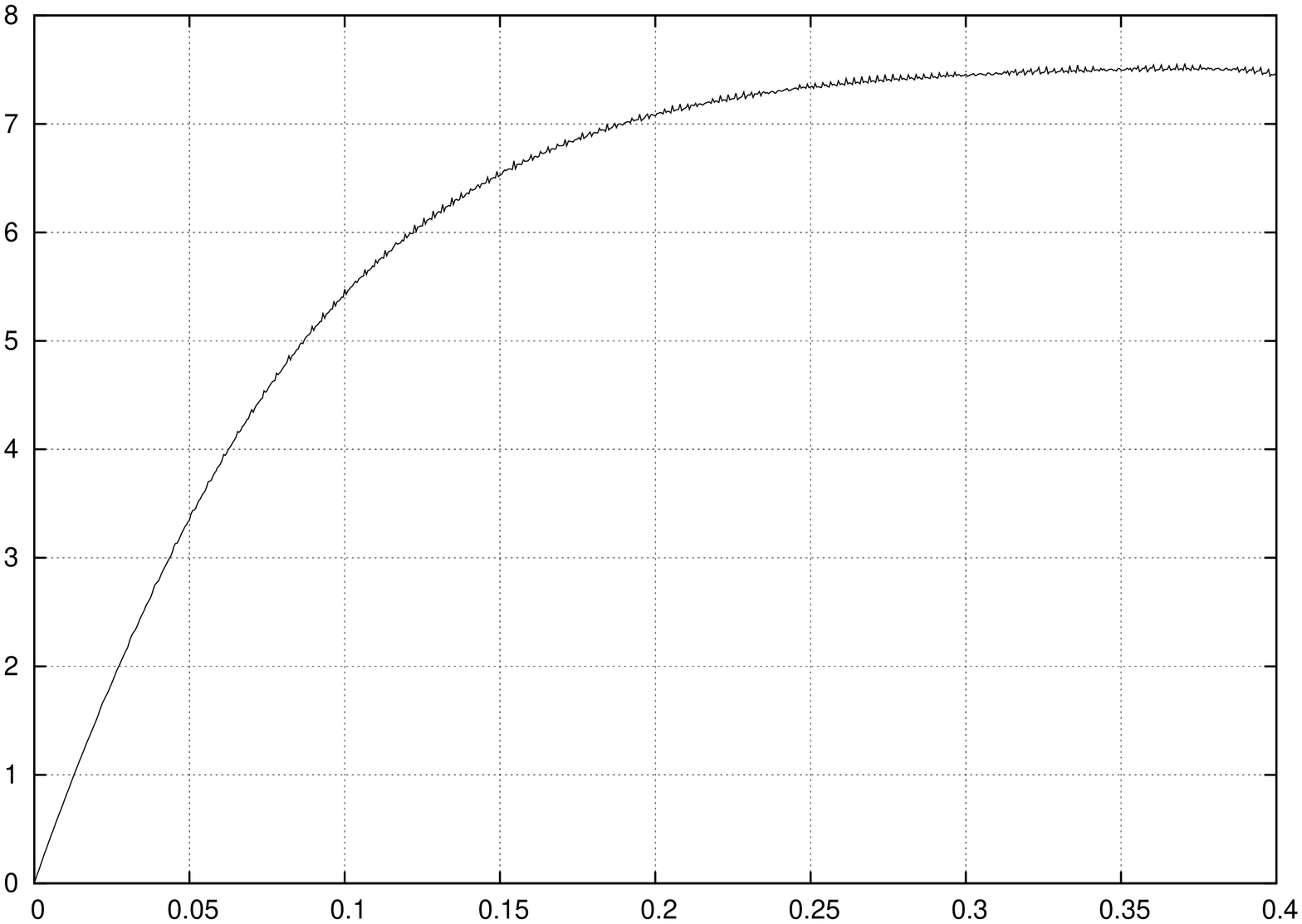}
\caption{(2\,adapt$_{7,3}^{(1)}$, $R_0 = 0.1$)
Sphericity (top), as well as 
centre of mass and rise velocity (bottom) for the rising toluene droplet.}
\label{fig:reusken2}
\end{figure}%

\def\soft#1{\leavevmode\setbox0=\hbox{h}\dimen7=\ht0\advance \dimen7
  by-1ex\relax\if t#1\relax\rlap{\raise.6\dimen7
  \hbox{\kern.3ex\char'47}}#1\relax\else\if T#1\relax
  \rlap{\raise.5\dimen7\hbox{\kern1.3ex\char'47}}#1\relax \else\if
  d#1\relax\rlap{\raise.5\dimen7\hbox{\kern.9ex \char'47}}#1\relax\else\if
  D#1\relax\rlap{\raise.5\dimen7 \hbox{\kern1.4ex\char'47}}#1\relax\else\if
  l#1\relax \rlap{\raise.5\dimen7\hbox{\kern.4ex\char'47}}#1\relax \else\if
  L#1\relax\rlap{\raise.5\dimen7\hbox{\kern.7ex
  \char'47}}#1\relax\else\message{accent \string\soft \space #1 not
  defined!}#1\relax\fi\fi\fi\fi\fi\fi}

\end{document}